\documentclass{amsart}

\usepackage[left=1.2in, right=1.2in, top=1in, bottom=1in]{geometry}
\usepackage{stmaryrd}

\usepackage{enumerate}
\usepackage{enumitem}
\usepackage{amsfonts}
\usepackage{stmaryrd}
\usepackage{amssymb}
\usepackage{xy}
\usepackage{sseq}
\input xy
\xyoption{all}
\usepackage{amsmath}
\usepackage{url}
\usepackage{verbatim}
\usepackage{hyperref}
\usepackage{xy}
\input xy
\xyoption{all}
\newcommand{\spec}{\mathrm{Spec}\,}
\newcommand{\holim}{\mathrm{holim}}

\newcommand{\einf}{\mathrm{CAlg}}

\renewcommand{\mod}{\mathrm{Mod}}

\newcommand{\rf}{\mathrm{Refine}}
\usepackage{amsthm}
\usepackage{cleveref}

\usepackage{tikz} 

\newcommand{\thref}[1]{\Cref{#1}}
\renewcommand{\sp}{\mathrm{Sp}}

\newcommand{\otop}{\mathcal{O}^{\mathrm{top}}}

\newcommand{\mell}{M_{{ell}}}
\newcommand{\mellc}{M_{\overline{ell}}}

\renewcommand{\hom}{\mathrm{Hom}}
\newtheorem{lemma}{Lemma}[section]

\newtheorem{corollary}[lemma]{Corollary}

\newtheorem{theorem}[lemma]{Theorem}
\newtheorem*{thm}{Theorem}
\newtheorem*{mthm}{Main Theorem}
\newtheorem{proposition}[lemma]{Proposition}

\newcommand{\XX}{\mathfrak{X}}
\newcommand{\YY}{\mathfrak{Y}}
\newcommand{\CC}{\mathcal{C}}

\newcommand{\Z}{\mathbb{Z}}
\newcommand{\Q}{\mathbb{Q}}
\newcommand{\gx}{\Gamma(\mathfrak{X}, \otop_\XX)}
\newcommand{\FF}{\mathcal{F}}
\newcommand{\LL}{\mathcal{L}}
\newcommand{\II}{\mathcal{I}}
\newcommand{\tensor}{\otimes}
\newcommand{\hocolim}{\mathrm{hocolim}}

\DeclareMathOperator{\colim}{colim}
\DeclareMathOperator{\Spec}{Spec\,}

\DeclareMathOperator{\id}{id}
\DeclareMathOperator{\Map}{Map}
\DeclareMathOperator{\Aut}{Aut}

\DeclareMathOperator{\Tot}{Tot}
\DeclareMathOperator{\pr}{pr}

\Crefname{lemma}{Lemma}{Lemma}
\Crefname{theorem}{Theorem}{Theorem}
\Crefname{proposition}{Proposition}{Proposition}
\Crefname{example}{Example}{Example}

\begin{document}

\setlength{\parskip}{0.5mm}
\renewcommand{\rightrightarrows}{\begin{smallmatrix} \to \\
\to \end{smallmatrix} }
\newcommand{\triplearrows}{\begin{smallmatrix} \to \\ \to \\ 
\to \end{smallmatrix} }

\newcommand{\sh}{\mathbf{Sh}}
 
\renewcommand{\ltimes}{\stackrel{\mathbb{L}}{\otimes}}
\newcommand{\psh}{\mathbf{PSh}}
\theoremstyle{definition}
\newtheorem{definition}[lemma]{Definition}

\newtheorem{cons}[lemma]{Construction}
\newcommand{\cd}{\mathrm{cd}}
\newcommand{\OO}{\widetilde{\mathcal{O}}}
\newcommand{\A}{\mathbb{A}}
\newcommand{\F}{\mathcal{F}}
\renewcommand{\P}{\mathbb{P}}
\newcommand{\bl}{\bullet}
\newcommand{\Tmf}{\mathrm{Tmf}}
\newcommand{\TMF}{\mathrm{TMF}}
\renewcommand{\ell}{\mathrm{Ell}}
\newcommand{\tmf}{\mathrm{tmf}}
 
 \theoremstyle{definition}
\newtheorem{remark}[lemma]{Remark}
\newcommand{\rng}{\mathrm{Ring}}
\newcommand{\gpd}{\mathbf{Gpd}}
\newcommand{\ei}{\mathbb{E}_1}
\renewcommand{\A}{\mathcal{A}_*}

\newcommand{\qcoha}{\qcoh^{\mathrm{ab}}}
\newtheorem{example}[lemma]{Example}
\newtheorem*{exm}{Example}

\title{Affineness and chromatic homotopy theory}
\date{\today}
\author{Akhil Mathew and Lennart Meier}
\email{amathew@math.harvard.edu; lmeier@math.uni-bonn.de} 
\address{Department of Mathematics, Harvard University, Cambridge, MA 02138; 
Mathematisches Institut, Universit{\"a}t Bonn, 
53115 Bonn, Germany}

\maketitle

\newcommand{\qcoh}{\mathrm{QCoh}}
\newcommand{\md}{\mathrm{Mod}}

\begin{abstract}
Given an algebraic stack $X$, one may compare the derived category of
quasi-coherent sheaves on $X$ with the category of dg-modules over the dg-ring of
functions on $X$. 
We study the analogous question in stable homotopy theory, for derived stacks
that arise via realizations of diagrams
of Landweber-exact homology theories. We identify a condition (quasi-affineness of
the map to the moduli stack of formal groups) under which the two categories
are equivalent, and 
study applications to topological modular forms. In particular, we provide new
examples of Galois extensions of ring spectra and vanishing results for Tate spectra. 
\end{abstract}

\tableofcontents
\setlength{\parskip}{0.8mm}

\section{Introduction}

Let $X$ be a scheme (or, more generally, an algebraic stack). Then one has a
natural abelian category $\qcoha(X)$ of
\emph{quasi-coherent sheaves} on $X$, which comes with a left exact functor 
\[ \Gamma\colon \qcoha(X) \to \md^{\mathrm{ab}}( R),  \]
into the abelian category $\md^{\mathrm{ab}}(R)$ of $R$-modules,
where $R = \Gamma(X, \mathcal{O}_X)$ is the ring of regular functions on $X$. 
When $X$ is affine, so that $X = \spec R$, the functor $\Gamma$ is an equivalence of
categories (and is in particular exact). 
Conversely, a classical result of Serre implies that if $X$ is a quasi-compact
scheme, then the converse holds: if $\Gamma$ is an
equivalence, then $X \simeq \spec \Gamma(X, \mathcal{O}_X)$, so that $X$ can be
recovered as the spectrum of the ring of global sections on $X$, which in turn
is determined by the category of $\Gamma(X, \mathcal{O}_X)$-modules.

Namely, one has: 

\begin{thm}[Serre]
Let $X$ be a quasi-compact scheme. Suppose that 
the higher cohomologies $\{H^i(X, \mathcal{F}) \}_{i \geq 1}$ vanish, for every
quasi-coherent sheaf $\mathcal{F} \in \qcoha(X)$. Then $X$ is affine. 
\end{thm} 

For a modern reference (and  the strongest statement), we refer to \cite[Tag
01XF]{stacks-project}.\footnote{These references can be looked up on \url{http://stacks.math.columbia.edu/tag}.
} 
One can consider the equivalent question in the derived setting. 
Given a stack ${X}$, one has a natural \emph{derived} category of
quasi-coherent sheaves on ${X}$, which we will denote by
$\qcoh({X})$. 
Rather than being an abelian category, it is a triangulated category or,
better, the underlying homotopy category of a stable $\infty$-category in the
sense of \cite{higheralg}. 
One has a similar (derived) global sections functor
\[ \Gamma\colon \qcoh(X) \to \md( R),  \]
where $R = \Gamma(X, \mathcal{O}_X) = R\Gamma(X, \mathcal{O}_X)$ is now no longer a commutative ring, but
itself a derived ring: it is a coconnective $E_\infty$-ring spectrum,
 obtained as the homotopy limit of
the discrete rings that map to $X$. In characteristic zero, $R$ is a
commutative, differential graded algebra such that $H^i(R) = 0$ for $i < 0$. 
In particular, $\md(R)$ itself is a stable $\infty$-category: if $R$ is
discrete, it is the derived category of the abelian category
$\md^{\mathrm{ab}}(R)$ of ordinary (i.e., discrete) 
$R$-modules. 

Certain phenomena work better in the derived context. For example, $\Gamma$ is
always ``exact'' in the stable sense, which means that it respects finite
homotopy limits and homotopy colimits. (Indeed, it respects arbitrary homotopy limits.) Unlike in the ordinary setting, it is possible for $\Gamma$ to be an
equivalence even if $X$ is not affine, although in these cases $R$ will usually be non discrete. 

\begin{exm} 
\label{BGa}
In this example, we work over the rational numbers. 
Let ${X} = B \mathbb{G}_a$ be the classifying stack of the additive
group. 
Then $\Gamma( {X}, \mathcal{O}_{{X}}) = \mathbb{Q}[x_{-1}]$
is the free $E_\infty$-algebra (over $\mathbb{Q}$) on a generator in degree
$-1$, i.e., the cochains on the circle $S^1$. 
Then it is known that taking global sections establishes an equivalence between
the derived $\infty$-category $\qcoh( {X})$ and the $\infty$-category $\md(
\mathbb{Q}[x_{-1}])$ of modules (i.e., module spectra) over $\mathbb{Q}[x_{-1}]$. 
Generalizations of this phenomenon have been explored in \cite{toen,
DAGQC}.\footnote{This result
can be extracted from \cite{DAGQC} as follows. The stack $X = B\mathbb{G}_a$
sends a rational connective $E_\infty$-ring $R$ to $\Omega^\infty ( \Sigma R)$. In the notation
of \cite[\S 4]{DAGQC}, one has $ X = \mathrm{c} \spec A$ for $A  = \mathbb{Q}[x_{-1}]$.  Now, as in \cite[\S 4.5]{DAGQC} one has a $t$-structure on $\md(A)$
whose connective objects are those $A$-modules $M$ such that $M \otimes_A
\mathbb{Q}$ (for the map $A \to \mathbb{Q}$, unique up to homotopy) is
 connective. One also has a $t$-structure on $\qcoh(X)$. The left
 adjoint $\md(A) \to \qcoh(X)$ 
 exhibits $\qcoh(X)$ as
the left completion of $\md(A)$ by \cite[Remark 4.5.6]{DAGQC}. 
But we claim that $\md(A)$ is already left complete, i.e., for any $M \in
\md(A)$, the natural map $M \to \varprojlim_{n} \tau_{\leq n} M$ is an
equivalence. In fact, if $N \in \md(A)_{\geq n}$, then  $N \otimes_{A}
\mathbb{Q} $ is an $(n-1)$-connective spectrum by definition. 
However, in view of the cofiber sequence of $A$-modules
\( \Omega \mathbb{Q} \to A \to \mathbb{Q} ,  \)
this implies easily 
that $N$ is a $(n-2)$-connective spectrum. 
It follows that for any $M \in \md(A)$, the cofiber of $M \to \tau_{\leq n} M$
is $n$-connective as a spectrum, so that $M \to \varprojlim_n \tau_{\leq n} M$ is an
equivalence of $A$-modules as desired. Thus, $\qcoh(X) \simeq \md(A)$. This equivalence of $\infty$-categories
also gives us $\Gamma(X, \mathcal{O}_{X}) \simeq A$.}

\end{exm} 

\newcommand{\affm}{\mathrm{Aff}^{\mathrm{et}}_{/{M}_{ell}}}
\newcommand{\mellb}{\mathfrak{M}_{ell}}
\newcommand{\mellbc}{\mathfrak{M}_{\overline{ell}}}

The purpose of this paper is to study this sort of affineness in a different setting, namely derived stacks in chromatic homotopy theory. 
Our motivational example is the (periodic) spectrum $\TMF$ of \emph{topological modular
forms}. It arises as the global sections of a sheaf of $E_\infty$-ring spectra
$\otop$ on the moduli stack of elliptic curves $\mell$, constructed by Goerss,
Hopkins and Miller, and later by Lurie. 

There are two natural $\infty$-categories one can associate to this construction: 
\begin{enumerate}
\item The $\infty$-category $\qcoh( \mellb)$ of quasi-coherent sheaves on the derived stack $\mellb = (\mell, \otop)$.
\item The $\infty$-category $\md(\TMF)$ of $\TMF$-modules. 
\end{enumerate}

An example of the affineness result we prove is:
\begin{thm} 
The global sections functor establishes an equivalence of symmetric
monoidal $\infty$-categories
$\qcoh(\mellb) \simeq \md(\TMF)$. 
\end{thm} 

This theorem was originally established away from the prime 2 by the second author in \cite{meier},
and is useful for both theoretical and computational purposes. The result was
also known to Lurie in unpublished work (by a different argument). We also prove  
a version for the compactified moduli stack of elliptic curves $\mellc$, which
carries a sheaf of $E_\infty$-ring spectra $\otop$ as well, defining a derived
stack $\mellbc = (\mellc, \otop)$. The global sections $\Gamma(\mellbc, \otop)$
are denoted by $\Tmf$. 
\begin{thm}The global sections functor establishes an equivalence of symmetric
monoidal $\infty$-categories
$\qcoh(\mellbc) \simeq \md(\Tmf)$. 
\end{thm} 

The main purpose of this paper is to prove these theorems in a more general
context, as a consequence of nilpotence technology.

Given any noetherian and separated Deligne-Mumford stack $X$ with a flat morphism $X \to
M_{FG}$ to the moduli stack $M_{FG}$ of formal groups, one
can construct 
a presheaf of even periodic Landweber-exact homology theories on $X$. Sometimes, it can be lifted to $E_\infty$-rings to produce a derived stack $\mathfrak{X}$, as in the case $X = \mell$. If it can be lifted, then one can ask the same
question as above: is the $\infty$-category $\qcoh(\mathfrak{X})$ of quasi-coherent sheaves on
$\mathfrak{X}$ equivalent to the $\infty$-category of modules over $\Gamma(\mathfrak{X},
\otop)$? 
This is certainly true when $X$ is affine. We show that the same conclusion holds in the following setting:
\begin{mthm}
If $X \to M_{FG}$ is quasi-affine, then the global sections functor establishes an
equivalence of symmetric monoidal $\infty$-categories $ \qcoh( \mathfrak{X})
\simeq \md( \Gamma( \mathfrak{X}, \otop))$. 
\end{mthm} 

Recall here that a map $X\to M_{FG}$ is \textit{quasi-affine} if for every map
$\Spec A \to M_{FG}$ the pullback $\Spec A\times_{M_{FG}} X$ is quasi-affine,
i.e., a quasi-compact open subscheme of an affine scheme. The result is proved via a consequence of derived Morita (Schwede-Shipley)
theory together with a version of the Hopkins--Ravenel smash product theorem.
The latter states that the localization functor $L_n$ commutes with homotopy
colimits. Likewise, a crucial part of our main theorem is that the global
sections functor commutes with homotopy colimits. This turns out to be true
even if $X\to M_{FG}$ is not quasi-affine, but only \textit{tame}, i.e., the order of every automorphism of a point of $X$ not detected by the formal group is invertible on $X$. 

We apply our main theorem to the study of Galois extensions of $E_\infty$-rings (in the sense of Rognes
\cite{rognes}) and to vanishing results about Tate spectra. As an example, we
consider the moduli stack of elliptic curves with $\Gamma(n)$-level structure
$\mell(n)$ and its compactified version $\mellc(n)$. These classify
(generalized) elliptic curves with a chosen isomorphism between the $n$-torsion
points and $(\Z/n\Z)^2$. The action of $GL_2(\Z/n\Z)$ on $(\Z/n\Z)^2$ defines
$GL_2(\Z/n\Z)$-actions on $\mell(n)$ and $\mellc(n)$. Both of these stacks
carry sheaves of $E_\infty$-ring spectra $\otop$, whose global sections are
denoted by $\TMF(n)$ and $\Tmf(n)$, respectively. The latter was recently
defined by work of Goerss--Hopkins and Hill--Lawson \cite{HillLawson}. 

We can prove the following two theorems: 
\begin{thm}For every $n$, the $E_\infty$-ring spectrum $\TMF(n)$ is a faithful
$GL_2(\Z/n\Z)$-Galois extension of $\TMF[\frac1n]$.\end{thm}
\begin{thm}For every $n$, the norm map $\Tmf(n)_{hGL_2(\Z/n\Z)} \to
\Tmf(n)^{hGL_2(\Z/n\Z)}$ is an equivalence. Equivalently, the Tate spectrum
$\Tmf(n)_{tGL_2(\Z/n\Z)}$ vanishes.\end{thm}
Note that the vanishing of Tate spectra is automatic for faithful Galois
extensions, but $\Tmf[\frac1n]\to \Tmf(n)$ is not a Galois extension. Note
furthermore that the second theorem was proven by Stojanoska in \cite{St12} in
the case $n=2$ in her investigation of the Anderson self-duality of $\Tmf$. We hope that our results about the vanishing of Tate spectra will have future applications to duality.\\

In \Cref{Section2}, we will discuss various background material. This includes the relationship between formal groups and even periodic ring spectra and, furthermore, derived stacks and coarse moduli spaces. The knowledgeable reader may just want to pick up our definitions of an even periodic refinement (\Cref{EPR}), of a derived stack (\Cref{DS}) and of a tame morphism (\Cref{DefinitionTame}). In \Cref{Section3}, we discuss first an abstract characterization of derived stacks for which the global sections functor is an equivalence (via the Schwede--Shipley theorem). Then we show certain descent and ascent properties of this class of derived stacks. In \Cref{Section4}, we specialize these abstract theorems to chromatic homotopy theory and obtain our main theorem. \Cref{Section5} contains our abstract theorems about Galois extensions and the vanishing of Tate spectra, which are then applied to examples in Sections \ref{Section6} and \ref{AppTMF}. \Cref{AppHLS} discusses the behavior of $\infty$-category 
valued sheaves with respect to a finite open cover of a topological space (as used in \Cref{Section3}). 

Throughout this paper, we use the language of quasicategories (i.e., 
$\infty$-categories) of \cite{Joy02} and \cite{HTT}, and the theory of structured ring
spectra as developed originally in \cite{EKMM}, and
formulated in $\infty$-categorical terms in \cite{higheralg}. 
We will let $\mathcal{S}$ denote the $\infty$-category of spaces, $\sp$ the
$\infty$-category of spectra, and we will write $\otimes$
for the smash product in the latter.

\subsection*{Acknowledgments}
We would like to thank  Dan Dugger, Mike Hopkins, Tyler Lawson, Jacob Lurie, Niko Naumann, and Vesna Stojanoska for several helpful
discussions related to the subject of this paper. 
The first author is supported by the NSF Graduate Research Fellowship
under grant DGE-110640. 

\section{Derived stacks}\label{Section2}

We will take a naive approach to derived stacks in this paper, and avoid the most
general theory. In this section, we summarize what we need, and briefly review the
role of formal groups. Furthermore, we will review the theory of coarse moduli spaces and the Zariski topology for algebraic stacks.

\subsection{Even periodic ring spectra and formal groups}\label{Even-FG}

Recall first:

\begin{definition} Let $M_{FG}$ be the \emph{moduli stack of formal groups}: that is, it is the
(infinite-dimensional) stack assigning to a commutative ring $R$ the groupoid
of one-dimensional, commutative formal groups over $R$ and isomorphisms between them. 
\end{definition}

Define $MUP = \bigvee_{k\in\Z} \Sigma^{2k} MU$ to be a periodic version of
complex bordism $MU$. A theorem of Quillen (see, e.g., \cite{Adams}) shows that $MU_* = MUP_0$
 is isomorphic to the Lazard ring $L$, which carries the universal formal group
 law. Even more is true: the simplicial scheme $\spec \pi_0( MUP^{\tensor
 \bullet+1})$ is isomorphic to the simplicial scheme $(\spec L)^{\times_{M_{FG}} \bullet+1}$. In particular, if we use the notation $\spec W = \spec L\times_{M_{FG}} \spec L$, it is true that the Hopf algebroids $(MUP_0, MUP_0 MUP)$ and $(L,W)$ are isomorphic. 

\begin{cons}
Given a spectrum $X$, both $MUP_0(X)$ and $MUP_1(X)$ are comodules over the Hopf algebroid $(MUP_0, MUP_0MUP)$. Via the equivalence between $(MUP_0, MUP_0 MUP)$-comodules and quasi-coherent sheaves on
 $M_{FG}$, this defined a $\Z/2$-graded sheaf $\FF_*(X)$. This sheaf is
 characterized by the property that the evaluation of $\FF_0(X)$ on $\spec L$
 agrees with $MUP_0(X)$ as comodules over $(MUP_0, MUP_0 MUP) \cong (L,W)$, and 
 evaluation of $\FF_1(X)$ on $\spec L$ agrees with $MUP_1(X)$. \end{cons}

The $\mathbb{Z}/2$-graded sheaf $\mathcal{F}_*(X)$
can often be
described explicitly. For example, the sphere $S^{-2}$ is associated to a line bundle
$\omega \in \mathrm{Pic}(M_{FG})$ which assigns to every formal group the dual
of its \emph{Lie algebra}. Moreover, the theorem above by Quillen implies that $\FF_0(MUP) = (\phi_L)_* \mathcal{O}_{\spec L}$ and $\FF_0(MUP\tensor MUP) = (\phi_W)_*\mathcal{O}_{\spec W}$, where $\phi_L\colon \spec L \to M_{FG}$ and $\phi_L\colon \spec W \to M_{FG}$ are the obvious maps. 
This point of view has been very fruitful in describing large-scale features of
stable homotopy theory via the special geometry of $M_{FG}$. 

An example of this connection is the partial correspondence between certain
ring spectra and formal groups: one can associate a formal group to certain
ring spectra, and in some cases one can recover the value of the
associated cohomology theory in terms of the $\mathbb{Z}/2$-graded
quasi-coherent sheaf $\mathcal{F}_*(X)$ on $M_{FG}$. 
In this way, complex bordism allows one to manufacture a great deal of new ring
spectra. 

\begin{definition}[\cite{AHS}] 
A homotopy commutative ring spectrum $E$ is said to be \emph{even periodic} if
$\pi_{i} E = 0$ for $i $ odd  and if $\pi_2 E$ is an invertible module over
$\pi_0 E$, with inverse $\pi_{-2} E$, such that 
\[ \pi_{2k} E \simeq (\pi_2 E)^{\otimes k}, \quad k \in \mathbb{Z},  \]
under multiplication. This is slightly weaker than the definition in
\cite{AHS}, which requires $\pi_2 E$ to be the trivial invertible module, i.e.\ to contain a unit. We will refer to such ring
spectra as \emph{strongly even periodic.}
\end{definition} 

Given an even periodic ring spectrum $E$, the Atiyah-Hirzebruch spectral
sequence for the $E$-cohomology of any \emph{even} space $X$ (i.e., with
integral homology free and concentrated in even dimensions) degenerates. For
example, if $E$ is strongly even periodic, there is an isomorphism of rings $E^0( \mathbb{CP}^\infty) = \pi_0
E[[x]]$, where the generator $x \in \widetilde{E}^0(\mathbb{CP}^\infty)$ is
noncanonical (and called a \emph{complex orientation} of $E$). The
multiplication on $\mathbb{CP}^\infty$ is dual to a \emph{comultiplication}
in $E^0(\mathbb{CP}^\infty)$, which gives $E^0(\mathbb{CP}^\infty)$ the
structure of a (continuous) commutative, cocommutative Hopf algebra.
Equivalently, the formal scheme $\mathrm{Spf} E^0(\mathbb{CP}^\infty)$ is canonically a
\emph{formal group} over $\pi_0 E$. This persists for a general even periodic
ring spectrum, although the formal group need only admit a coordinate Zariski
locally on $\pi_0 E$, and we get a map $\spec \pi_0 E \to M_{FG}$.
This is one direction of the correspondence between ring spectra  and
formal groups. 

In some cases, one can reconstruct the cohomology theory (and even the ring
spectrum)
from the formal group. 
For example, the \emph{Landweber exact functor theorem} \cite{lex} gives a concrete and
often easily checked criterion for a map $\phi\colon \spec R \to M_{FG}$ to be flat. Given
such a flat map,  and a spectrum $X$, one can pull back  the $\mathbb{Z}/2$-graded quasi-coherent
sheaf
$\mathcal{F}_*(X)$ to $\spec R$ to define an invariant of $X$, which is in
fact an even periodic homology theory $E$. More precisely, we define
\begin{align*}
E_{2k}(X) &:=  \Gamma(\spec R, \phi^*(\FF_0(X)\tensor \omega^{\tensor k})) = (\FF_0(X)\tensor \omega^{\tensor k})(\spec R)\\
E_{2k+1}(X) &:=  \Gamma(\spec R, \phi^*(\FF_1(X)\tensor \omega^{\tensor k})) = (\FF_1(X)\tensor \omega^{\tensor k})(\spec R)
\end{align*}

Given a flat morphism $\phi\colon \spec R \to M_{FG}$, the formal group of the Landweber-exact even periodic cohomology theory
that one obtains is precisely classified by the map $\phi$. Conversely, given
an even periodic ring spectrum $E$, one obtains a map $\phi\colon \spec \pi_0 E \to
M_{FG}$ classifying the formal group $\mathrm{Spf} E^0(\mathbb{CP}^\infty)$; if
this map is flat, then $E$ is the Landweber-exact theory obtained from 
$\phi$. An important example of such a theory is given by complex $K$-theory
$KU$, as was first
shown (without using Landweber's theorem) by Conner and Floyd.

The following proposition is well-known.
\begin{proposition}\label{LandweberSmash}
Given two flat morphisms $\phi_R\colon \spec R \to M_{FG}$ and $\phi_{R'}\colon \spec R' \to M_{FG}$, we denote the corresponding Landweber exact spectra by $E_R$ and $E_{R'}$. With this notation, we have an isomorphism 
\[ \pi_{2k}(E_R\tensor E_{R'} )\cong \omega^{\tensor k}(\spec R\times_{M_{FG}}\spec R').\]
\end{proposition}
\begin{proof}
We first investigate the situation $\phi_{R'} = \phi_L$ so that $E_{R'} = MUP$. By definition,
\[MUP_0(E_R) = (E_R)_0(MUP) = (\FF_0(MUP))(\spec R).\]
The latter agrees with
 \[((\phi_L)_*\mathcal{O}_{\spec L})(\spec R) = \mathcal{O}_{M_{FG}}(\spec L\times_{M_{FG}} \spec R) = ((\phi_R)_*\mathcal{O}_{\spec R})(\spec L).\]
 Similarly, $(MUP\tensor MUP)_0(E_R) = ((\phi_R)_*\mathcal{O}_{\spec R})(\spec W)$, which implies $\FF(E_R) = (\phi_R)_*\mathcal{O}_{\spec R}$. 
 
 In the general case, we get now:
 \begin{align*}\pi_{2k}(E_R\tensor E_{R'} ) &= (E_{R'})_{2k}(E_R) \\
 &= ((\phi_R)_*\mathcal{O}_{\spec R} \tensor \omega^{\tensor k})(\spec R') \\
 &= ((\phi_R)_*(\phi_R)^*\omega^{\tensor k})(\spec R') \\
 &=\omega^{\tensor k}(\spec R\times_{M_{FG}}\spec R')
 \end{align*}
\end{proof}

\newcommand{\affx}{\mathrm{Aff}^{et}_{/X}}
\newcommand{\aff}{\mathrm{Aff}^{et}}

\subsection{Even periodic enhancements and derived stacks}
The upshot of the discussion of the previous section is that there
is a \emph{presheaf of even periodic homology theories} on the affine flat site of $M_{FG}$.
Equivalently, for every commutative ring $R$ with a formal 
group over $R$ classified by a flat map $\spec R \to M_{FG}$, one
obtains an even periodic homology theory $E$ with $E_0$ given by $R$, and
one obtains morphisms between homology theories from morphisms of formal groups. 
One can show that one actually gets a homotopy commutative
ring spectrum from each such Landweber-exact homology theory, and that
each map of formal groups gives a map of ring spectra, such that all
functoriality holds up to homotopy, albeit not coherent homotopy (see
\cite[Theorem 2.8]{HoveyS} or \cite[Lecture 18]{chromatic}).

For the purposes of homotopy theory, a diagram such as above, which takes
values in the \emph{homotopy category} of spectra, is insufficient to make
many natural constructions, such as homotopy limits and colimits. For example, 
there is no way to extend the above construction to a non-affine scheme (or
stack) flat over $M_{FG}$. 
Given a (discrete) group acting on a formal group, that does not produce a
strict group action on the
associated spectrum. 
Moreover, the ring spectra one obtains do not have
the structure needed to perform algebraic constructions with them: for example,
one cannot 
 generally obtain a good theory of modules 
(e.g., a triangulated category or stable $\infty$-category) over an
unstructured ring spectrum. 

However, in certain
restricted cases, it is possible to realize diagrams of 
homology theories much more rigidly. A survey of this problem, including a
general result of Lurie, is in
\cite{goerss}.

Let $X$ be a Deligne-Mumford stack together with a flat map\footnote{A map $X
\to M_{FG}$, not necessarily representable, is \emph{flat} if for every \'etale
covering $\spec R \to X$, the composite $\spec R \to X \to M_{FG}$ is
flat in the sense that for every map $\spec A \to M_{FG}$, the pull-back $\spec
A \times_{M_{FG}} \spec R \to \spec A$ (which is a map of schemes) is flat.} 
\[ X \to M_{FG} , \]
so that, as above, one obtains a presheaf of multiplicative homology theories
on the affine flat site of $X$. 
Let $\affx$ be the affine, \emph{\'etale} site of $X$. 

\begin{definition}\label{EPR}
An \emph{even periodic enhancement}  or \emph{even periodic refinement} $\mathfrak{X}$ of $X$ is a sheaf $\otop$
of even periodic $E_\infty$-rings on the site $\affx$, lifting the above
diagram of homology theories on $\affx$. 
\end{definition} 

In other words, for an \'etale map $\spec R \to X$, the $E_\infty$-ring $\otop(
\spec R)$ defines an even periodic cohomology theory, with formal group given by 
the classifying map $\spec R \to X \to M_{FG}$: it yields the Landweber-exact
(co)homology theory associated to this  formal group. The ``sheaf'' condition is
actually redundant here, because by construction, the homotopy groups of
$\otop$ already form a sheaf on the affine \'etale site. Note that the phrase ``sheaf
of spectra'' refers to a functor from the category $\affx$ into the
$\infty$-category of spectra (e.g., realized via a functor into some model
category) and does not refer to the homotopy category. 

Even periodic enhancements are examples of (even periodic) derived stacks. Our notion of a derived stack is a special case of the notion of
a nonconnective spectral Deligne--Mumford stack in Lurie's DAG series (see \cite[8.5 and 8.42]{DAGss} for his definition). We
prefer to spell this special case out (informally) for the convenience of the reader.
\begin{definition}\label{DS}A \textit{derived stack} $\XX$ will be for us a
Deligne--Mumford stack $X$ together with a sheaf of $E_\infty$-ring spectra
$\otop = \otop_\XX$ on $\affx$ and an isomorphism $\pi_0 \otop_{\mathfrak{X}}
\cong \mathcal{O}_X$. Here $\pi_i \otop_{\mathfrak{X}}$ is the sheaf $U \mapsto \pi_i \left( \otop_{\mathfrak{X}}(U)\right)$ on $\affx$. 
Furthermore, one requires $\pi_i \otop_{\mathfrak{X}} $ to be quasi-coherent as
an $\mathcal{O}_X$-module. 

The derived stack $\XX$ is called \textit{even periodic} if $\omega
= \pi_2\otop_\XX$ is a line bundle 
such that multiplication induces isomorphisms 
\[\pi_{2k}\otop_\XX \otimes \pi_{2l} \otop_{\XX}\cong
\pi_{2(k+l)}\otop_\XX, \quad k, l \in \mathbb{Z},\]
 and we have 
\[\pi_i\otop_\XX = 0, \quad \text{for }i \ \text{odd}.\] 
\end{definition}
Next, we want to define morphisms of derived stacks. If $f\colon Y \to X$ is a map of Deligne--Mumford stacks and $\FF$ a sheaf of spectra on $\affx$, then we can define a sheaf of spectra $f^{-1}\FF$ on $\aff_{/Y}$ as the sheafification of the presheaf given by
\[(f^{-1}_{pre}\FF)(U\to Y) = \hocolim_{U\to V \to X,\, V\to X \text{ \'etale}}\, \FF(V).\]
As this homotopy colimit is filtered, in a 2-categorical sense, it follows that if $\otop_\XX$ is a sheaf of
$E_\infty$-ring spectra, then $f^{-1}_{pre}\otop_\XX$ is a presheaf of $E_\infty$-ring spectra on $\aff_{/Y}$ (\cite[3.2.3.2]{higheralg}) and thus $f^{-1}\otop_\XX$ a sheaf of $E_\infty$-ring spectra (\cite[1.15]{DAGss}). Furthermore $f^{-1}\pi_*(\FF) \to\pi_*(f^{-1}\FF)$ is an isomorphism. 
\begin{definition}Let $\XX = (X,\otop_\XX)$ and $\YY = (Y,\otop_\YY)$ be
derived stacks. Then a morphism $f\colon \YY \to \XX$ of derived stacks
consists of a morphism $f_0\colon  Y \to X$ of the underlying Deligne--Mumford
stacks and a morphism $\alpha\colon f_0^{-1}\otop_\XX \to \otop_\YY$ of
$E_\infty$-ring spectra such that $\pi_0 \alpha$ coincides with the morphism
$f_0^{-1}\mathcal{O}_X \to \mathcal{O}_Y$ defined by $f_0$.\end{definition}
Given such a morphism $f\colon \YY \to \XX$ of derived stacks and an $\otop_\XX$-module $\FF$, we define $f^*\FF$ as $f^{-1}\FF \otimes_{f^{-1}\otop_\XX} \otop_\YY$.

In an evident manner, an even periodic enhancement of $X$ defines even periodic
enhancements of each stack \'etale over $X$. 
Given an even periodic enhancement, it follows that 
one \emph{can} evaluate the sheaf $\otop$ on 
any stack $Y$ \'etale over $X$. Namely, one defines
\[ \otop( Y) = \mathrm{holim}_{\spec R \to Y} \otop( \spec R),  \]
as $\spec R \to Y$ ranges over all the \'etale morphisms from affine schemes. 
Then $\otop(Y)$ is naturally an $E_\infty$-ring. 
Such spectra will generally fail (if $Y$ is not affine) to be Landweber-exact
or even periodic, and may exhibit intricate torsion phenomena. 
For example, we can consider $\otop( X)$ itself, which we can think of the ring
of ``functions.''
Below, we will write $\Gamma( \mathfrak{X}, \otop)$ for this.

\begin{remark}
It is also fruitful to consider derived stacks as representing some type of
moduli problem for (possibly nonconnective) structured ring spectra. This point
of view was used by Lurie to give a construction of the even periodic
enhancement  of the moduli
stack of elliptic curves in \cite{survey}, producing the spectrum of
topological modular forms $\TMF$.  
\end{remark}

\subsection{Quasi-coherent sheaves}
In this subsection, we will review the basics of quasi-coherent sheaves 
on derived stacks. Fix one such $\mathfrak{X}=({X}, \otop)$. Given an $E_\infty$-ring $A$, we write
$\md(A)$ for the stable $\infty$-category of $A$-modules. 

\begin{definition} 
\label{def:qcoh}
The $\infty$-category $\qcoh(\mathfrak{X})$ of quasi-coherent sheaves on $\mathfrak{X}$
is the homotopy limit
\[ \qcoh(\mathfrak{X}) \stackrel{\mathrm{def}}{=} \mathrm{holim}_{( \spec R \to X ) \in
\affx}  \md ( \otop( \spec R)).  \]

In other words, a quasi-coherent sheaf on $\mathfrak{X}$ assigns to every
\'etale map $\spec R \to X$, a module  $M_R$ over $\otop( \spec R)$, together
with equivalences
\[ M_R \otimes_{\otop(\spec R)} \otop(\spec R') \simeq M_{R'} , \]
for each (2-)commutative diagram 
\begin{equation} \label{2comm} \xymatrix{
\spec R' \ar[rr] \ar[rd] & &  \spec R \ar[ld] \\
& X
},\end{equation}
and 
appropriate compatibility data between these equivalences. 
\end{definition} 

We note that one is constructing a homotopy limit of presentable, stable
$\infty$-categories under colimit-preserving, exact functors. It follows that the
homotopy limit is itself a presentable, stable $\infty$-category where homotopy  colimits
are computed ``pointwise'' (see \cite[Proposition 5.5.3.13]{HTT}). 

Using the derived version of flat descent theory \cite{DAGdesc}, which states that the
assignment $A \mapsto \md(A)$ for an $E_\infty$-ring $A$ is a sheaf of
$\infty$-categories in the flat
topology on affine (derived) schemes, it follows that one can give an alternative
definition. Suppose first $X$ has affine diagonal.  Choose an \'etale
\emph{surjection} $\spec R \to X$. Then 
$\qcoh(\mathfrak{X})$ is the homotopy limit of the cosimplicial diagram
of $\infty$-categories
\[ 
\md( \otop(\spec R)) \rightrightarrows \md( \otop(\spec R \times_{{X}}
\spec R)) 
\triplearrows 
\dots . 
\]
If the diagonal of $X$ is not affine, then one should use an \'etale hypercover
rather than a Cech cover. 

Let $\mathcal{F}$ be a quasi-coherent sheaf on $\mathfrak{X}$. Then, for each
$k$, the
assignment 
\[ (\spec R \to X) \in \affx \mapsto \pi_k \left( \mathcal{F}( \spec R)\right) , \]
defines a quasi-coherent sheaf $\pi_k \mathcal{F}$ on the \emph{ordinary} stack $X$: that is, it
assigns an $R$-module (in the classical sense) to each \'etale map $\spec R \to
X$, together with appropriate equivalences and compatibility data. We note that 
no further sheafification is required since we are working with affine schemes. Given a 2-commuting diagram
\eqref{2comm}, the map $\otop( \spec R) \to \otop(\spec R')$ is flat (even \'etale)
on homotopy groups, and it follows that one has canonical \emph{isomorphisms}
\[  \pi_k \left( \mathcal{F}( \spec R)\right) \otimes_{R} R' \simeq 
\pi_k \left( \mathcal{F}( \spec R')\right) .
\]
and thus $\pi_k\FF$ is quasi-coherent.
In the even periodic case, only $\pi_0$ and $\pi_1$ are necessary for bookkeeping, because 
\[ \pi_{n + 2k} \mathcal{F} \simeq \pi_n \mathcal{F} \otimes \omega^k,  \]
where $\omega = \pi_2\otop$. 

\begin{example} 
Let $T$ be a spectrum. Then one has a quasi-coherent sheaf $\otop \otimes T \in
\qcoh( \mathfrak{X})$, given by 
\[ ( \spec  R\to X ) \mapsto \otop( \spec R) \otimes T.  \]
In fact, the category $\qcoh(\mathfrak{X})$ (like any presentable, stable
$\infty$-category) is canonically \emph{tensored} over spectra in this way.

Suppose $\mathfrak{X}$ is an even periodic
refinement of a flat map $X \to M_{FG}$. 
Then the  homotopy groups $\pi_0( \otop \otimes T), \pi_1( \otop \otimes T)$ are given by the pull-back of the
$\mathbb{Z}/2$-graded sheaf $\mathcal{F}_*(T)$ on $M_{FG}$ to $X$ via
the given map $X \to M_{FG}$, since we have assumed that the diagram $\otop$
of $E_\infty$-rings lifts the diagram of Landweber-exact homology theories. 
\end{example}

These homotopy groups $\pi_k \mathcal{F}$ are important for several reasons;
one is that the homotopy groups of the \emph{global sections} 
\[ \Gamma( \mathfrak{X}, \mathcal{F}) = \mathrm{holim}_{(\spec R \to X) \in
\affx} \mathcal{F}(\spec R)  \]
of $\FF$ are the abutment of a
\emph{descent spectral sequence}
\[ H^i(X, \pi_j \mathcal{F} ) \implies \pi_{j-i} \Gamma(
\mathfrak{X, }\mathcal{F}). \]
We will sometimes abbreviate the descent spectral sequence to DSS. 

Let $\Gamma(\mathfrak{X}, \otop)$ be the 
$E_\infty$-ring of global sections of the structure sheaf. Then the global
sections functor on $\qcoh(\mathfrak{X})$ takes values in $\Gamma(\mathfrak{X}, \otop)$-modules. Indeed, one has a
functor of ``tensoring up''
\[ \md (\Gamma(\mathfrak{X}, \otop)) \to \qcoh(\mathfrak{X}),  \]
that sends an $\Gamma(\mathfrak{X}, \otop)$-module $M$ to the quasi-coherent
sheaf $$( \spec R \to X) \mapsto \otop( \spec R) \otimes_{\Gamma(\mathfrak{X},
\otop)} M.$$ The global sections functor is the right adjoint to ``tensoring up.''
The relation between these two $\infty$-categories given by this adjoint pair
is the main subject of this paper. 

In the rest of the next two subsections, we will discuss several important examples of even
periodic refinements. 

\subsection{Affine schemes}
We begin with the following basic observation: 
\begin{proposition} 
\label{affinecase}
Let $A$ be a Landweber exact, even periodic $E_\infty$-ring. Then the affine scheme $\spec
\pi_0 A $, together with the natural map $\spec \pi_0 A \to M_{FG}$, has a
canonical even periodic enhancement, and its category of quasi-coherent
sheaves is equivalent to $\md(A)$. 
\end{proposition} 
\begin{proof} 
It suffices to show that for every \'etale $\pi_0 A$-algebra $A'_0$, there
exists an even periodic $E_\infty$-$A$-algebra $A'$ with the
property that 
$\pi_0 A' \simeq A'_0$, and that this construction can be done functorially in
$A'_0$. This follows from \S 8.4 of \cite{higheralg}, reviewed below, which implies
that the $\infty$-category of such $A$-algebras is equivalent to the discrete
category of \'etale $\pi_0 A$-algebras. 
\end{proof}

The basic result about $E_\infty$-rings needed for the above is the following
derived version of the ``topological invariance of the \'etale site,''
a proof of which appears in \S 8.4 of \cite{higheralg}:

\begin{theorem} \label{topinv}Let $R$ be an $E_\infty$-ring, and consider
the $\infty$-category $\einf_{R/}$ of $E_\infty$-rings under $R$. Let
$\einf_{R/}^{\mathrm{et}} \subset \einf_{R/}$ be the subcategory of \'etale
$R$-algebras: that is, those $R$-algebras $R'$ with the properties that:
\begin{enumerate}
\item $\pi_0 R'$ is \'etale over $\pi_0 R$. 
\item  $\pi_* R' \simeq \pi_* R \otimes_{\pi_0 R} \pi_0 R'$.\end{enumerate}

Then we have an equivalence of $\infty$-categories
\[ \einf_{R/}^{\mathrm{et}} \stackrel{\pi_0}{\simeq} \rng_{\pi_0R/}^{\mathrm{et}},\]
with the (discrete) category of \'etale $\pi_0R$-algebras. 
\end{theorem} 

In other words, if $A \to B$ is an \'etale morphism in $\einf$, then for any $
B' \in \einf$, we have a homotopy cartesian square
of spaces
\begin{equation} \label{fiberhomsquare} \xymatrix{
\hom_{\einf}(B, B') \ar[d] \ar[r] &  \hom_{\einf}(A, B') \ar[d]  \\
\hom_{\mathrm{Ring}}( \pi_0 B, \pi_0 B') \ar[r] & \hom_{\mathrm{Ring}}( \pi_0 A, \pi_0 B')
},\end{equation}
where both horizontal arrows are given by precomposition. 
 It will be useful to have the following slight generalization (and corollary) of \Cref{topinv}. 
\begin{corollary} 
\label{rem:etale}
Let $\mathcal{C}$ be an $\infty$-category, and 
\( F\colon \mathcal{C} \to \einf,  \)
be a functor to $E_\infty$-rings. Consider the composite $\overline{F} \colon \mathcal{C}
\to \einf \stackrel{\pi_0}{\to} \rng$, and consider an extension
\(  \overline{G} \in \mathrm{Fun}( \mathcal{C} \times \Delta^1, \rng), \)
of $\overline{F}$, in the sense that the restriction of $\overline{G}$ to the
first vertex is identified with $\overline{F}$. Suppose that for each $x \in
\mathcal{C}$, the morphism $\overline{G}(x)$ is \'etale. Then there is a unique
extension 
\( G \in \mathrm{Fun}( \mathcal{C} \times \Delta^1 , \einf),  \)
of both $F$ and $\overline{G}$. 
\end{corollary} 
\begin{proof} 
\newcommand{\fet}{\mathrm{Fun}^{\mathrm{et}}}
Let $\fet(\Delta^1, \einf)$ denote the full subcategory of $\mathrm{Fun}(\Delta^1,
\einf)$ spanned by those morphisms of $E_\infty$-rings which are \'etale. 
Define $\fet( \Delta^1, \mathrm{Ring})$ similarly.
Then \Cref{topinv} gives us that the natural functor
\[ \fet( \Delta^1, \einf) \to \einf \times_{\mathrm{Ring}}
\fet(\Delta^1, \mathrm{Ring}), \quad (A \mapsto B) \mapsto \left\{A,
\pi_0 A, \pi_0 A \to \pi_0 B\right\}  \]
is  an equivalence of $\infty$-categories.

Indeed, the existence part of \Cref{topinv}
gives essential surjectivity. 
To see full faithfulness, consider two objects $A \to B, A' \to B'$ in
$\fet(\Delta^1, \einf)$. 
Then 
\begin{align*} \hom_{\fet (\Delta^1, \einf)}( (A \to B), (A' \to B')) 
& \simeq   \hom_{\einf}(A, A')\times_{\hom_{\einf}(A, B')}\hom_{\einf}(B, B') \\
& \simeq \hom_{\einf}(A, A') \times_{\hom_{\mathrm{Ring}(\pi_0 A, \pi_0 B')}} 
\hom_{\mathrm{Ring}}(\pi_0 B, \pi_0 B') ,\end{align*}
where the last equivalence holds because \eqref{fiberhomsquare} is 
homotopy cartesian. 
This shows that our functor is fully faithful. 

Finally, we find that 
\[ \mathrm{Fun}(\mathcal{C}, \fet(\Delta^1, \einf)) 
\simeq \mathrm{Fun}(\mathcal{C}, \einf) \times_{\mathrm{Fun}(\mathcal{C},
\mathrm{Ring}) 
}
\mathrm{Fun}(\mathcal{C}, \fet(\Delta^1, \mathrm{Ring}))
\]
is an equivalence of $\infty$-categories, which is equivalent to the desired statement. 
\end{proof} 

In other words, the ``topological invariance of
the \'etale site'' can be done functorially.

\begin{example} 
Let $G$ be a group acting on an $E_\infty$-ring $R$. Suppose given an \'etale extension $T_0$ of
$\pi_0 R$. Suppose $T_0$ is given a $G$-action in such a way that 
\[ \pi_0 R \to T_0  \]
is $G$-equivariant. Then the \'etale extension $R \to T$ constructed in
\Cref{topinv} canonically has a $G$-action in view of \Cref{rem:etale}. 
\end{example} 
\subsection{Group actions}
\label{group:actions}
In this subsection, we give the most basic non-affine example of an even periodic
refinement. 

Let $R$ be an even periodic, Landweber-exact $E_\infty$-ring, and let $G$ be a
finite group acting on $R$ (in the $\infty$-category of $E_\infty$-rings). 
Then we get a map from even periodicity,
$\spec \pi_0 R \to M_{FG}$, which has canonically the structure of a $G$-equivariant
map: that is,   $G$ acts compatibly on the formal group. Consequently, we get a map of stacks,
\[ (\spec \pi_0 R)/G \to M_{FG} .\]

\newcommand{\spf}{\mathrm{Spf}}
For example, take $R$ to be complex $K$-theory $KU$. Then there is a
$\mathbb{Z}/2$-action on $KU$ coming from complex conjugation of vector
bundles, which can be made into a $\mathbb{Z}/2$-action in $E_\infty$-rings.
At the level of formal groups, one has
\[ \spf KU^0 ( \mathbb{CP}^\infty) \simeq \widehat{\mathbb{G}_m},  \]
i.e., the formal group of $KU$ is the formal multiplicative group, which is
classified by a flat map
\[ \spec \mathbb{Z} \to M_{FG},  \]
so that $KU$ is Landweber-exact. 
The $\mathbb{Z}/2$-action on $KU$ corresponds to the involution of
$\widehat{\mathbb{G}_m}$ given by $x \mapsto x^{-1}$. In particular, one
obtains a map 
\[ B \mathbb{Z}/2 \to M_{FG}.  \]
This map takes a $\mathbb{Z}/2$-torsor over a scheme $\spec R$ and outputs the 
formal completion of  the associated 
\emph{one-dimensional torus} over $\spec R$, not necessarily split, in such a
way that the $\mathbb{Z}/2$-action on a torsor maps to the
$\mathbb{Z}/2$-action on the torus given by inversion. Since $\mathrm{Aut}(
\mathbb{G}_m) \simeq \mathbb{Z}/2$, the stack $B \mathbb{Z}/2$ classifies
precisely one-dimensional tori.

The next result shows that we can obtain an even periodic refinement of stacks
such as $B \mathbb{Z}/2$. 
\begin{proposition} \label{gpquotient}
If $R$ and $G$ are as above and $R$ is Landweber exact, then there is a canonical even periodic
refinement of $(\spec \pi_0 R)/G \to M_{FG}$. 
\end{proposition} 

See also \cite{LN2} for the example of $KU$-theory. 
\begin{proof} 
Consider an \'etale map
\(  \spec T \to ( \spec \pi_0 R)/G,\) from which we form the pull-back
\[ \xymatrix{
\spec T' \ar[d] \ar[r] & \spec \pi_0 R \ar[d] \\
\spec T \ar[r] &  (\spec \pi_0 R)/G
}.\]

Since $\spec T' $ is \'etale over $\spec \pi_0 R$, we have defined an even
periodic, Landweber-exact $E_\infty$-ring $\otop( \spec T')$, which is
\'etale over $R$. 
Since the group $G$ acts on $R$, it follows (\Cref{rem:etale}) that it acts on $\otop( \spec T')$
in a compatible manner; we set
\[ \otop( \spec T ) \stackrel{\mathrm{def}}{=} \otop( \spec T')^{h G}.  \]

Since $G$ acts freely on $\spec T'$ (that is, the map $\spec T' \to \spec T$
is a $G$-torsor),  it follows that 
there is no higher cohomology for the $G$-action on $\pi_* \otop( \spec T')$. 
This is a consequence of the fact that $T \to T'$ is a $G$-Galois extension
of commutative rings: that is, after the faithfully flat base-change $T \to T'$,
we have an equivalence of $T'$-modules with $G$-action, $T' \otimes_T T'
\simeq \prod_G T'$. Moreover, coinduced representations of $G$ have no higher cohomology. 
The homotopy fixed point spectral sequence thus 
degenerates and we get
\[ \pi_* \otop( \spec T) \simeq \left(\pi_* \otop( \spec T')\right)^G,  \]
which implies that $\otop( \spec T')$ is the desired even periodic, Landweber
exact $E_\infty$-ring. 

This procedure thus gives, for any \'etale map $\spec T \to (\spec \pi_0 R)/G$,
a even periodic, Landweber-exact $E_\infty$-ring $\otop(\spec T)$, and this is the
structure sheaf for the even periodic refinement of $(\spec \pi_0 R)/G$ desired. 
\end{proof}

This result has a converse. If $\mathfrak{X}=({X}, \otop)$ is an even periodic
refinement of $X \to M_{FG}$, and if $X$ is the quotient $(\spec R)/G$ for a
finite group $G$ acting on an affine scheme $\spec R$, then $\mathfrak{X}$ arises in this way from the $G$-action on the $E_\infty$-ring $\otop(
\spec R)$.

Let $R$ be an $E_\infty$-ring with a $G$-action as above, let
$\mathfrak{X}$ be the associated even periodic refinement of $\spec \pi_0 R$,
and let $\mathfrak{Y}$ be the associated even periodic refinement of $(\spec
\pi_0 R )/G$. The next result describes quasi-coherent sheaves on
$\mathfrak{Y}$ in terms of $\mathfrak{X}$. 
Note first that since $G$ acts on $R$, it acts on the stable $\infty$-category
$\mod(R)$, in symmetric monoidal $\infty$-categories. 

\begin{proposition}\label{GaloisDesc}
One has equivalences of symmetric monoidal $\infty$-categories $$\qcoh( \mathfrak{Y}) \simeq \qcoh( \mathfrak{X})^{h G} \simeq
\mod(R)^{hG}. $$
\end{proposition} 
\begin{proof} 
This is a formal descent-theoretic statement: in an appropriate $\infty$-category of derived
stacks, $\mathfrak{Y}$ is the homotopy quotient $(\mathfrak{X})_{hG}$, and the
construction $\qcoh$ is defined so as to send homotopy colimits to homotopy limits (of stable
$\infty$-categories). Let us prove it directly in our setup. 

We have an \'etale cover $\spec \pi_0 R \to (\spec \pi_0 R)/G$, and therefore
\[ \qcoh( \mathfrak{Y}) \simeq \mathrm{holim} \left( 
\mod( \otop( \spec \pi_0 R) ) \rightrightarrows \mod( \otop( \spec \pi_0 R \times_{(\spec
\pi_0 R )/G} \spec \pi_0 R)) \triplearrows \dots \right).
\]
Since $\spec \pi_0 R \to (\spec \pi_0 R)/G$ is a $G$-torsor, the iterated fiber
products that appear in the above construction are precisely 
\[ G \times G \times \dots \times \spec \pi_0 R,  \]
the rings in question are $\prod_{G^n} R$, and the above cosimplicial diagram
is the usual cobar construction for homotopy fixed points: the
construction $R \mapsto \mod(R)$ sends products in $R$ to products of
$\infty$-categories.  
\end{proof}

In other words, to give a quasi-coherent sheaf on $\mathfrak{Y}$ is equivalent
to giving an $R$-module $M$, together with a $G$-action on $M$ intertwining the
$G$-action on $R$.

\begin{example}
\label{KO:example}
The most basic example of all this comes from the $\mathbb{Z}/2$-action on $KU$
described above. 
By \Cref{gpquotient}, it endows the stack $B\mathbb{Z}/2$ of one-dimensional
tori with an even periodic refinement. The global sections of the structure
sheaf give $KU^{h \mathbb{Z}/2} \simeq KO$. 

The $\infty$-category of quasi-coherent sheaves on the derived version 
of $B \mathbb{Z}/2$ is precisely $\md(KU)^{h \mathbb{Z}/2}$, where the
$\mathbb{Z}/2$-action is by complex conjugation on $KU$-modules: it takes a
$KU$-module $M$ and ``twists'' the $KU$-action by $\Psi^{-1}$. We will show later (as is
well-known) that this is
equivalent to the $\infty$-category $\md(KO)$. 
\end{example}

\begin{example}[Classical Galois descent] \label{galoisdescclass}
In classical commutative algebra, recall that if $R \to R'$ 
is a morphism of rings which is a $G$-torsor for a finite group $G$ (or rather,
$\spec R' \to \spec R$ is a $G$-torsor), then we have an equivalence between
the category of (ordinary) $R$-modules and the homotopy fixed points of the
$G$-action on the category of $R'$-modules. 

Namely, given an $R$-module $M$, we can form the tensor product $M \otimes_R
R'$, which acquires a $G$-action with $G$ acting on the second factor.
Conversely, given an $R'$-module $M'$ with a compatible $G$-action, the
$G$-fixed points $M'^{G}$ define an $R$-module, which is the inverse of the
previous functor. 

This equivalence persists at the level of derived $\infty$-categories, with homotopy
fixed points replacing fixed points. 
\end{example} 

\subsection{Coarse moduli spaces}
It is crucial for our purposes to give criteria for when an algebraic
(Artin) stack has
finite cohomological dimension. In a quasi-compact and separated setting, every
scheme and even every algebraic space has finite cohomological dimension. The
best approximation of an algebraic stack by an algebraic space is the coarse
moduli space. Later in this subsection we will define the notion of tameness,
which allows us to conclude that an algebraic stack already has finite
cohomological dimension, by relating it to its coarse moduli space. Throughout the subsection we choose implicitly a base scheme $S$. 

Recall first that a \emph{coarse moduli space} of an algebraic stack $X$ is an algebraic space $Y$ together with a map $f\colon 
X\to Y$ which
\begin{enumerate}
 \item is initial among all maps from $X$ to algebraic spaces, and
 \item induces a bijection $\pi_0 X(\Spec k) \to \pi_0 Y(\Spec k)$ for every algebraically closed field $k$, where $\pi_0$ denotes the set of isomorphism classes.
\end{enumerate}

The following was first proven by Keel and Mori (\cite{Keel-Mori}) and reformulated by Brian Conrad in \cite{Con05}.

\begin{theorem}\label{Coarse}Let $X$ be an algebraic stack locally of finite presentation over a
base scheme $S$ with finite inertia stack $\pi\colon X \times_{X\times_S X} X\to X$.
Then $X$ has a coarse moduli space $f\colon X \to Y$. The algebraic space $Y$ is
separated if $X$ is separated, and the map $f$ is proper and quasi-finite.
Moreover, the formation of coarse moduli spaces commutes with flat base change.  \end{theorem}

For example, every locally noetherian, separated Deligne--Mumford stack has
finite inertia. Indeed, by \cite[Lemme 4.2]{L-M00}, we know that the diagonal
$X \to X\times_S X$ of every Deligne--Mumford stack is quasi-finite. Since  $X$
is separated, the diagonal is proper; hence, the diagonal is finite. 

In the following, we will always assume implicitly that our algebraic stacks are locally of finite presentation over a base scheme $S$ and have finite inertia. 

A very convenient class of algebraic stacks is given by the so-called \textit{tame stacks} as studied in \cite{AOV08}. 

\begin{definition}An algebraic stack $X$ is called \textit{tame} if the map $f\colon X \to Y$ to its coarse moduli space induces an exact functor $f_*\colon \qcoh(X) \to \qcoh(Y)$.\end{definition}

The question remains how to decide whether a stack is tame. This was completely answered in \cite{AOV08}. We begin with a few preliminary definitions and propositions.
In the following, we ignore the notation from the introduction and, for
an ordinary stack $X$, we write $\qcoh(X)$ for the ordinary (abelian) category of
quasi-coherent sheaves on $X$. If $G$ is a group scheme over $X$, we let
$\qcoh^G(X)$ be the ordinary category of $G$-representations in $\qcoh(X)$. 

\begin{definition}A group scheme $G\to S$ is \textit{linearly reductive} if the
functor $\qcoh^G(S) \to \qcoh(S)$, $F\mapsto F^G$, sending an equivariant sheaf to its fixed points, is exact. Note that this is
equivalent to the tameness of the stack quotient $S/G$ as $\qcoh^G(S) \simeq \qcoh(S/G)$.\end{definition}

Recall that the datum of an affine group scheme $G$ over $\Spec R$ is
equivalent to that of a commutative Hopf algebra $\Gamma$ over $R$. The group
scheme $G$ is commutative if and only if $\Gamma$ is cocommutative. For
example, given a (discrete) abelian group $G'$, we can form the group algebra $R[G']$. The corresponding group scheme is called \textit{diagonalizable}. Examples include $\mu_n$ (with $G'=\Z/n\Z$). 
\begin{theorem}[{\cite[Theorem 2.19]{AOV08}}]Let $G\to S$ be a finite, flat group scheme of finite presentation. Then $G$ is linearly reductive if and only if fpqc locally, we can write $G$ as a semidirect product $\Delta \rtimes H$, where $\Delta$ is diagonalizable and $H$ is constant of an order prime to all residue characteristics of $S$. \end{theorem}

\begin{theorem}[{\cite[Theorem 3.2]{AOV08}} ]An algebraic stack $X$ is tame if and only if for every geometric point $\Spec k \to X$ and object $\xi\in X(\Spec k)$, the automorphism group scheme $\underline{\Aut}_k(\xi)$ is linearly reductive over $\Spec k$. Here, $\underline{\Aut}_k(\xi)$ is defined to be the scheme equivalent to the pullback of the inertia stack $X\times_{X\times_S X} X$ along the map $\Spec k \to X$ classifying $\xi$ and the group structure is induced by the diagonal $X \to X\times_S X$. \end{theorem}
\begin{remark}Let $q: Y \to \Spec k$ be a $k$-scheme and $q^*\xi\in X(Y)$ be the pullback of $\xi$. Then $\underline{\Aut}_k(\xi)(Y)$ is isomorphic to the automorphism group of $q^*\xi$ in the groupoid $X(Y)$. 

Indeed, a morphism $Y \to  \spec k \times_{X\times_S X} X$ over $\spec k$ consists of the choice of $\eta \in X(Y)$ together with two isomorphisms $f: q^*\xi \to \eta$ and $g: q^*\xi \to \eta$ in $X(Y)$ agreeing in $S(Y)$. A morphism between $(\eta, f,g)$ and $(\eta',f',g')$ consists of $i:\eta \to \eta'$ such that $f'=if$ and $g' = ig$. Thus, every morphism $Y \to  \spec k \times_{X\times_S X} X$ over $\spec k$  is isomorphic to a unique $(q^*\xi, f: q^*\xi \to q^*\xi, \id: q^*\xi \to q^*\xi)$.
\end{remark}

Suppose $X$ is Deligne-Mumford, so that the automorphism group schemes of geometric points are
\'etale and thus constant. Then $X$ is tame if and only if the orders of the automorphism groups
 (at geometric points) are invertible on $X$. 

\begin{example}\label{Char0Tame}Let $X$ be a Deligne--Mumford stack over a field of
characteristic zero.  Then $X$ is tame. \end{example}

\begin{example}If $U$ is a scheme over a base scheme $S$ and $G$ a finite group
acting on $U$ such that $|G|$ is invertible on $S$,
the quotient stack $U/G$ is tame. For example, one can take the moduli stack of elliptic
curves over $\Z[\frac16]$. Indeed, $\mell[1/6] \simeq \mell(3)[\frac
16]/GL_2(\mathbb{F}_3)$, where $\mell(3)$ is the moduli \textit{scheme} of
elliptic curves with level $3$ structures.  \end{example}

Recall for the next proposition that a stack $X$ is called \textit{quasi-compact} if for every collection $\{f_i: U_i\to X\}_{i\in I}$ of open immersions such that $\coprod_{i\in I}f_i: U_i \to X$ is surjective, there exists a finite subset $J\subset I$ such that $\coprod_{j\in J}f_j: U_j \to X$ is still surjective. We will see in the next subsection that $X$ is quasi-compact if and only if its coarse moduli space is. 

\begin{proposition}\label{tbcd}Let $X$ be a tame algebraic stack that is quasi-compact and separated. Then there is a natural number $n$ such that $H^i(X; \FF) = 0$ for all $i>n$ and all quasi-coherent $\mathcal{O}_X$-modules $\FF$.\end{proposition}
\begin{proof}In the case of $X$ an algebraic space, this is \cite[072B]{stacks-project}.

In the general case, denote by $f\colon X\to Y$ the map to the coarse moduli space and let $\FF$ be a quasi-coherent sheaf on $X$. Then we have a Leray spectral sequence 
 \[H^p(Y; R^qf_*(\FF)) \Rightarrow H^{p+q}(X;\FF)\]
with $R^qf_*(\FF) = 0$ for $q>0$. The result follows as $Y$ is a quasi-compact
and separated algebraic space and thus has finite cohomological dimension.
\end{proof}

Now we want to introduce a relative version of tameness. For this, we will no longer assume our stacks to have finite inertia. 

\begin{definition}\label{DefinitionTame}Let $f\colon X \to Y$ be a morphism of stacks. We
call $f$ \textit{tame} if for every geometric point $\xi \colon \spec k \to X$, the kernel of the induced map $\underline{\Aut}_k(\xi) \to \underline{\Aut}_k(f(\xi))$ is finite and linearly reductive over $\spec k$.\end{definition}

If $X$ is Deligne-Mumford, this is equivalent to assuming that for every $\xi\in X(k)$ for an algebraically closed field $k$, the kernel
of the map $\Aut_{X(k)}(\xi) \to \Aut_{Y(k)}(f(\xi))$ has
order coprime to the characteristic of $k$. Indeed, the automorphism group scheme of $\xi$ is a discrete group (scheme). 

Recall for the next proposition that a map $X\to Y$ of stacks is \textit{quasi-compact} if for every map $\spec A \to Y$, the stack $X\times_Y \spec A$ is quasi-compact.

\begin{proposition}\label{tamebcd}Let $f\colon X \to Y$ be a quasi-compact, separated and tame morphism of stacks, where we assume $X$ to be algebraic, but for $Y$ only that the diagonal $Y\to Y\times_S Y$ is representable by an algebraic stack. Then $\spec A \times_Y X$ has finite cohomological dimension for every map $q\colon \spec A \to Y$.\end{proposition}
We do not assume that $Y$ is an an Artin stack because the moduli stack of formal groups $M_{FG}$ is not an Artin stack. But it has still representable (even affine) diagonal. 

\begin{proof}Let $q\colon \spec A \to Y$ be a morphism. We have to show that $Z
= X \times_Y \spec A$ has finite cohomological dimension. First note that
$X\times_Y \spec A \cong (X\times_S \spec A) \times_{Y\times_S Y} Y$ is an algebraic stack. Let now $\xi\colon \spec
R \to Z$ be an $R$-valued point. This corresponds to a point $\xi_X\colon \spec
R \to X$, a point $\xi_A\colon \spec R \to \spec A$ and an isomorphism
$\phi\colon f(\xi_X) \to q(\xi_A)$. An automorphism of this is an automorphism
$\psi$ of $\xi_X$ such that $\phi\circ f(\psi) = \phi$. This is equivalent to
$\psi$ being in the kernel of $\underline{\Aut}(\xi)(R) \to
\underline{\Aut}(f(\xi))(R)$. In particular, it follows that $Z$ has
quasi-finite inertia and all automorphism group schemes of geometric points of
$Z$ are kernels of $\underline{\Aut}_k(\xi) \to \underline{\Aut}_k(f(\xi))$ for
geometric points $\xi$ of $X$. Since $Z$ is separated, $Z$ has actually finite inertia. Thus, $Z$ is tame. 
 By \Cref{tbcd} it follows that $Z$ has also finite cohomological dimension. 
\end{proof}

\subsection{The Zariski topology}
In this subsection, we will define and investigate the Zariski site of an
algebraic stack. This is important because certain properties only allow
Zariski descent and not \'etale or fpqc descent as we will see in Section \ref{ZarDescent}. 

\begin{definition}Let $X$ be an algebraic stack. The \textit{Zariski site} of $X$ is given by all open immersions into $X$ and open immersions between them, where a covering is a jointly surjective map. Recall here that a map $Y\to X$ is an open immersion if it is representable and for every map $Z\to X$ from a scheme, the map $Y\times_X Z \to Z$ is an open immersion.\end{definition}
We define an algebraic stack $X$ to be \emph{quasi-compact} if every Zariski cover of $X$ has a finite subcover. 

The Zariski site of an algebraic stack $X$ is actually always equivalent to the site
of open subsets of the \textit{underlying space} $|X|$ of $X$. We will first define the space $|X|$, following \cite{L-M00}, and then prove this equivalence.

The points of $|X|$ are equivalence classes of objects in the groupoids $X(\spec k)$ for $k$ a field. Two such objects $x_1 \in X(\spec k_1)$ and $x_2\in X(\spec k_2)$ are equivalent if there is a common field extension $K$ of $k_1$ and $k_2$ such that $(x_1)_K$ and $(x_2)_K$ are isomorphic in $X(\spec K)$. The open subsets of $|X|$ are those of the form $|U|$ for an open substack $U$ of $X$. Recall that substack means in particular that $U(Z)$ is a full subcategory of $X(Z)$ for every scheme $Z$. The construction $X \mapsto |X|$ is functorial (see \cite[Section 5]{L-M00} for details). 

If $X$ satisfies the conditions of Theorem \ref{Coarse}, the map $X\to Y$ in its coarse moduli space induces a homeomorphism $|X| \to |Y|$. Indeed, it is clearly a continuous bijection and it is also closed as $X\to Y$ is proper.

\begin{lemma}\label{PresLemma}Let $X$ be an algebraic stack and $U \to X$ be a
presentation, i.e., a smooth surjective map from an algebraic space. Then $|X|$ is the coequalizer of the maps $\pr_1,\pr_2\colon |U|\times_{|X|} |U| \to |U|$.\end{lemma}
\begin{proof}By \cite[Remarque 5.3]{L-M00}, the map $|U| \to |X|$ is surjective. By \cite[Corollary 14.34]{G-W10} and \cite[Proposition 5.6]{L-M00}, this map is also open. 
By the definition of the pullback, the images of two points 
\[y_1,y_2 \in U(\spec k)\]
 are isomorphic in $X(\spec k)$ if and only if there are isomorphic 
\[z_1, z_2 \in (U\times_X U)(\spec k)\]
 with $\pr_1(z_1) = y_1$ and $\pr_2(z_2) = y_2$. Thus, the images of $y_1$ and $y_2$ in $|X|$ are equal iff there is a point $z$ in $|U\times_X U|$ with $|\pr_1|(z) = y_1$ and $|\pr_2|(z) = y_2$. Thus, $|X|$ is the coequalizer of the two maps 
\[|\pr_1|,|\pr_2|\colon |U\times_{X} U| \to |U|.\] 
As the map 
\[|U\times_X U| \to |U|\times_{|X|} |U|\]
 is surjective by \cite[Proposition 5.4(iv)]{L-M00}, the result follows. \end{proof}

\begin{proposition}\label{bijection}The functor $U \mapsto |U|$ defines an order-preserving bijection between open substacks of $X$ and open subsets of $|X|$.\end{proposition}
\begin{proof}Let $W \subset |X|$ be an open subset. Choose a presentation $f\colon Y \to X$ with $Y$ an algebraic space and $f$ smooth and surjective. The preimage $|f|^{-1}(W)$ is open in $|Y|$. By \cite[03BZ]{stacks-project}, there is a unique open algebraic subspace $V$ of $|Y|$ with $|V| = |f|^{-1}W$. If 
\[\pr_1,\pr_2\colon Y\times_X Y \to Y\]
 denote the two projections, we get likewise an open algebraic subspace $V'$ of $Y\times_X Y$ corresponding to 
\[(|\pr_1|\circ |f|)^{-1}(W) = (|\pr_2|\circ |f|)^{-1}(W).\]
 By stackifying the groupoid defined by 
\[(\pr_1)|_{V'},(\pr_2)|_{V'}\colon V' \to V,\]
 we get an open substack $U$ of $X$. By the last lemma, we have $|U| = W$. 

On the other hand, if $U$ is an open substack of $X$, then the open substacks associated with $|f|^{-1}|U|$ and $|f\circ \pr_1|^{-1}|U|$ agree with $U\times_X Y$ and $U\times_X Y\times_X Y$ by \cite[03BZ]{stacks-project}. As $U$ equals the substack of $X$ associated with 
\[\pr_1,\pr_2\colon U\times_X Y\times_X Y \to U\times_X Y,\]
 the proposition follows.\end{proof}

\begin{corollary}\label{Zariski-Equivalence}The Zariski topology on an
algebraic stack $X$ is equivalent to the site of open subsets of $|X|$. In particular, $X$ is quasi-compact if and only if $|X|$ is. \end{corollary}
\begin{proof}Note first that every open immersion into $X$ is equivalent over $X$ to an open substack by considering its image. 
 Thus, we only have to show that $\{U_i \to X\}$ is a covering by open substacks if and only if $\{|U_i| \to |X|\}$ is an open covering by open subsets. This follows directly from the fact that one can test the surjectivity of a map between algebraic stacks on the underlying topological spaces by \cite[Proposition 5.4(ii)]{L-M00}.  
\end{proof}

As a last point, we want to discuss non-vanishing loci.

\begin{proposition}Let $X$ be an algebraic stack and $\LL$ a line bundle on
$X$. Let $f\in\Gamma(X, \LL)$. Then:
\begin{itemize}
\item[(a)]Let $x_1\colon \spec k_1 \to X$ and $x_2\colon \spec k_2 \to X$ be two morphisms for $k_1$ and $k_2$ fields. If $x_1$ and $x_2$ define the same point in $|X|$, then $(x_1)^*f = 0$ if and only if $(x_2)^*f = 0$.
\item[(b)]The locus of points in $|X|$ where $f$ does not vanish is open.
\end{itemize}
\end{proposition}
\begin{proof}As field extension are always faithfully flat, part (a) follows easily. 

For part (b), consider a presentation $q\colon Y \to X$ such that $q^*\LL$ is trivial and $Y$ is a scheme. As discussed in the proof of Lemma \ref{PresLemma}, the map $|q|\colon |Y| \to |X|$ is open and surjective. Furthermore, $x^*(q^*f) = 0$ if and only if $(q\circ x)^*f= 0$ for $x\colon \spec k \to Y$ a point. Thus, we can assume that $\LL= \mathcal{O}_X$ and that $X$ is a scheme. The result is well-known in this case. 
\end{proof}

\begin{definition}\label{non-vanishing}Let $X$ be an algebraic stack and $\LL$
a line bundle on $X$. Let $f\in \Gamma(X, \LL)$. Then we define the
non-vanishing locus $D(f)$ of $f$ to be the open substack of $X$ corresponding
to the non-vanishing locus of $f$ on $|X|$ by \Cref{bijection}.\end{definition}

The following property will later be freely used: 

\begin{proposition}Let $q\colon Y \to X$ be a map of algebraic stacks. Let
furthermore, $\LL$ be a line bundle on $X$ and $f\in \Gamma(X, \LL)$. Then there is a
natural equivalence $D(q^*f) \simeq D(f) \times_X Y$ where $q^*f \in \Gamma(Y,
q^* \LL)$ is the pullback.\end{proposition}
\begin{proof} The map $\pr_2\colon D(f) \times_X Y \to Y$ is an open immersion.
The image of $|\pr_2|$ in $|Y|$ agrees with those points $y$ in $|Y|$ such that
$|q|(y) \in |D(f)|$. This agrees with $|D(q^*f)|$. Thus, $D(q^*f)$ agrees with the image of $D(f) \times_X Y \to Y$ in $Y$. \end{proof}

\section{Abstract affineness results}\label{Section3}

The aim of this section is to give criteria when the global sections functor establishes an equivalence
between quasi-coherent sheaves and modules over the ring of global sections for
a derived stack. We call such derived stacks \textit{0-affine} (following
\cite{gaits}):

\begin{definition}A derived stack $\XX = (X, \otop)$ is called \textit{0-affine} if the global sections functor
\[\Gamma\colon \qcoh(\XX) \to \md(\Gamma( \mathfrak{X}, \otop))\]
is an equivalence of symmetric monoidal $\infty$-categories.\end{definition}
In particular, $\Gamma$ is a symmetric monoidal functor (and not only a lax
symmetric monoidal functor) if $(\XX,\otop)$ is $0$-affine. In the next subsection we will show that a derived stack $\XX$ is $0$-affine if and only if
\begin{enumerate}
\item the functor of taking global sections $\Gamma$ commutes with homotopy colimits in
$\qcoh(\mathfrak{X} )$ and 
\item the functor $\Gamma$ is conservative, i.e., whenever $\mathcal{F} \in
\qcoh(\mathfrak{X})$ is such that $\Gamma(\mathfrak{X},
\mathcal{F}) $ is contractible, then $\mathcal{F}$ is itself contractible. 
\end{enumerate} 

In the following two subsections we show that $0$-affineness descends under certain (topological) finiteness conditions fpqc-locally and under certain ampleness conditions Zariski-locally. In the last subsection, we will show that $0$-affineness ascends under affine morphisms and open immersions.

The theorems in this section are abstract in the sense that they are not special to chromatic homotopy theory. The main known techniques to guarantee the finiteness assumptions of these abstract theorems, though, will use powerful nilpotence technology in chromatic homotopy theory. This will be the topic of the following section. 

One of the main issues can already be illustrated by the following example: If $2$ is not inverted, the functor
\[ E \to E^{h \mathbb{Z}/2},  \quad \mathrm{Fun}(B \mathbb{Z}/2, \sp) \to \sp, \]
from spectra with a $\mathbb{Z}/2$-action to spectra, fails to commute with
homotopy 
colimits, or equivalently fails to send wedges to wedges
(\Cref{ex:failswedges} below). The homotopy groups of $E^{h \mathbb{Z}/2}$ are the abutment of a
homotopy fixed point spectral sequence each of whose terms (the group
cohomology of $\pi_* E$) sends wedges in $E$ to direct sums. However, the
potential infiniteness of the spectral sequence (in particular, the infinitude
of the filtration) does not allow us to conclude that the abutment sends wedges to wedges. In chromatic homotopy theory, however, it is possible to show that
the analogous filtrations are finite under certain conditions, as we will see in the next section.

\begin{example} 
\label{ex:failswedges}
We provide a simple illustration of the fact that the functor 
\[ E \to E^{h \mathbb{Z}/2},  \quad \mathrm{Fun}(B \mathbb{Z}/2, \sp) \to \sp\]
 fails
to commute with wedges. Consider the spectrum $X = \bigvee_{n \in \mathbb{Z}} H
\mathbb{Z}/2[n]$, and give it the trivial $\mathbb{Z}/2$-action. 
Then we claim that 
$$X^{h\mathbb{Z}/2} = \mathrm{F}(B\mathbb{Z}/2, X) \not\simeq \bigvee_{n \in
\mathbb{Z}} \mathrm{F}(B\mathbb{Z}/2, H \mathbb{Z}/2[n]).
$$
In fact, this follows from the fact that 
$\pi_* \mathrm{F}(B\mathbb{Z}/2, X)$ is an uncountable abelian group. 
We can write
\[ \mathrm{F}(B\mathbb{Z}/2, X) = \mathrm{F}(\mathbb{RP}^\infty, X) \simeq
\varprojlim_m \mathrm{F}(\mathbb{RP}^m, X) ,  \]
and $\pi_* \mathrm{F}(\mathbb{RP}^m, X) \simeq H^*(\mathbb{RP}^m;
\mathbb{Z}/2) \otimes_{\mathbb{Z}/2} \pi_*(X)$. As $m \to \infty$, the Milnor
exact sequence shows that 
$\pi_* \mathrm{F}(B\mathbb{Z}/2, X) \simeq \varprojlim_m H^*(\mathbb{RP}^m;
\mathbb{Z}/2) \otimes_{\mathbb{Z}/2} \pi_*(X)
$
is actually uncountable.
If we regard $X$ as a ring spectrum such that $\pi_*(X) \simeq
\mathbb{F}_2[u^{\pm 1}]$ with $|u| = 1$, then $\pi_* \left(X^{B\mathbb{Z}/2}\right) \simeq
\mathbb{F}_2[u^{\pm 1}]\llbracket v\rrbracket$ with $|v| = 0$, while 
$\pi_* \left( \bigvee_{n \in
\mathbb{Z}} \mathrm{F}(B\mathbb{Z}/2, H \mathbb{Z}/2[n])\right)
$ gives only the polynomial subring $\mathbb{F}_2[u^{\pm 1}][v]$. 
\end{example}

\subsection{Schwede--Shipley theory}

Let $\mathcal{A}$ be an abelian category with all colimits. 
If $\mathcal{A} = \md(R)$ is the category of (discrete) modules over a (not necessarily
commutative) ring $R$, then $\mathcal{A}$ has a \emph{compact, projective
generator}: that is, $R$ itself. 
More precisely, the functor $\hom_{\mathcal{A}}(R, \cdot)\colon \mathcal{A}
\to\mathbf{Ab}$ (which assigns to a
module its underlying abelian group) commutes with all colimits, and is conservative. 

It is a basic principle that 
module categories are characterized precisely by this property: that is, an
abelian category is equivalent to a category of modules precisely when it has a
compact, projective generator. 
This point of view explains the classical Morita theorem that describes
equivalences between categories of modules: they arise from compact, projective
generators. 

In the derived setting, the objects of study are not abelian categories, but
presentable, stable $\infty$-categories, and the question one asks is when such
an $\infty$-category is the $\infty$-category of modules over an $A_\infty$-ring. An answer in the
language of stable model categories was given by Schwede and Shipley in
\cite{schwedeshipley}; a reformulation of the statement in terms of $\infty$-categories is in
\cite[Theorem 8.1.2.1]{higheralg}. 

\begin{theorem}[Schwede-Shipley; Lurie]
A presentable, stable $\infty$-category $\mathcal{C}$ is equivalent to the
$\infty$-category of modules over an $A_\infty$-ring if and only if it has a compact 
generator $X \in \mathcal{C}$: that is, $X $ is such that
$\hom_{\mathcal{C}}(X, \cdot)\colon \mathcal{C} \to \sp$ commutes with
filtered homotopy colimits and sends nonzero objects to noncontractible spectra.  
\end{theorem}

\begin{example}[Be\u{i}linson \cite{Beilinson}] The derived category of quasi-coherent sheaves on
projective space $\mathbb{P}^n$ is equivalent to the derived category of modules over
the (discrete) ring $\mathrm{End}_{\qcoh(\mathbb{P}^n)}(
\mathcal{O}_{\mathbb{P}^n}\oplus \dots \oplus \mathcal{O}_{\mathbb{P}^n}(n))$.
Namely, Be\u{i}linson shows that 
$\mathcal{O}_{\mathbb{P}^n}\oplus \dots \oplus \mathcal{O}_{\mathbb{P}^n}(n)$
is a compact generator for the derived category of coherent sheaves on
$\mathbb{P}^n$. \end{example} 

We will also need a version in the symmetric monoidal case. When $R$ is an
$E_\infty$-ring, the $\infty$-category $\md(R)$ is symmetric monoidal, and it
has the property that the unit object (i.e., $R$ itself) is a compact generator. 
This is essentially the distinguishing feature of such module categories
according to the next result. 

\begin{theorem}[{\cite[Proposition 8.1.2.7]{higheralg}}] \label{monss} Let $(\mathcal{C}, \otimes,
\mathbf{1})$ be a presentable stable,
symmetric monoidal $\infty$-category where the tensor product preserves
homotopy 
colimits. The endomorphism ring $R = \mathrm{End}( \mathbf{1})$
has a canonical structure of an $E_\infty$-ring.   If the unit object $\mathbf{1} \in \mathcal{C}$ is a compact
generator, one has a symmetric monoidal equivalence
\[ \mathcal{C} \simeq \md(R), \quad X \mapsto \hom_{\mathcal{C}}( \mathbf{1},
X),  \]
between $\mathcal{C}$ and the category of $R$-modules. 
\end{theorem} 

\begin{example} 
We asserted in \thref{BGa} that the derived category of quasi-coherent sheaves
on the stack $B \mathbb{G}_a$ (over the base field $\mathbb{Q}$),
or equivalently the derived category of $\mathbb{G}_a$-representations, was
equivalent to $\md( \mathbb{Q}[x_{-1}])$ via an adjunction of symmetric
monoidal $\infty$-categories.  This is equivalent to the assertion that the structure sheaf itself,
which corresponds to the trivial one-dimensional representation of
$\mathbb{G}_a$, is  a compact generator, and this in turn is closely related
to the unipotence of $\mathbb{G}_a$. 
\end{example} 

Our strategy will be to apply the Schwede--Shipley theorem to the
$\infty$-category of
quasi-coherent sheaves on a derived stack. More precisely, we will use the following corollary:
\begin{corollary}\label{SSCor}A derived stack $\XX =(X,\otop_\XX)$ is $0$-affine if and only if the global sections functor \[\Gamma\colon \qcoh(\XX) \to \md(\gx) \]
commutes with homotopy colimits and is conservative. Here, conservative means that for $\FF \in \qcoh(\XX,\otop_\XX)$ a quasi-coherent $\otop_\XX$-module, $\Gamma(\FF) = 0$ already implies $\FF = 0$.\end{corollary}
\begin{proof}The global sections functor $\Gamma$ is corepresented by $\otop_\XX$. Thus, $\Gamma$ commutes with filtered homotopy colimits and is conservative if and only if $\otop_\XX$ is a compact generator of $\qcoh(\XX,\otop_\XX)$. By Theorem \ref{monss}, the result follows. Note here that if $\Gamma$ is an equivalence, it commutes automatically with all homotopy colimits.
\end{proof}

\subsection{fpqc-descent for $0$-affineness}

In this section, we describe a basic technique for showing that certain
homotopy limits (given by global sections functors) commute with homotopy
colimits. The strategy is to first verify that this holds after smashing with
something that generates the original category as a thick tensor-ideal; then it
is possible to apply ``descent.'' 
We note that the idea of descent via thick tensor-ideals has been explored further in
\cite{balmersep, galoischromatic}.

Let us first recall the definition of a thick tensor-ideal.

\begin{definition}
Given an $E_\infty$-ring $R$, a \textit{thick tensor-ideal} of $\md(R)$ is a
full subcategory $\CC \subset \md(R)$, containing the zero object, such that:
\begin{itemize}
\item The fiber and cofiber of every morphism $M \to N$ in $\CC$ is in $\CC$ again.
\item If $X\oplus Y$ is in $\CC$, then $X \in \mathcal{C}$ and $Y \in
\mathcal{C}$.
\item If $X\in\CC$ and $Y\in\md(R)$ is arbitrary, then $X\otimes_R Y \in \CC$.
\end{itemize}
We say that an $R$-module $M$ \textit{generates} $\CC$ as a thick tensor-ideal if $\CC$ is the smallest thick tensor-ideal of $\md(R)$ containing $M$. 
\end{definition}

\begin{proposition}\label{AbstractCommHolim}
Let $\XX = (X, \otop)$ be a derived stack whose underlying stack $X$ is a
quasi-compact, separated Deligne-Mumford stack. Suppose: 
\begin{enumerate}
\item There is a flat, affine morphism $q\colon Y \to X$ from an algebraic stack of finite cohomological dimension.
\item There is a $\gx$-module $M$ that generates $\md(\gx)$ as a thick tensor-ideal such that we have an isomorphism 
\[\pi_*\left( \otop_\XX\otimes_{\gx} M\right) \cong q_*q^*\pi_*(\otop_\XX)\]
of $\pi_*\otop_\XX$-modules.
\end{enumerate}
Then the global sections functor $\Gamma$ from quasi-coherent $\otop_\XX$-modules to $\gx$-modules commutes with homotopy colimits. \end{proposition}
\begin{proof}We start by showing that the functor
\[ \mathcal{F} \mapsto \Gamma( \mathfrak{X}, \mathcal{F} \otimes_{\gx} M ), \quad
\qcoh(\mathfrak{X}) \to \md( \Gamma(\mathfrak{X}, \otop)),\]
commutes with homotopy colimits in $\mathcal{F}$. Since the functor is an
exact functor between stable $\infty$-categories, it suffices to show that it
commutes with arbitrary direct sums (i.e., wedges). 

We will prove that the $E^2$-term of the descent spectral sequence 
\[  H^i( X, \pi_j( \FF \otimes_{\gx} M)) \implies \pi_{j-i} \Gamma(
\mathfrak{X}, \mathcal{F} \otimes_{\gx} M),  \]
is concentrated in finitely many rows. 
In fact, by the isomorphism 
\[\pi_*(\otop_\XX\otimes_{\gx} M) \cong q_*q^*\pi_*(\otop_\XX)\]
 we know that $\pi_*\left( \otop_\XX\otimes_{\gx} M\right)$ is flat as a
 $\pi_*\otop_\XX$-module, since $q$ is flat and affine. Note further that by the projection formula, the morphism 
\[\pi_k\otop_\XX \tensor_{\mathcal{O}_X} q_*q^*\mathcal{O}_X \to q_*q^*\pi_k\otop_\XX\]
 is an isomorphism as $q$ is affine. Thus, we have the following isomorphisms of $\mathcal{O}_X$-modules:
\begin{align*}\pi_j( \mathcal{F} \otimes_{\gx} M) &\cong (\pi_*\FF\tensor_{\pi_*\otop_\XX} \pi_*(\otop_\XX \otimes_{\gx} M))_j\\
 &\cong (\pi_*\FF\tensor_{\pi_*\otop_\XX} q_*q^*\pi_*(\otop_\XX))_j\\
&\cong  (\pi_*\FF\tensor_{\pi_*\otop_\XX} \pi_*\otop_\XX \tensor_{\mathcal{O}_X} q_*q^*\mathcal{O}_X)_j \\
&\cong \pi_j\FF\tensor_{\mathcal{O}_X} q_*q^*\mathcal{O}_X.
\end{align*}

As $q$ is affine, the projection formula allows us to rewrite this as 
\[ q_*(q^*\pi_j\FF \tensor_{\mathcal{O}_Y} \mathcal{O}_Y) \cong q_*q^*\pi_j\FF.\]
Thus, we have a Leray spectral sequence
\[ H^l(Y, (R^mq_*)q^*\pi_j\FF) \Rightarrow H^{l+m}(X, \pi_j( \FF \otimes_{\gx} M)).  \]
By (1), the $E^2$-term of this spectral sequence is concentrated in finitely many columns (bounded by the cohomological dimension of $Y$) and in the $0$-row; hence, we see that $H^i( X, \pi_j( \FF \otimes_{\gx} M))$ is zero for large $i$. 

Since the $E^2$ page of the spectral sequence for $\pi_*\Gamma( \mathfrak{X}, \FF \tensor_{\gx} M)$ commutes with
direct sums (as $X$ is quasi-compact and separated; see
\Cref{lem:cohcommute} below), it follows thus that these homotopy groups themselves commute with
direct sums in $\FF$. Indeed, for a collection $(\FF_i)_{i\in I}$ of quasi-coherent $\otop_\XX$-modules, the natural map
\[ \bigoplus_{i\in I}\Gamma(\XX,\FF_i) \to \Gamma(\XX, \bigoplus_{i\in I}\FF_i)\]
induces an isomorphism on the $E^2$-terms of the corresponding descent spectral sequences, thus on the $E^\infty$-terms and because of the finiteness of the filtration also on the abutment.

Let us consider now the collection $\mathcal{C}$ of all $\gx$-modules $T$ such that the
functor
$$\mathcal{F} \mapsto \Gamma(\mathfrak{X}, \mathcal{F}\tensor_{\gx} T), \quad
\qcoh(\mathfrak{X}) \to \sp,$$
commutes with homotopy colimits. As we have just seen, $M \in \mathcal{C}$. Since the composition of homotopy 
colimit-preserving functors is homotopy colimit-preserving, it follows that
$\mathcal{C}$ is an \emph{ideal}: If $T \in \mathcal{C}$ and $T'$ is any
$\gx$-module, then $T \otimes_{\gx} T' \in \mathcal{C}$. Moreover, $\mathcal{C}$ is a
stable subcategory of $\sp$, and $\mathcal{C}$ is closed under retracts. (A
retract of a functor that preserves homotopy colimits preserves homotopy
colimits.) As $M$ generates $\gx$-modules as a thick tensor-ideal, we see that $\mathcal{C}$ consists of all of $\gx$-modules; in particular, $\gx\in \mathcal{C}$ and $\Gamma$ commutes with homotopy colimits. 
\end{proof}

In the last proof, we used the following algebraic lemma stating that cohomology commutes with filtered colimits. 
\begin{lemma}\label{lem:cohcommute}
Let $X$ be a quasi-compact, separated stack. Then the cohomology group functors
$H^i(X, \cdot)$ on the category of quasi-coherent sheaves on $X$ commute with filtered colimits. 
\end{lemma}
\begin{proof} 
Choose an affine, flat cover $\spec A \to X$. The iterated fiber products
$\spec A \times_X \spec A, \dots$ are all affine schemes, so the cohomology of
$\mathcal{F}$ is the cohomology of the cochain complex associated to the
cosimplicial abelian group
\[ \mathcal{F}( \spec A) \rightrightarrows \mathcal{F}( \spec A \times_{X}
\spec A) \triplearrows \dots,  \]
i.e., the Cech construction. But this clearly commutes with filtered colimits
in $\mathcal{F}$. 
\end{proof} 

\Cref{lem:cohcommute} is analogous to the following fact: the homotopy fixed
point functor $$E \mapsto E^{h \mathbb{Z}/2}, \quad \mathrm{F}( B \mathbb{Z}/2, \sp)
\to \sp$$ \emph{does} commute with filtered homotopy colimits if we restrict to the
subcategory of $\mathrm{F}(B \mathbb{Z}/2, \sp)$ whose underlying spectra have
\emph{bounded-above} (by some fixed value) homotopy groups.

The arguments for the conservativity of $\Gamma$ are related but different. We will present an algebraic and a topological analog of the last proposition for this purpose. The latter will turn out to be more powerful, yet is also more subtle regarding its input. But first we state a little lemma:

\begin{lemma}\label{AlreadySheaf}Let $\XX = (X,\otop_\XX)$ be a derived stack and assume that 
\[\Gamma\colon \qcoh(\XX) \to \md(\gx)\]
 commutes with homotopy colimits. Then the natural map 
\[\Gamma(\FF)\tensor_{\gx}N \to \Gamma(\FF \tensor_{\gx} N)\]
is an equivalence for every $\gx$-module $N$ and every quasi-coherent $\otop_\XX$-module $\FF$.\end{lemma}
\begin{proof}This is by definition true for $N= \gx$. As the class of $\gx$-modules for which it is true is closed under homotopy colimits, it is true for every $\gx$-module $N$.
\end{proof}
In particular, we see that the left adjoint to $\Gamma$ (i.e., ``tensoring up'') is fully faithful if $\Gamma$ commutes with homotopy colimits. 

\begin{proposition}
Let $\XX = (X, \otop)$ be a derived stack whose underlying stack $X$ is a
quasi-compact, separated Deligne-Mumford stack. Suppose: 
\begin{enumerate}
\item There is a faithfully flat, affine morphism $q\colon Y \to X$ from a quasi-affine scheme of cohomological dimension $\leq 1$. 
\item There is a $\gx$-module $M$ such that we have an isomorphism
\[\pi_*\left(\otop_\XX\otimes_{\gx} M\right) \cong q_*q^*\pi_*(\otop_\XX).\]
\item The global sections functor $\Gamma\colon \qcoh(\XX,\otop_\XX) \to \md(\gx)$ commutes with homotopy colimits.
\end{enumerate}
Then the global sections functor $\Gamma$ is conservative. \end{proposition}

By \Cref{AbstractCommHolim}, the last condition is satisfied if $M$ generates
$\mod( \Gamma( \mathfrak{X}, \otop_{\mathfrak{X}})$ as a thick tensor-ideal.

\begin{proof}Let $\FF \in \qcoh(\XX,\otop_\XX)$ and assume $\Gamma(\FF) = 0$. We have to show that $\pi_j\FF = 0$ for every $j\in\Z$. 

By the last lemma, $\Gamma(\FF\tensor_{\gx} M) \simeq \Gamma(\FF) \tensor_{\gx} M = 0$. As in the last proof, we have
\[\pi_j( \mathcal{F} \otimes_{\gx} M) \cong q_*q^*\pi_j\FF.\]
Thus, the descent spectral sequence for $\FF\tensor_{\gx} M$ has $E^2$-term isomorphic to
\[H^i(X; q_*q^*\pi_j\FF) \cong H^i(Y; q^*\pi_j\FF).\]
As $Y$ has cohomological dimension $\leq 1$, this spectral sequence degenerates at $E^2$. Since it converges to $0$, it follows that $H^0(Y; q^*\pi_j\FF) = 0$. As $Y$ is quasi-affine, this implies $q^*\pi_j\FF = 0$ by \cite[Proposition 13.80]{G-W10} and thus $\pi_j\FF = 0$ as $q$ is faithfully flat.
\end{proof}

For the next proposition, we need the following definition:
\begin{definition}We call a morphism $f\colon \YY \to \XX$ of derived stacks \textit{quasi-compact} or \textit{separated} if the underlying map of classical algebraic stacks is. We call it \textit{(faithfully) flat} if the map $f_0\colon Y \to X$ of the underlying stacks is and the map $f^*\pi_k\otop_\XX \to \pi_k \otop_\YY$ 
is an isomorphism for every $k\in\Z$. \end{definition}

\begin{lemma}Let $f\colon\YY\to \XX$ be a flat map of derived stacks. Then \[\pi_* f^*\FF \cong f^*\pi_*\FF\] for every quasi-coherent $\otop_\XX$-module $\FF$.\end{lemma}
\begin{proof}Recall that $f^*\FF$ is defined to be $f^{-1}\FF\otimes_{f^{-1}\otop_\XX} \otop_\YY$. As $f$ is flat, 
\(\pi_0\otop_\YY \cong \mathcal{O}_Y\)
 is a flat module over $\pi_0 f^{-1}\otop_\XX \cong f^{-1}\mathcal{O}_X$. By the K\"unneth spectral sequence, it follows that \[\pi_*(f^{-1}\FF\otimes_{f^{-1}\otop_\XX} \otop_\YY) \cong f^{-1}(\pi_*\FF)\otimes_{f^{-1}\mathcal{O}_X} \mathcal{O}_Y = f^*\pi_*\FF.\]
\end{proof}

\begin{lemma}\label{DerivedFaithfulness}Let $f\colon\YY\to \XX$ be a faithfully flat map of derived
stacks. Then the functor $$f^*\colon \qcoh(\XX) \to
\qcoh(\YY)$$ is faithful.\end{lemma}
\begin{proof}Let $\FF$ be a quasi-coherent $\otop_\XX$-module such that $f^*\FF$ is equivalent to the $0$-object. We need to show that $\pi_*\FF = 0$. By the last lemma, we know that $f^*\pi_*\FF \cong \pi_*f^*\FF = 0$. Since $f$ is faithfully flat, the result follows.\end{proof}

\begin{proposition}\label{AbstractConservative}Let $\XX = (X,\otop_\XX)$  be a
derived stack whose underlying stack $X$ is  quasi-compact, separated,
and Deligne-Mumford.
Suppose given a faithfully flat, quasi-compact and separated morphism $q\colon \mathfrak{Y} \to \mathfrak{X}$ from a derived stack $\YY = (Y,\otop_{\mathfrak{Y}})$.

Assume the following: 
\begin{enumerate}
\item The global sections functor $\Gamma\colon \qcoh(\YY) \to \md(\Gamma(\YY,\otop_\YY))$ is conservative. 
\item There is a $\gx$-module $M$ such that we have an equivalence
\[\otop_\XX \tensor_{\gx} M \to q_*\otop_\YY\] of $\otop_\XX$-modules.
\item The global sections functor $\Gamma\colon \qcoh(\XX,\otop_\XX) \to \md(\gx)$ for $\XX$ commutes with homotopy colimits.
\end{enumerate}
Then the global sections functor $\Gamma$ for $\XX$ is conservative. \end{proposition}
\begin{proof}
Let $\FF \in \qcoh(\XX,\otop_\XX)$ and assume $\Gamma(\FF) = 0$. We have to show that $\FF = 0$ or, equivalently, $q^*\FF = 0$ by the last lemma. 

By assumption, we have
\[\FF\tensor_{\gx} M \simeq \FF\tensor_{\otop_\XX}\otop_\XX \tensor_{\gx} M \simeq \FF\tensor_{\otop_\XX} q_*\otop_\YY.\]

By the projection formula (see \cite[Remark 1.3.14]{DAGProp}), this is equivalent to
\[q_*(q^*\FF\tensor_{\otop_\YY} \otop_\YY) \simeq q_*q^*\FF.\]

Thus, we get
\begin{align*}\Gamma(q^*\FF)&\simeq \Gamma(q_*q^*\FF)\\ 
&\simeq \Gamma(\FF\tensor_{\gx} M)\\
&\simeq \Gamma(\FF)\tensor_{\gx} M \\
&\simeq 0\end{align*}

Since the global sections functor on $(\YY,\otop_\YY)$ is conservative, it follows that $q^*\FF = 0$. 
\end{proof}

To apply the last proposition, we need as input a good supply of derived stacks
with conservative global sections functor. In particular, this will turn out to
be the case when the underlying stack $X$ is a quasi-affine scheme. In the
next subsection, we will show that this criterion is \emph{Zariski-local} in
certain cases. 

\subsection{Zariski-descent for $0$-affineness}\label{ZarDescent}
This subsection is concerned with understanding to what extent $0$-affineness can be checked Zariski-locally. In this paper, only the quasi-affine case of Corollary \ref{Scheme0affineness} will be used, but the other criteria are still useful in other situations. 

Recall that for $0$-affineness it is sufficient that the global sections functor commutes with homotopy colimits and is conservative. The former property can (nearly) always be checked Zariski-locally:

\begin{proposition}\label{ZarColim}Let $\XX = (X, \otop_\XX)$ be a derived stack with underlying separated and quasi-compact Deligne--Mumford stack $X$. For $\{U_i \to X\}_{i\in I}$ a finite Zariski covering by open substacks, we get induced derived stacks $\mathfrak{U}_i=(U_i,\otop_{\mathfrak{U}_i})$. Assume that the global sections functors
\[\Gamma\colon \qcoh(\mathfrak{U}_i) \to \md(\Gamma(\mathfrak{U}_i, \otop_{\mathfrak{U}_i}))\]
 for all $\mathfrak{U}_i$ commute with homotopy colimits. Then the global sections functor 
\[\Gamma\colon \qcoh(\XX) \to \md(\gx) \]
for $\XX$ commutes with homotopy colimits as well.\end{proposition}
The idea is that the global sections over $\mathfrak{X}$ are obtained as a
\emph{finite} homotopy limit of the sections over the $U_i$ and their
intersections. It is important
here that a Zariski cover is used. 

\begin{proof}
Observe that pushforward along an open immersion of derived stacks commutes with homotopy colimits by Example 2.5.6 and Proposition 2.5.12 of \cite{DAGQC}. 
It follows easily that the global sections functor for every open substack of
each $\mathfrak{U}_i$ also commutes with homotopy colimits. Observe furthermore, that restriction of a quasi-coherent sheaf to an open substack commutes with arbitrary homotopy colimits. Thus, the functor \[\qcoh(\XX) \to \sp,\qquad \FF \mapsto \FF(U)\]
commutes with arbitrary homotopy colimits for any substack $U$ of some $U_i$. 

Let $\FF$ be an $\otop_\XX$-module. By \Cref{CubeLimit} and \Cref{AlgStackZar}, the canonical map
\[ \Gamma( \mathfrak{X}, \mathcal{F}) \to \holim_{\mathcal{P}_I^{op}} \FF(\mathfrak{C}^{U,c})\]
is an equivalence. Here, $\mathcal{P}_I$ denotes the (finite) poset of non-empty subsets of $I$ and $\mathfrak{C}^{U,c}(S) = \bigcap_{i\in S}U_i$ for a subset $S\subset I$. 

As before, it is enough to show that the global sections functor 
\[\Gamma\colon \qcoh(\XX) \to \md(\gx)\]
 commutes with direct sums. So let $(\FF_j)_{j\in J}$ be a family of quasi-coherent $\otop_\XX$-modules. Consider the commutative diagram
\[\xymatrix{\bigoplus_{j\in J} \Gamma(\FF_j) \ar[r] \ar[d] & \Gamma(\bigoplus_{j\in J}\FF_j) \ar[d] \\
\bigoplus_{j\in J} \holim_{\varnothing \neq S\subset I}\FF_j(\mathfrak{C}^{U,c}(S))  \ar[r] &\holim_{\varnothing \neq S\subset I}\bigoplus_{j\in J}  \FF_j( \mathfrak{C}^{U,c}(S) ) }\]
As just discussed, the vertical arrows are equivalences. Moreover, the lower
horizontal arrow is an equivalence since finite homotopy limits  commute
with arbitrary homotopy colimits (in a stable $\infty$-category). Thus, the upper horizontal arrow is an equivalence as well.
\end{proof}

Now we turn to the conservativeness of $\Gamma$. This will will depend on the notion of an ample line bundle, which in turn, depends on the notion of non-vanishing loci as in Definition \ref{non-vanishing}. 

\begin{definition}Let $X$ be a quasi-compact and separated Deligne--Mumford stack with coarse
moduli space $f\colon X \to Y$. We call then a line bundle $\LL$ on $X$
\textit{ample} if $Y$ is a scheme and the non-vanishing loci $D(x)$ of sections $x \in
\Gamma(X, \LL^{\tensor k})$ form a basis of the \textit{Zariski} topology of $X$. This agrees with the usual definition if $X$ is a scheme by \cite[Proposition 13.47]{G-W10}. 

Let $\XX = (X, \otop_\XX)$ be a derived stack. Let $\LL$ be a
locally free $\otop_\XX$-module of rank $1$. We say that $\LL$ is
\textit{ample} if the non-vanishing loci $D(\overline{x})$ of the reductions
$\overline{x} \in \Gamma(\pi_k(\LL^{\tensor l}))$ of elements $x\in
\pi_k\Gamma(\mathfrak{X}, \LL^{\tensor l})$ form a basis of the Zariski topology of $X$.\end{definition}

\begin{proposition}\label{ampleequivalence}
Let $X$ be a quasi-compact and separated Deligne--Mumford stack and $\LL$ be a line bundle on $X$. Then the following are equivalent:
\begin{enumerate}
\item $\LL$ is ample.
\item There are finitely many sections $x_i \in \Gamma(X,\LL^{\tensor k_i})$ with $k_i\geq 1$ such that the $D(x_i)$ have affine coarse moduli space and cover $X$.
\item The \emph{affine} non-vanishing loci $D(x)$ of sections $x \in
\Gamma(X, \LL^{\tensor k})$ form a basis of the Zariski topology. 
\end{enumerate}
\end{proposition}
\begin{proof}
This is true if $X$ is a scheme by \cite[Propositions 13.47 and 13.49]{G-W10}. As the Zariski topologies of $X$ and its coarse moduli space $Y$ agree, we have only to show in the cases (2) and (3) that the coarse moduli space $Y$ is a scheme. This follows from \cite[Theorem 3.1]{Con05}.
\end{proof}

\begin{example}The line bundle $\omega \cong \pi_2\otop$ on $\mellc$ is ample. Indeed, the non-vanishing loci of $c_4\in \Gamma(\omega^{\tensor 4})$ and $\Delta \in \Gamma(\omega^{\tensor 12})$ have affine coarse moduli space and cover $\mellc$. More details about $(\mellc,\otop)$ will be given in \Cref{AppTMF}.\end{example}

Let $\XX = (X,\otop_\XX)$ be an even periodic derived stack, $f\in\pi_k\gx$ and $\overline{f} \in \Gamma(X,\pi_k\otop_\XX)$ its reduction. Let $\FF$ be a quasi-coherent $\otop_\XX$-module. By the theory of \cite[Section 8.2.4]{higheralg}, we can consider the localization \[\Gamma(\XX,\FF) \to \Gamma(\XX,\FF)[1/f].\] This has the following universal property: let $M$ be a $\gx$-module such that $f$ operates invertible on $\pi_*M$. Then the induced map
\[\Map(\Gamma(\XX,\FF)[1/f], M) \to \Map(\Gamma(\XX,\FF),M) \]
is an equivalence. 

Now assume that the global sections functor $\Gamma\colon \qcoh(\XX) \to \md(\gx)$ commutes with homotopy colimits. Then the presheaf 
\[\FF[1/f]\colon U\mapsto \FF(U) \tensor_{\gx} \gx[1/f]\simeq \FF(U)[1/f]\] 
is already a sheaf by Lemma \ref{AlreadySheaf}. As $\FF \to \FF[1/f]$ is an
equivalence \'etale locally on $D(\overline{f})$, we can conclude thus that
$\FF(D(\overline{f})) \simeq \FF(D(\overline{f}))[1/f]$. In particular, there is thus a canonical map $\Gamma(\FF)[1/f] \to \FF(D(\overline{f}))$. 

\begin{lemma}Let $\XX = (X,\otop_\XX)$ be an even periodic, quasi-compact
and separated derived stack. Assume that the global sections functor commutes
with homotopy colimits. Let $f\in \pi_k\gx$ and $\FF$ be a quasi-coherent
$\otop_\XX$-module. Then the canonical map \[\Gamma(\XX, \FF)[1/f] \to \FF(D(\overline{f}))\] is an equivalence, where $\overline{f} \in \Gamma(\pi_k\otop_\XX)$ is the reduction of $f$.
\end{lemma}

In other words, restricting to a basic open affine gives the corresponding
localization at the level of sections. 

\begin{proof}First assume that $\XX$ is an affine derived scheme with $\pi_2\otop_\XX$
trivial. In particular, we can assume $f$ to be in $\pi_0\Gamma(\XX, \otop)$. By the definition of a derived scheme, we know that $\Gamma(\otop_\XX)[1/f] \simeq \otop_\XX(D(\overline{f}))$. Now the result follows by the quasi-coherence of $\FF$.

Now consider the general case. Let $p\colon \mathfrak{U} \to \XX$ be an affine etale
cover with $p$ affine and such that $\pi_2\otop_\XX$ is trivial on
$\mathfrak{U}$. Define $p_n\colon \mathfrak{U}_n = \mathfrak{U}^{\times_\XX n} \to \XX$. We have a commutative diagram
\[\xymatrix{ \Gamma(\FF)[1/f] \ar[d] \ar[r] & \FF(D(\overline{f})) \ar[d] \\
 \holim \left(  \FF(\mathfrak{U}_n)[1/p_n^*f] \right) \ar[r]& \holim\, \FF(D(p_n^*\overline{f}))
  \simeq \holim\, \FF(\mathfrak{U}_n\times_\XX D(\overline{f})) 
 } \]
The vertical maps are equivalences since $\FF[1/f]$ and $\FF$ are sheaves. The lower horizontal map is an equivalence by the affine case. Thus, the result follows.\end{proof}

\begin{proposition}\label{Zariski-locally-conservative}Let $\XX = (X,\otop_\XX)$ be
an even periodic, quasi-compact, separated derived stack.
Assume that the global sections functor commutes with homotopy colimits and
that $\otop_\XX$ is ample. Let $\{U_i \to X\}_{i\in I} \to X$ be a
Zariski covering and $\mathfrak{U}_i = (U_i,\otop)$ the induced open derived
substacks. Assume that the global sections functor for every $\mathfrak{U}_i$
is conservative. Then the global sections functor 
\[\Gamma: \qcoh(\XX) \to \md(\gx)\]
for $\XX$ is conservative. 

In other words: under the assumptions, the conservativity of the global
sections functor is a Zariski-local property.\end{proposition}
\begin{proof}By shrinking the $U_i$, we can assume that $U_i =
D(\overline{x_i}) \subset X$, where $\overline{x_i} \in
\Gamma(\pi_{k_i}\otop_\XX)$ is the reduction of an element $x_i \in \pi_{k_i}
\gx$. By \Cref{open0affine}, this preserves the property that the global sections functor is conservative.

Now let $\FF$ be a quasi-coherent $\otop_\XX$-module. By the last lemma, we know that 
\[\FF(D(\overline{x_i})) \simeq \Gamma(\FF)[1/x_i] = 0.\]
As the global sections functor is conservative on each $D(\overline{x_i})$, we have $\FF|_{D(\overline{x_i})} = 0$ for every $i\in I$. Since $\FF$ is a sheaf, it follows that $\FF = 0$. 
\end{proof}

Next, we want to prove an algebraic criterion for ampleness of the structure sheaf $\otop_\XX$. We need first a simple lemma. 

\begin{lemma}Let $X$ be a quasi-compact and separated Deligne--Mumford stack, $\LL$ a line bundle on $X$ and $x\in H^0(X;\LL)$. Then $H^i(D(x);\LL^{\tensor *}) \cong H^i(X; \LL^{\tensor *})[1/x]$ for  $\LL^{\tensor *} \cong \bigoplus_{n\in\Z} \LL^{\tensor n}$.
\end{lemma}
\begin{proof}
Let $j\colon D(x) \to X$ be the inclusion of the non-vanishing locus. Define the quasi-coherent graded $\mathcal{O}_X$-module  $\LL^{\tensor *}[1/x]$ as the colimit over 
\[\LL^{\tensor *} \xrightarrow{r_x} \LL^{\tensor *} \xrightarrow{r_x} \cdots\]
where $r_x$ denotes  multiplication by $x$. Then the map $\LL^{\tensor *} \to j_*j^*\LL^{\tensor *}$ factors over $\LL^{\tensor *}[1/x]$. The map $\LL^{\tensor *}[1/x] \to j_*j^*\LL^{\tensor *}$ is an isomorphism as it is on affine schemes with $\LL$ trivial. As $j$ is affine, we have an isomorphism
\[H^*(X;j_*j^*\LL^{\tensor *}) \cong H^*(D(x); j^*\LL^{\tensor *}).\]

It remains to show that cohomology commutes with localization at $x$ on $X$. This follows from 
the fact that cohomology commutes with filtered colimits, by
\Cref{lem:cohcommute}. 
\end{proof}

\begin{proposition}Let $\XX = (X,\otop_\XX)$ be an even periodic,
quasi-compact and separated derived Deligne-Mumford stack. Assume that $\pi_k\otop_\XX$ is ample for some $k\in\Z$ and that the length of the differentials in the descent spectral sequence for $\otop_\XX$ is bounded. Then $\otop_\XX$ itself is ample.\end{proposition}
\begin{proof}Consider some $x \in H^0(X, \pi_k(\otop_\XX))$ such that the non-vanishing locus $D(x)$ has affine coarse moduli space. We want to show that some power of $x$ is a permanent cycle in the descent spectral sequence for $\otop_\XX$. Note that $k$ has to be even. 

Let $d_{i_1}(x)$ be the first non-zero differential of $x$. Consider the morphism $\XX \to \XX_\Q = (X_\Q, (\otop_\XX)_\Q)$ and the image $x_\Q$ of $x$ in $H^0(X_\Q,\pi_k(\otop_\XX)_\Q)$. We show first that $d_{i_1}x_\Q$ is annihilated by a power of $x_\Q$. Denote by $j\colon D(x_\Q) \to X_\Q$ the inclusion. Then we have a map of spectral sequences $DSS((\otop_\XX)_\Q) \to DSS(j_*j^*(\otop_\XX)_\Q)$. As $j_*$ is affine, the $E^2$-term of the latter is isomorphic to $H^*(D(x_\Q); j^*\pi_*(\otop_\XX)_\Q)$. As shown in the last lemma, this is isomorphic to $H^*(X; \pi_*(\otop_\XX)_\Q)[1/x_\Q]$ since $\pi_{2k}(\otop_\XX)_\Q \cong (\pi_2\otop_\XX)_\Q^{\tensor k}$. On the other hand, $H^i(D(x_\Q); j^*\pi_*(\otop_\XX)_\Q) = 0$ for $i>0$ as the coarse moduli space of $D(x_\Q)$ is affine and $D(x_\Q)$ is tame by \Cref{Char0Tame} as it is rational. Thus, indeed, $d_{i_1}(x_\Q)$ has to be annihilated by a power of $x_\Q$. 

For $m\in\mathbb{N}$, we have $d_{i_1}x^m = mx^{m-1}d_{i_1}(x)$. Hence, $d_{i_1}x^m_\Q = 0$ for some $m$ and thus $d_{i_1}x^m$ is $l$-torsion for some $l$. It follows that $d_{i_1}(x^{lm}) = lx^{m(l-1)}d_{i_1}(x^m) = 0$. The argument can be repeated for the first non-trivial differential of $x^{lm}$ etc. and comes to an end somewhere as the length of the differentials is bounded. Thus, there is some power $x^K$ of $x$ that is a permanent cycle and $D(x^K)=D(x)$.

As the $D(x)$ with affine coarse moduli space form a basis of the
Zariski topology by Proposition \ref{ampleequivalence}, $\otop_\XX$ is ample. 
\end{proof}

The existence of an upper bound for the length of differentials in the descent spectral
sequence is actually closely related to the cocontinuity of $\Gamma$. One situation
where these conditions are trivially fulfilled is that of bounded cohomological
dimension, which is satisfied for a derived \emph{scheme}:

\begin{corollary}\label{Scheme0affineness}Let $\XX = (X,\otop_\XX)$ be an even
periodic, quasi-compact and separated derived \emph{scheme} $X$. Assume that $\pi_k\otop_\XX$ is an ample $\mathcal{O}_X$-module for some $k\in \Z$. Then the global sections functor 
\[\Gamma \colon \qcoh(\XX) \to \md(\gx)\]
is an equivalence. In particular, this is true for $X$ quasi-affine as every line bundle is ample on a quasi-affine scheme. \end{corollary}
\begin{proof}There is an $n\geq 0$ such that $H^i(X; \pi_l\otop_\XX) = 0$ for $i>n$ and all $l\in\Z$. The global sections functor $\Gamma$ commutes with homotopy colimits as the descent spectral sequence for $\otop_\XX$ is concentrated in finitely many rows. Likewise it follows that the length of differentials in the descent spectral sequence for $\otop_\XX$ is bounded. By the last proposition it follows that $\otop_\XX$ is ample. Furthermore, it is certainly Zariski-locally true that the global sections functor is conservative (as it is true on every affine scheme). Thus, we can apply Proposition \ref{Zariski-locally-conservative} to see that $\Gamma$ is conservative.  This implies $0$-affineness by Corollary \ref{SSCor}.\end{proof}
\begin{remark}In the quasi-affine case, Corollary \ref{Scheme0affineness} was
already proven in Propositions 2.4.4 and 2.4.8 of \cite{DAGQC} for  connective
spectral Deligne--Mumford stacks. From this, the even periodic case can be
easily recovered by taking the connective cover. 
Any (possibly nonconnective) spectral Deligne-Mumford stack 
  has a connective cover it maps to (just as any $E_\infty$-ring $R$ receives
  a map from its connective cover $\tau_{\geq 0} R$), and any
 quasi-coherent sheaf  can be pushed forward to the connective cover. The
 global sections are the same in either case. In this way, one can always
 reduce to the case in which the derived stack is connective.  

Our original formulation of our main results was restricted to the case of
cohomological dimension one. We are indebted to Jacob Lurie for explaining to
us his (only slightly different) argument for \Cref{Scheme0affineness}, which consequently yielded a stronger
formulation of Theorem~\ref{main}. 
\end{remark}

\begin{example}Consider the compactified moduli stack of elliptic curves
$\mellc(n)$ with level $n$ structure. By work of Goerss--Hopkins and
Hill--Lawson (\cite{HillLawson}), this can be refined to an even periodic
derived stack $(\mellc(n), \otop)$ with $\pi_2\otop$ ample. Thus it follows
directly from \Cref{Scheme0affineness} that for $n\geq 3$, when $\mellc(n)$ is a scheme, $(\mellc(n), \otop)$ is $0$-affine. As explained in Section \ref{AppTMF}, this is actually true for all $n$.
\end{example}

\subsection{Ascent for $0$-affineness}
In this subsection, we note a couple of additional easy criteria for
0-affineness.

\begin{proposition} \label{open0affine}
Suppose $\XX = (X, \otop)$ is a 0-affine derived stack. Let $U \subset X$ be an open
substack and let $\mathfrak{U} = (U, \otop_{\mathfrak{U}})$ be the induced derived
stack. Then $\mathfrak{U}$ is also 0-affine. 
\end{proposition} 
\begin{proof} 
Given a quasi-coherent sheaf $\mathcal{F}$ on $\mathfrak{U}$, if its global
sections are zero, then we show $\mathcal{F} = 0$ as follows: form the
pushforward $j_* \mathcal{F}$ along the inclusion $j\colon U \hookrightarrow X$. 
By assumption, $\Gamma( j_*\mathcal{F}) \simeq \Gamma( \mathcal{F}) \simeq 0$,
so that $j_* \mathcal{F} \simeq 0$ by the 0-affineness of $\mathfrak{X}$. Since
$j^*j_* \mathcal{F} \simeq \mathcal{F}$, it follows that $\mathcal{F} \simeq 0$. 

Furthermore, the pushforward functor $j_*$ commutes with homotopy colimits by
Example 2.5.6 and Proposition 2.5.12 of \cite{DAGQC}. Thus,
$\Gamma(\mathfrak{U},-) \simeq \Gamma(\XX,-) \circ j_*$ also commutes with homotopy 
colimits. 
\end{proof} 

\begin{proposition} \label{affineover0affine}
Let $f: \YY \to \XX$ be a morphism of derived stacks for  $\XX = (X, \otop_\XX)$ and $\YY = (Y,
\otop_\YY)$ such that the underlying morphism $Y \to X$ is affine. If
$\mathfrak{X}$ is 0-affine, then
$\mathfrak{Y}$ is 0-affine. 
\end{proposition} 
\begin{proof} 
The conservativeness of $\Gamma(\YY,-)$ follows again, because pushforward 
along $\mathfrak{Y} \to \mathfrak{X}$ is a conservative functor. Indeed, we can
choose an \'etale cover $\{U_i \to X\}_{i\in I}$ by affine schemes. Thus, $\{U_i\times_X Y \to Y\}_{i\in I}$ is an \'etale cover of $Y$. If $f_*\FF = 0$ for a quasi-coherent sheaf $\FF$ on $\YY$, then $f_*\FF(U_i) = \FF(U_i\times_X Y) = 0$ for every $i\in I$. Thus, $\FF = 0$. 

Furthermore, the pushforward functor $f_*$ commutes with homotopy colimits by
Example 2.5.6 and Proposition 2.5.12 of \cite{DAGQC}. Thus, $\Gamma(\YY,-)
\simeq \Gamma(\XX,-) \circ f_*$ also commutes with homotopy colimits. 
\end{proof}

\section{Affineness results in chromatic homotopy theory}\label{Section4}
The main result of this section is the following:
\begin{theorem}  \label{main}
Let $X$ be a noetherian and separated Deligne-Mumford stack, equipped with a flat map $X \to M_{FG}$ that is quasi-affine. Let $\mathfrak{X}$ be an
even periodic refinement of $X$. Then the derived stack $\XX = (X, \otop)$ is $0$-affine.
\end{theorem} 
Recall here that a map $X\to M_{FG}$ is \textit{quasi-affine} if for every map
$\Spec A \to M_{FG}$ the pullback $\Spec A\times_{M_{FG}} X$ is quasi-affine,
i.e., a quasi-compact open subscheme of an affine scheme. 

\begin{remark}
 The condition that $X$ is separated is not very restrictive. Recall that $X$
 is separated if its diagonal is universally closed. This implies for (locally
 noetherian) Deligne--Mumford stacks only that the diagonal is finite, but not that it is a closed immersion. The Deligne--Mumford stacks most commonly considered by algebraic geometers, like the (compactified) moduli stack of elliptic curves or PEL-Shimura stacks, are separated. 
\end{remark}

We will first show Theorem \ref{main} locally at every prime $p$ using
Propositions \ref{AbstractCommHolim} and \ref{AbstractConservative}. This
relies crucially on the fact that Morava $E$-theory $E_n$ generates the $E_n$-local stable homotopy category as a thick tensor-ideal, which in turn follows from the nilpotence technology to be reviewed in the next subsection. We will then glue the $p$-local results together to deduce an integral statement. 

\subsection{Nilpotence technology}
The power of the use of formal groups in stable homotopy theory is especially
illustrated by the nilpotence and periodicity theorems of \cite{DHS} and
\cite{HS}, and their cousin, the Hopkins-Ravenel smash product theorem. 

\begin{theorem}[Nilpotence theorem; \cite{DHS}]
\begin{enumerate}
\item 
Let $R$ be a (not necessarily structured) ring spectrum and let $\alpha \in
\pi_* R$. Suppose $\alpha$ maps to zero in $MU_*(R)$. Then $\alpha$ is
nilpotent. 
\item  Let $f\colon \Sigma^k X \to X$ be a self-map of a finite spectrum. Then $f$
is nilpotent if and only if $MU_*(f)$ is nilpotent. 
\end{enumerate}
\end{theorem} 

In particular, the association $X \mapsto \mathcal{F}_*(X)$, from spectra to
quasi-coherent sheaves on $M_{FG}$, is sufficient to detect all maps except up
to nilpotence, at least for finite spectra. Conversely, many of the
``periodicities'' visible in the geometry of $M_{FG}$ can be realized
topologically, via the \emph{periodicity theorem} of \cite{HS}. 

Next, recall that a subcategory of the $\infty$-category $\sp_{(p)}^{\omega}$ of
finite $p$-local spectra is
called \emph{thick} if it is stable and closed under 
retracts. 

\begin{theorem}[Thick subcategory theorem; \cite{HS}] 
The thick subcategories of $\sp_{(p)}^{\omega}$ are in natural bijection with
the reduced, closed substacks of $(M_{FG})_{(p)}$ of finite presentation. In
particular, if a thick subcategory $\mathcal{C}$ contains a spectrum with
nontrivial rational homology, then $\mathcal{C} = \sp_{(p)}^{\omega}$. 
\end{theorem}

The next result is the strongest finiteness theorem that we need. 
Let $E_n$ be the $n$th \emph{Morava $E$-theory}; then $E_n$ is a Landweber-exact, even periodic $E_\infty$-ring with
$\pi_0 E_n \simeq {W}( \mathbb{F}_{p^n})[[v_1, \dots, v_{n-1}]]$. 
Given a
spectrum, the functor of $E_n$-localization can be thought of as (after
localizing at $p$) restriction to the open substack of $(M_{FG})_{(p)}$ 
parametrizing formal groups of height $\leq n$. 

\begin{theorem}[Smash product theorem; \cite{ravenelorange}]
\label{smashproduct1}
The $E_n$-localization functor $L_n\colon \sp \to \sp$ commutes with homotopy
colimits, so that for any spectrum $X$, the natural map
\[ X \otimes L_n S^0 \to L_n X,  \]
is an equivalence. 
\end{theorem} 

Note here that $E_n$ has the same Bousfield class as the Johnson--Wilson theory $E(n)$ by \cite[Proposition 5.3]{HoveyS}.

Theorem~\ref{smashproduct1} is essentially a statement about certain homotopy colimits and limits
commuting with each other, and is crucial to describing the structure of the
$\infty$-category $L_n \sp$ of $E_n$-local spectra. 
A general reference for this category is \cite{HoveyS}. 
The smash product theorem implies that $L_n \sp$ 
is a full subcategory of $\sp$ which is closed under homotopy limits \emph{and
colimits}. The $\infty$-category $L_n \sp$ has much better finiteness
properties than the category of spectra (or even the category of $p$-local
spectra), as we will explain next. 

We will need various slightly stronger versions of \Cref{smashproduct1}
(which are ultimately the ingredients used to prove it) for our purposes.  
We can recover the $E_n$-local sphere
$L_n S^0$
via the totalization of the  classical \emph{cobar construction}
\[ E_n \rightrightarrows E_n \otimes E_n  \triplearrows \dots,  \]
whose associated $\mathrm{Tot}$ spectral sequence is the $E_n$-local ANSS. 
A strong form of the smash product theorem implies that this spectral sequence (drawn
with the Adams indexing $(s, t-s)$) degenerates at a
finite stage with a \emph{horizontal} vanishing line. More generally, this is
true for any $E_n$-local spectrum replacing $L_n S^0$ by the following result.

\begin{theorem}[{\cite[Prop. 6.5]{HoveyS}}] 
\label{horizvanishingline}
There is a uniform bound $N = N(n)$ such that 
given any $E_n$-local spectrum $X$, the ANSS for $X$ satisfies $E^{s,t}_{N} =
0$ for $s > N$.
\end{theorem} 

We can formulate a closely related statement in the language of \emph{pro-spectra.}
The $\infty$-category $\mathrm{Pro}( \sp)$ is the $\infty$-category of
pro-objects in $\sp$ in the sense, e.g., of Chapter 5 in \cite{HTT}: a
pro-spectrum is  a formal filtered homotopy inverse limit of spectra. 

Given a cosimplicial diagram $F\colon \Delta \to \sp$, we can form the homotopy inverse 
limit $\mathrm{Tot} F$ in spectra, but we can also do it in pro-spectra. This
amounts to considering the $\mathrm{Tot}$ tower
\[ \dots \to  \mathrm{Tot}^2 F \to \mathrm{Tot}^1 F \to \mathrm{Tot}^0 F,  \]
as a pro-object. 
There is a fully faithful inclusion of $\infty$-categories $\sp \to
\mathrm{Pro}( \sp)$. 
Pro-spectra in the image of $\sp$ are called \emph{constant}. 
Another reformulation of the smash product theorem is the following. 

\begin{theorem}[Hopkins-Ravenel]\label{HRsmashconst} The pro-spectrum associated to the cobar
construction on $E_n$ is constant with value $L_n S^0$. 
\end{theorem} 
The proof of this is explained in Lectures 30 and 31 of Lurie's course on chromatic homotopy theory \cite{chromatic} and \cite[Chapter 8]{ravenelorange}. 
As explained in 
Lecture 30 of \cite{chromatic}, this is closely related to (and in fact follows
from) the 
horizontal vanishing line in the $L_n$-local Adams-Novikov spectral sequence
(\thref{horizvanishingline}), using a delicate criterion for constancy of pro-spectra
due to Bousfield. In particular, \Cref{horizvanishingline} is actually used to
prove \Cref{HRsmashconst}.  

\Cref{HRsmashconst} states that there is an equivalence of pro-objects between the
$\mathrm{Tot}$-tower for the cobar construction and the constant pro-object
with value $L_n S^0$. In particular, the natural maps
\[ L_n S^0 \to \mathrm{Tot}^m( ( E_n^{\otimes (\bullet+1)}))  \]
have, for large enough $m$, sections up to homotopy. 
For further discussion, see also Lecture 30 of \cite{chromatic} and, for
connections between these ideas and descent theory, 
\cite[\S 3, 4]{galoischromatic}. 

In particular, we get the following corollary, which is the piece of the
nilpotence technology which we shall use.

\begin{corollary} \label{usefulcor}
$E_n$ generates $L_n \sp$ as a thick tensor-ideal. 
\end{corollary} 

\begin{proof} 
Let $\mathcal{C} \subset L_n \sp$ be a thick tensor-ideal containing $E_n$. 
This implies that the partial totalizations $\mathrm{Tot}^m$ of the
cobar construction on $E_n$ belong to $\mathcal{C}$. Since $\mathcal{C}$ is
closed under retracts, it follows that $L_n S^0 \in \mathcal{C}$ and thus
$\mathcal{C} = L_n \sp$. 
\end{proof} 

A similar result is discussed in Theorem 5.3 of \cite{hoveysadofsky}, where it is shown that every spectrum in
$L_n \sp$ is ``$E(n)$-nilpotent,''  or equivalently belongs to the thick
tensor-ideal generated  
by $E(n)$. 
\begin{remark} 
\label{MUnotdescendable}
Similarly, the nilpotence theorem is closely related to the statement that for any
connective spectrum $X$, the ($MU$-based) Adams-Novikov spectral sequence for
$X$ has a \emph{vanishing curve} of slope tending to zero as $t-s \to \infty$
at $E^\infty$.
See \cite{hoptalk} for further discussion of this. 
However, the $MU$-based cobar construction (whose homotopy inverse limit is $S^0$) is
definitely nonconstant, because of
the existence of nontrivial $MU$-acyclic spectra (for instance, the
Brown-Comenentz dual 
$I$ of the sphere; see Appendix B of \cite{HS}). 
\end{remark} 

\subsection{The $p$-local case}
Let $X$ be a noetherian separated Deligne--Mumford stack over $\mathbb{Z}_{(p)}$ with a flat map to $M_{FG}$. In this section, everything is implicitly localized at $p$: for instance,
$M_{FG}$ really means $M_{FG} \times_{\spec \mathbb{Z}} \spec \mathbb{Z}_{(p)}$. 

We want to find in this case instances of the abstract theorems of the last section. We begin by choosing an $n \in \mathbb{N}$ such that 
the flat map $X \to M_{FG}$ factors through the open substack $M_{FG}^{\leq n}
\subset M_{FG}$ of formal groups of height at most $ n$. 
We can do this because we have a descending sequence of closed substacks
\[ M_{FG} \supset M_{FG}^{\geq 1} \supset M_{FG}^{\geq 2} \supset \dots, \]
where each $M_{FG}^{\geq n+1}$ is cut out by the vanishing of a regular element
on $M_{FG}^{\geq n}$. Since $X$ is noetherian, $X \times_{M_{FG}} M_{FG}^{\geq m}$ must
be empty for $m \gg 0$. 

Observe the following (well-known) lemma:
\begin{lemma}Let $f\colon \Spec R \to M_{FG}$ be a flat map, which factors over $M_{FG}^{\leq n}$, and $E_R$ be the corresponding Landweber exact spectrum. Then $E_R$ is $E_n$-local.\end{lemma}
\begin{proof}Let $T$ be a spectrum. The $MU_*$-comodule $(MU_*T, (MU\otimes
MU)_*(T))$
defines a $\Z/2$-graded quasi-coherent sheaf $\FF_*$ on $M_{FG}$ (see
Subsection \ref{Even-FG} for a discussion). Then $\pi_*(E_R \tensor T)\cong
f^*\FF_*(\Spec R)$. Denote by $q$ the canonical map $\Spec \pi_0E_n \to
M_{FG}^{\leq n} \subset M_{FG}$. So, likewise, we have  an isomorphism $\pi_*(E_n\tensor T) \cong q^*\FF_*(\Spec R)$. Now assume that $T$ is $E_n$-acyclic. Thus, $q^*\FF_* = 0$ and hence, since $\Spec \pi_0E_n \to M_{FG}^{\leq n}$ is faithfully flat, $\FF_*|_{M_{FG}^{\leq n}} = 0$. This implies $f^*\FF_* = 0$ and thus $E_R\tensor T = 0$. Thus 
\[ [T, E_R]  = 0  \]
for every $E_n$-acyclic spectrum $T$, since $E_R$ (as a ring spectrum) is
local with respect to itself. \end{proof}

Now let $\XX = (X,\otop_\XX)$ be an even periodic refinement of $X\to M_{FG}$. By the lemma, for every \'etale map $\spec R \to X$, the $E_\infty$-ring $\otop(
\spec R)$ is $E_n$-local. Since $E_n$-local spectra are closed under homotopy limits, thus for every \'etale map $Y \to X$, the $E_\infty$-ring $\otop(Y)$ is $E_n$-local. We see that the whole argument takes place in the $E_n$-local category. 

\begin{proposition}  \label{plocalfinite}
Let $X$ be a noetherian and separated Deligne-Mumford stack over a $p$-local
ring, equipped with a flat and \textit{tame} map $X \to M_{FG}$. Let $\mathfrak{X} = (X,\otop_\XX)$ be an
even periodic refinement of $X$. Then the functor
of taking global sections
\[ \Gamma\colon \qcoh(\mathfrak{X}) \to \md( \Gamma(\mathfrak{X}, \otop ))  \]
commutes with homotopy colimits. \end{proposition} 
\begin{proof}Let $Y = X \times_{M_{FG}} \spec \pi_0 E_n$. This has finite
cohomological dimension by \Cref{tamebcd} as the map $X \to M_{FG}$ is tame.
It is also quasi-compact and separated as $X$, $\spec \pi_0E_n$ and the diagonal of $M_{FG}$ are. We want to apply Proposition \ref{AbstractCommHolim} with $M = \gx
\tensor E_n$. It follows from Corollary \ref{usefulcor} that $M$ generates
$\md(\gx)$ as a thick tensor-ideal. We have to show that 
\[\pi_*\left(\otop_\XX\otimes_{\gx} M\right) \cong \pi_*(\otop_\XX\tensor E_n) \cong q_*q^*\pi_*(\otop_\XX)\]
 for $q\colon Y \to X$ the projection map. This follows from
 \Cref{LandweberSmash}
 for flat maps $\spec R\to M_{FG}$ and $\spec R' \to M_{FG}$ the smash product
 of the two Landweber exact spectra $E_R$ and $E_{R'}$ can be computed as 
\[ \pi_{2k} \left(E_R \tensor E_{R'} \right) \cong \omega^{\tensor k}(\spec R\times_{M_{FG}} \spec R').\]
Specialized to our situation, we get that for every flat map $\spec R \to X$ we have the following natural isomorphisms
\[ \pi_{2k} \left(\otop_\XX(\spec R)\tensor E_n\right) \cong \pi_{2k}(\otop_\XX)(\spec R\times_{M_{FG}} \spec \pi_0 E_n) \cong q_*q^*\pi_{2k}(\otop_\XX).\]\end{proof}

Set again $Y = X \times_{M_{FG}} \spec \pi_0 E_n$. We want to define an even periodic refinement of this. There are two equivalent ways of doing this:
\begin{enumerate}
\item 
Let $U \to X$ be an affine \'etale cover and $U_\bullet$ the corresponding Cech simplicial object. We can define even periodic refinements on $U_k\times_{M_{FG}} \pi_0 E_n$ by considering $\spec \otop_\XX(U_k)\tensor E_n$. Note here that $E_n$ has the structure of an $E_\infty$-ring spectrum by the Goerss--Hopkins--Miller theorem. Then we can define $\YY := \hocolim\, \spec (\otop_\XX(U_k)\tensor E_n)$. 

In other words, one notes that to realize $Y$ as a derived stack, one needs to
construct an appropriate \emph{diagram} of even periodic $E_\infty$-rings,  
corresponding to any given presentation of $Y$ as an ordinary stack. Given any
\'etale map $\spec R \to X$, we can realize $\spec R \times_{M_{FG}} \spec
\pi_0 E_n$ via the $E_\infty$-ring $\otop( \spec R) \otimes E_n$. This
constructs a diagram of $E_\infty$-rings which is enough (by a descent
procedure)  to produce the sheaf of $E_\infty$-rings on the \'etale site of $Y$ (in an analogous way to \Cref{gpquotient}). 
\item We can define $\YY$ as the ``relative $\spec$'' of the sheaf of algebras
$\otop \tensor E_n$, using essentially the previous construction. 
\end{enumerate}
Thus, we get an even periodic stack $\YY = (Y,\otop_\YY)$ with a faithfully flat, separated and quasi-compact map $q:\YY \to \XX$ such that \[q_* \otop_\YY \simeq \otop_\XX \tensor E_n.\] 
We now get:

\begin{theorem}\label{p-main}We use the same notation and assumptions from
Proposition \ref{plocalfinite}. Assume furthermore that $Y$ is quasi-affine.
Then $\Gamma\colon \qcoh(\XX) \to \md(\gx)$ is conservative and thus an equivalence.\end{theorem}
\begin{proof}In light of \Cref{plocalfinite}, this is a direct application of Proposition \ref{AbstractConservative} and Corollary \ref{Scheme0affineness} as the underlying stack $Y$ of $\YY$ is quasi-affine. 
\end{proof}

\begin{remark}Note that the condition that $\pi_k\otop_\YY$ is ample is not more general as $\pi_k\otop_\YY \cong \pi_0\otop_\YY$ for $k$ even and $0$ else (since $E_n$ is strongly even periodic).\end{remark}

\subsection{The integral version}

In this section, we complete the proof of Theorem~\ref{main}. To start with, we
extend the proof of the first step 
when localized at $p$, as done in the previous section, to an integral statement. 
Once again, we have a slightly stronger statement. 
\begin{theorem} \label{general:colimits}
Let $X$ be a noetherian and separated Deligne-Mumford stack, equipped with a flat map $X \to M_{FG}$. 
Let $\mathfrak{X}$ be an
even periodic refinement of $X$. 
Then, if $X \to M_{FG}$ is tame, the 
global sections functor
\[  \Gamma\colon \qcoh(\mathfrak{X}) \to \md( \Gamma(\mathfrak{X}, \otop ))  \]
commutes with homotopy colimits.
\end{theorem} 
\begin{proof}
In order to prove this, we will use the $p$-local version proved earlier, for
each prime, together with an arithmetic square to fit everything together
integrally. There is an obstacle in that the ``arithmetic square'' is infinite in
nature. To deal with this, we use:

\begin{lemma}Let $X$ be a quasi-compact and separated Deligne--Mumford stack
over some quasi-compact scheme $S$. Then there is an $N$ such that $X[\frac 1N]$ has bounded cohomological dimension.\end{lemma}
\begin{proof}We want first to show that the order of automorphism groups of
points in $X$ is bounded. As $X$ is separated, the inertia stack
$X\times_{X\times_S X} X = \II_X$ is finite over $X$. Since $X$ is
quasi-compact, there is an \'etale covering $q\colon\spec A \to X$ for some
ring $A$. The pullback $q^*\II_X \to \spec A$ corresponds to an $A$-module
generated by $m$ elements for some $m$. If $x\colon \Spec k \to X$ is a
geometric point, the pullback $\spec k \times_X \spec A$ is equivalent to a
disjoint union of $\Spec k$. Thus, $x^*\II_X \to \spec k$ has also rank at most
$m$, i.e., the stabilizer of $x$ has at most $m$ elements. 

Let $N = m!$. Then all stabilizers have invertible order on $X[\frac1N]$. Thus, $X[\frac1N]$ is tame, 
which implies the result by \Cref{tbcd}. 
\end{proof}

It follows from this that there exists an integer $N \in \mathbb{N}$ such that,
after tensoring with the localization $S^0[N^{-1}]$, the functor 
\[ \mathcal{F } \mapsto \Gamma(\mathfrak{X}, \mathcal{F} \otimes S^0[N^{-1}]), 
\quad \qcoh(\mathfrak{X}) \to \md( \Gamma(\mathfrak{X}, \otop)),
\]
commutes with homotopy colimits. In fact, the spectral sequence to compute the
homotopy groups of 
$\Gamma(\mathfrak{X}, \mathcal{F} \otimes S^0[N^{-1}])$ starts from the cohomology
of $\pi_* \mathcal{F}$ on the open substack $\mathfrak{X}[N^{-1}]$, since
cohomology commutes with localization, and is consequently concentrated in
finitely many rows at $E^2$. 

In view of \Cref{plocalfinite} and \Cref{tamebcd}, we can thus apply the following lemma to
conclude the proof of the theorem. 

\begin{lemma} 
Let $F\colon \mathcal{C} \to \mathcal{D}$ be an exact functor between cocomplete stable
$\infty$-categories. Suppose that: 
\begin{enumerate}
\item 
$F( \cdot \otimes S^0_{(p)} )$ commutes with homotopy 
colimits for every prime number $p$. 
\item 
There exists an integer $N$ such that $F( \cdot \otimes S^0[1/N])$ commutes
with homotopy colimits. 
\end{enumerate}
Then $F$ commutes with homotopy colimits. 
\end{lemma} 
\begin{proof} 
Consider the collection $\mathcal{I}$ of spectra $T$ such that $F( \cdot \otimes T)$ commutes
with homotopy colimits. It is an \emph{ideal} in spectra: that is, if $X $ is
a spectrum and $Y \in \mathcal{I}$, then $X \otimes Y \in \mathcal{I}$. 
By hypothesis, this ideal contains $S^0[1/N]$ for some $N$ and each $S^0_{(p)}$
for each prime number $p$. We want to show that it contains $S^0$.

To do this, we use an inductive argument. Let $N \in \mathbb{Z}_{>0}$ be chosen \emph{minimal}
such that $S^0[1/N] \in \mathcal{I}$. We want to show that $N = 1$. 
Observe that if $(m, p) = 1$, then 
there is an \emph{arithmetic square}, i.e., a homotopy pullback diagram
\[ \xymatrix{
S^0[1/m] \ar[d] \ar[r] &  S^0[1/mp] \ar[d] \\
S^0_{(p)}[1/m] \ar[r] &  S^0_{\mathbb{Q}}
}.\]
It follows that if $N  > 1$, then $N = p m $ for $(p, m) = 1$ ($N$ is
squarefree by minimality), and then the above arithmetic square implies that
$S^0[1/m] \in \mathcal{I}$, a contradiction. Thus $N = 1$ and we are done. 
\end{proof}

This completes the proof of \Cref{general:colimits}.
\end{proof}

\begin{proof}[Proof of Theorem~\ref{main}]
Let us now complete the proof of the main theorem. By Corollary \ref{SSCor}, it suffices now to show that if $\Gamma(\XX,\FF) = 0$ for some quasi-coherent sheaf $\FF$ on $\XX$, then $\FF = 0$. 

So assume $\Gamma(\XX,\FF) = 0$. Then $\Gamma(\XX_{(p)}, \FF_{(p)}) \simeq
\Gamma(\XX, \FF)_{(p)} = 0$ for every prime $p$. Indeed, since $\Gamma$ commutes with homotopy colimits, it commutes with localization at $p$. By \Cref{p-main}, it follows that $\FF_{(p)} = 0$ for every prime $p$. Thus, $\FF = 0$. 
\end{proof}

\section{Applications to Galois theory}\label{Section5}

Let $R$ be an $E_\infty$-ring. Recall that an $E_\infty$-$R$-algebra $R'$ is said to be
\emph{\'etale} if $\pi_0 R \to \pi_0 R'$ is an \'etale morphism of commutative
rings, and the natural map $\pi_0 R' \otimes_{\pi_0 R} \pi_* R \to \pi_* R'$ is
an isomorphism. 
The theory of \'etale extensions in this sense is entirely algebraic: the
$\infty$-category of \'etale $R$-algebras is equivalent to the (ordinary)
category of \'etale $\pi_0 R$-algebras. 

This definition excludes useful examples such as the map $KO \to KU$, 
which behaves in many respects as an \'etale morphism  in commutative
algebra, albeit not on the level of homotopy groups. Since $\pi_0 KO \simeq \mathbb{Z}$, there are no finite \'etale
extensions of $KO$.  
Nonetheless, $KO \simeq KU^{h \mathbb{Z}/2}$, and, as we have shown
(e.g., in view of \Cref{main}), there is a good theory
of $\mathbb{Z}/2$-``descent'' from $KU$ to $KO$. Rognes's notion of a faithful Galois extension 
(\Cref{def:rognes}) is a generalization of the above notion of \'etaleness
(or at least the Galois version) that
encompasses examples such as $KO \to KU$. 

In this section, we analyze the Galois theory--- in this sense---for
$E_\infty$-rings 
such as $KO$ and $\TMF$ which arise as ``rings of functions''
$\Gamma(\mathfrak{X}, \otop)$ of $0$-affine derived stacks. Our main
result (\Cref{stackgivesgalois}) is that
a Galois cover (in the algebraic sense) of the underlying stack yields a
faithful Galois
extension of $\Gamma( \mathfrak{X}, \otop)$. This provides examples of 
Galois extensions of (localizations of) $\TMF$ via level structures, for instance.

\subsection{Galois extensions}

\label{subsec:gal}
Let $B$ be an $E_\infty$-ring with the action of a finite group $G$, and let
$A = B^{h G}$ be the homotopy fixed points. 
We recall the following definition of Rognes: 

\begin{definition}[\cite{rognes}]
\label{def:rognes}
$A \to B$ is said to be a \emph{$G$-Galois extension} if the map
of $E_\infty$-$A$-algebras
\[ B \otimes_A B \to \prod_{g \in G} B,  \]
which informally is given by $b_1 \otimes b_2 \mapsto \{b_1 \cdot g(b_2)\}_{g \in
G}$, is an equivalence. A Galois extension is said to be \emph{faithful} if the Bousfield classes of
$A$ and $B$ (for $A$-modules) are equivalent: that is, if an $A$-module smashes to zero with $B$,
then it itself is zero. 
\end{definition}

This is inspired by the notion of a Galois extension of (discrete) commutative rings, which can be defined in the same way, but where faithfulness is automatic. Equivalently, a map $R\to S$ of commutative rings is $G$-Galois if $\spec S \to \spec R$ is an etale $G$-torsor in the sense of algebraic geometry. 

Faithful Galois extensions (which are the only type of Galois extensions we
shall consider) are very well-behaved. 
The map $A \to B$ exhibits $B$ as a \emph{perfect} (i.e., compact or dualizable) $A$-module, and for any
$E_\infty$ $A$-algebra $A'$, the map of rings $A' \to B \otimes_A A'$ is
again faithful and $G$-Galois. Moreover, one can develop \cite{galoischromatic}
a version of Grothendieck's \'etale fundamental group formalism in this setting.

We start by noting a few examples and properties of faithful Galois extensions. 
\begin{example}\label{etalegal}
Suppose $A$ is an $E_\infty$-ring, and suppose $B_0$ is a $G$-Galois extension
of the ring $\pi_0 A$. Then there exists a unique $E_\infty$-ring $B$
\'etale over $A$ with $\pi_0 B \simeq B_0$, and a $G$-action on $B$ in the
$\infty$-category of $A$-algebras such that
the natural map $A \to B^{hG}$ is an equivalence (by \Cref{topinv}). 
\end{example}

\begin{example}[{\cite[Proposition 3.6]{BR2} and \cite[Prop.
6.28]{galoischromatic}}]
\label{fieldgal} Suppose $A$ is an even periodic $E_\infty$-ring such that
$\pi_0 A$ is a field. Then $G$-Galois extensions of $A$ are equivalent to
$G$-Galois extensions of $\pi_0 A$: that is, they are \'etale. The main ingredient is the K\"unneth isomorphism for $A$-modules. 
\end{example} 

\begin{example} 
A simple example of a Galois extension that is not \'etale is as follows: let
$A$ be an $E_\infty$-ring with $\pi_*(A) \simeq \mathbb{Z}[1/2, t^{\pm 1}]$,
where $|t| = 2$. Consider a $\mathbb{Z}/2$-action on $A$ that sends $t \mapsto
-t$. In this case, the map $A^{h\mathbb{Z}/2} \to A$ is a $\mathbb{Z}/2$-Galois
extension realizing on homotopy the map $\mathbb{Z}[1/2, u^{\pm 1}] \to
\mathbb{Z}[1/2, t^{\pm 1}], \ u \mapsto t^2$, as we will now show.

 First observe that the map
\begin{align*}\Phi\colon \Z[1/2, t^{\pm1 }] \tensor \Z[1/2, t^{\pm1 }] &\to \Z[1/2, t^{\pm1 }]\times \Z[1/2, t^{\pm1 }] \\
x\tensor y &\mapsto (x\cdot y, x\cdot g(y)) \end{align*}
is surjective, where $g$ generates $\Z/2$. As this map is $\Z[1/2, t^{\pm1 }]$-linear, this follows from the fact that $\Phi(\frac12\tensor 1 + \frac12t^{-1}\tensor t)=(1,0)$ and  $\Phi(\frac12\tensor 1 - \frac12t^{-1}\tensor t)=(0,1)$. By a graded version of \cite[Theorem 1.6]{Greither}, we see that $\mathbb{Z}[1/2, u^{\pm 1}] \to
\mathbb{Z}[1/2, t^{\pm 1}], \ u \mapsto t^2$ is a $\mathbb{Z}/2$-Galois
extension in the graded sense. As $\mathbb{Z}[1/2, t^{\pm 1}]$ is free over $\mathbb{Z}[1/2, u^{\pm 1}]$, this implies that $A^{h\mathbb{Z}/2} \to A$ is a $\mathbb{Z}/2$-Galois.
\end{example}

\begin{example} 
\label{faithfulgalconn}
While the notion of a faithful Galois extension generalizes that of an \'etale Galois
extension (see \cite{BakerRichteralgebraic}), the notions coincide on \emph{connective} $E_\infty$-rings $A$. We
prove this here if $\pi_0(A)$ is \emph{noetherian.} In fact, 
let $A$ be as in the previous sentence, and let $B$ be a faithful $G$-Galois
extension. For any morphism $\pi_0 A \to k$, for $k$ a field, we get a map
of $E_\infty$-rings
\[ A \to \tau_{\leq 0} A \simeq H \pi_0 A \to H k,  \]
and the base-change $B \otimes_A H k$ is therefore a faithful $G$-Galois extension of
$Hk$, which, thanks to the K\"unneth isomorphism, is necessarily discrete (and the Eilenberg-MacLane spectrum
associated to a product of copies of finite
separable extensions of $k$).  

It follows that $B$, which is a perfect $A$-module, is actually connective, and
indeed \emph{flat} in the sense of \cite{higheralg}, \S8.2.2: 
In \cite{higheralg}, it is shown that an $A$-module $M$ is flat if and only
if, for every discrete $A$-module (i.e., $\pi_0 A$-module) $N$, the $A$-module $M \otimes_A N$ is discrete. 
However, it suffices to show that $M \otimes_A H \pi_0 A$ is a discrete, flat
$H \pi_0 A$-module. 
Now we can appeal to a classical fact from commutative algebra (see the
discussion in \cite[Tag 0656]{stacks-project} for the local case to which one
reduces) given a
commutative noetherian ring $R$, and a perfect complex $P^\bullet$ of $R$-modules, then $P^\bullet$ is
quasi-isomorphic to a projective $R$-module concentrated in dimension zero
 if and
only if the same holds (over $k$) for $P^\bullet \otimes_R k$ for every residue field $k$
of $R$.

It follows that
$\pi_0(B)$ is flat over $\pi_0(A)$ and is unramified in the sense of classical
commutative algebra: therefore, $\pi_0(A) \to \pi_0(B)$ is \'etale. Since $A
\to B$ is flat, we are done. See also \cite[Theorem 6.16]{galoischromatic}. 
\end{example} 
Our goal is to show that even periodic refinements provide a rich source of 
Galois extensions which are not \'etale. 
\begin{example}
$KO \to KU$ is a $\mathbb{Z}/2$-Galois extension, as shown in
Chapter 5 of \cite{rognes} using the following result, a proof of which
appears in \cite{htmf}.

\begin{theorem}[Wood] There is an equivalence of spectra $KO \otimes \Sigma^{-2}
\mathbb{CP}^2 \simeq KU$. 
\end{theorem}
\end{example}

Our next result is a generalization of this, which states that Galois 
coverings of an associated {stack} can be used to manufacture Galois extensions of ring spectra. For
example, the $\mathbb{Z}/2$-Galois extension $KO \to KU$ arises in this way
from the $\mathbb{Z}/2$-torsor
\( \spec \mathbb{Z} \to B \mathbb{Z}/2.  \)

\begin{theorem} \label{stackgivesgalois}
Let $G$ be a finite group acting on a Deligne-Mumford stack $X$, with $Y =
X/G$ the stack quotient. Consider a flat map $Y \to M_{FG}$.
Let $\mathfrak{Y}$ be a 0-affine even periodic refinement of $Y$ and $\mathfrak{X}$ the
induced refinement of $X$. Then 
$\Gamma(\mathfrak{X}, \otop_{\XX})$ is a faithful $G$-Galois extension of
$\Gamma(\mathfrak{Y}, \otop_{\YY}) = \Gamma(\mathfrak{X}, \otop_{\XX})^{h G}$. In
particular, the Tate spectrum of $G$ acting on 
$\Gamma(\mathfrak{X}, \otop_{\XX})$ is contractible. 
\end{theorem} 

\begin{proof} 
Choose an \'etale map $\spec R \to Y$. Then the map
\begin{equation} \label{Ggalois1} \otop_{\mathfrak{Y}}( \spec R) \to \otop_{\mathfrak{X}}( \spec R \times_Y X)
,\end{equation}
is a $G$-Galois extension: in fact, it is so even on homotopy groups, in the
sense of \Cref{etalegal}. 
In particular, the map
$$\otop_{\mathfrak{X}}( \spec R \times_Y X) \otimes_{\otop_{\mathfrak{Y}}( \spec R)  
} \otop_{\mathfrak{X}}( \spec R \times_Y X) \to 
\prod_{g \in G}\otop_{\mathfrak{X}}( \spec R \times_Y X)
,$$
is an equivalence. 
In other words, if $f\colon \XX \to \YY$ is the projection, then the map
\[ f_* \otop_{\XX} \otimes_{\otop_{\YY}} f_* \otop_{\XX} \to \prod_G f_*
\otop_{\XX},  \]
is an equivalence. 
Now, using the fact that 
$\Gamma(\mathfrak{Y}, \cdot)$ is a symmetric monoidal functor
(by 0-affineness), we find that the
map
$$\Gamma( \mathfrak{X}, \otop_{\XX}) \otimes_{\Gamma(\mathfrak{Y} , \otop_{\YY})
} 
\Gamma(\mathfrak{X}, \otop_{\XX}) \to 
\prod_{g \in G} \Gamma(\mathfrak{X}, \otop_{\XX})
,$$
is an equivalence, as desired. 

The claim about faithfulness 
follows from the following commutative square of $\infty$-categories
\[ \xymatrix{
\md( \Gamma( \mathfrak{Y}, \otop_{\YY})) \ar[d]^{\simeq} \ar[r] & \md( \Gamma( \mathfrak{X},
\otop_{\XX})) \ar[d]^{\simeq}  \\
\qcoh( \mathfrak{Y}) \ar[r]^{f^*} &  \qcoh( \mathfrak{X})
},\]
where the lower horizontal functor (pull-back) has trivial kernel, since $Y \to X$ is
faithfully flat (see \Cref{DerivedFaithfulness}). 
This shows that the Galois extension is faithful, and implies that
the Tate spectrum vanishes (\cite{rognes}, Proposition 6.3.3). 
\end{proof} 

\subsection{Tate spectra}

In this section, we give a strengthening of the earlier result on vanishing of
Tate spectra, which will apply in certain non-Galois cases as well. 

We begin by reviewing the Tate spectrum in more detail. 
Let $X$ be a spectrum with the action of a finite group $G$. Recall that there
is a \emph{norm map}
\[ X_{hG}\to X^{hG}, \]
from homotopy coinvariants to homotopy invariants, 
whose cofiber is defined to be the \emph{Tate spectrum} $X_{tG}$. 
If $X$ has a ``free $G$-action'' in that it is freely induced from an ordinary
spectrum $Y$, then the Tate spectrum is contractible. 
The Tate spectrum commutes with \emph{finite} homotopy colimits and limits in
the $\infty$-category $\mathrm{Fun}(BG, \sp)$ of spectra with a $G$-action, so
it vanishes identically on the thick subcategory of $\mathrm{Fun}(BG, \sp)$
generated by the spectra with free $G$-action. 

\begin{example} \label{tatevanishex}
Suppose $X \in \mathrm{Fun}(BG, \sp)$ has the property that the functor
\[ Y \mapsto (Y \otimes X)^{hG}, \quad \mathrm{Fun}(BG, \sp) \to \sp,  \]
commutes with homotopy colimits. 
(Equivalently, $X$ has the property that the functor $Y \mapsto (Y \otimes
X)_{tG}$ commutes with homotopy colimits.) 
Here $Y$ is a spectrum with a $G$-action, 
and $Y \otimes X$ is given the ``diagonal $G$-action:'' that is, at the level
of functors, the smash product is computed pointwise. 
Then, the Tate construction $X_{tG}$ is contractible. 

To see this, observe first that if $Y = \bigsqcup_G Z$ is free, then 
\[ Y \otimes X \simeq \bigsqcup_G Z \otimes X,  \]
so that the Tate construction $(Y \otimes X)_{tG} $
is contractible. Since we can write the sphere $S^0$ with trivial $G$-action as
a geometric realization (via the bar construction) of objects in
$\mathrm{Fun}(BG, \sp)$ with free $G$-action, it follows that $(S^0 \otimes
X)_{tG} \simeq (X)_{tG}$ is contractible too. 
\end{example}

We can now prove our main result on the vanishing of Tate spectra. 
\begin{theorem}  \label{tatevanish}
Let $X$ be a noetherian and separated Deligne-Mumford stack equipped with a flat map $X \to M_{FG}$ which is tame. 
Let $Y \to X$ be a $G$-torsor for a finite group $G$. 

Let $\mathfrak{X} = ({X}, \otop_{\mathfrak{X}})$ be an even periodic refinement, and let
$\mathfrak{Y} = ({Y}, \otop_{\mathfrak{Y}})$ be the induced  even periodic
refinement of $Y \to X \to M_{FG}$, which acquires a $G$-action. Let $q:
\mathfrak{Y} \to \mathfrak{X}$ be the induced morphism. 
Then, for any $\mathcal{F} \in \qcoh( \mathfrak{X})$, we have
\[ \left( \Gamma( \mathfrak{Y}, q^* \mathcal{F}) \right)_{tG} \simeq 0.   \]
\end{theorem} 
\begin{proof} 
By Galois descent, we obtain an equvalence of $\infty$-categories
\[ \qcoh( \mathfrak{X}) \simeq \qcoh( \mathfrak{Y})^{hG},  \]
where the $G$-action on $\mathfrak{Y}$ induces a $G$-action on the
$\infty$-category of quasi-coherent sheaves. This is true locally in view
of \'etale descent as in \cite{DAGdesc, DAGQC}, and then follows globally by sheafification. 
Moreover, for any $\mathcal{F} \in \qcoh( \mathfrak{X})$, we get $\Gamma(
\mathfrak{X}, \mathcal{F}) \simeq \Gamma( \mathfrak{Y}, q^* \mathcal{F})^{hG}$. 

As a result, given a spectrum $T$ with a $G$-action and given any
quasi-coherent sheaf $\mathcal{F} \in \qcoh( \mathfrak{X})$, we can form a
twisted pull-back $\mathcal{F} \otimes' T \in \qcoh( \mathfrak{Y})^{hG} \simeq
\qcoh( \mathfrak{X})$, which intertwines the $G$-action on $T$. 
At the level of global sections, we have
\[ \Gamma( \mathfrak{Y}, q^*( \mathcal{F} \otimes' T)) \simeq \Gamma( \mathfrak{Y}, q^*
\mathcal{F}) \otimes T \in  \mathrm{Fun}(BG, \sp),   \]
i.e., using the diagonal $G$-action on each tensor factor. 
We note that $\Gamma\colon \qcoh( \mathfrak{Y}) \to \md( \Gamma( \mathfrak{Y},
\otop))$  and 
$\Gamma\colon \qcoh( \mathfrak{X}) \to \md( \Gamma( \mathfrak{X},
\otop))$
commute with homotopy colimits by \Cref{general:colimits}. 

Now, it follows from Galois descent again
that we have natural equivalences
\[ \Gamma( \mathfrak{X} , \mathcal{F} \otimes' T) \simeq ( \Gamma(
\mathfrak{Y}, q^* \mathcal{F}) \otimes T)^{hG}, \quad T \in \mathrm{Fun}(BG,
\sp). \]
Since $\Gamma$ commutes with homotopy colimits on $\qcoh( \mathfrak{X})$, and
since the construction $\otimes'$ preserves homotopy colimits, it follows by
\Cref{tatevanishex} that the Tate construction $\left(\Gamma( \mathfrak{Y}, q^*
\mathcal{F})\right)_{tG}$ is contractible. 
\end{proof}

\section{Some examples}\label{Section6}
In this section, we discuss a few basic examples of even periodic refinements
and discuss applications of the results of this paper. The main example that
motivated us, that of topological modular forms, will be discussed in more
detail in the next section. 
\subsection{Non-examples}

\label{nonexamples}
As a non-example, consider the $\mathbb{Z}/2$-action on $KU$-theory where
$\mathbb{Z}/2$ acts trivially. The induced map $B \mathbb{Z}/2 \to M_{FG}$ is
the 
``constant'' map $B \mathbb{Z}/2 \to \spec \mathbb{Z} \to M_{FG}$, where $\spec
\mathbb{Z} \to M_{FG}$ classifies the multiplicative formal group. In
particular, it is a flat morphism. 
Since $KU$-theory is an $E_\infty$-ring, and it is possible to make
$\mathbb{Z}/2$ act trivially on $KU$, this gives a derived version of $B
\mathbb{Z}/2$, whose global sections are given by $K^{h \mathbb{Z}/2}  $. 

In this case, 
\[ K^{h \mathbb{Z}/2} \simeq \mathrm{F}( B \mathbb{Z}/2, K),  \]
whose homotopy groups are computed, by the classical {Atiyah-Segal
completion theorem} \cite{atiyahsegal}, to be the completion of the representation ring of
$\mathbb{Z}/2$ in even dimensions and zero in odd dimensions. 

The homotopy fixed point spectral sequence (equivalently, the
Atiyah-Hirzebruch spectral sequence for $K^*(\mathbb{RP}^\infty)$) has no room for differentials and
degenerates at $E^2$, with an infinite ``checkerboard''  of nonzero terms,
and thus without a horizontal vanishing line. It follows that \Cref{main} and
many of the 
results in this paper definitely fail for derived stack arising from a flat morphism $X \to M_{FG}$
which is not representable. For example, the associated pro-object is not
constant, as there are elements in $E_\infty$ of arbitrarily high filtration. 

Even if $X \to M_{FG}$ is representable, \Cref{main} may fail for more mundane
reasons. For instance, let us work over $\mathbb{Q}$, so that $M_{FG} \simeq B
\mathbb{G}_m$ and any map to $M_{FG}$ is flat: to give a formal group over a $\mathbb{Q}$-algebra is equivalent
to giving its cotangent space $\omega$, a line bundle. 
Given a scheme $X$ over $\mathbb{Q}$ and a line bundle $\omega$ on $X$, we can
produce a sheaf $\otop$ of $E_\infty$-rings on $X$ via 
\[ \otop \stackrel{\mathrm{def}}{=} \mathrm{Sym}^*( \Sigma^2 \omega )
[\Sigma^2 \omega^{-1}],  \]
where the notation means that over an open affine $U \simeq \spec R \subset X$
over which $\omega$ is trivial, $\otop(\spec R) \simeq R[x, x^{-1}]$ is the
free $E_\infty$ $R$-algebra on a generator $x$ in degree two, with $x$ inverted. The gluing 
data comes from the gluing data on $\omega$. 
In particular, the choice of $(X, \omega)$ determines a canonical  choice (not
necessarily unique) of even periodic refinement $ \mathfrak{X} = ({X}, \otop)$. 

In this case, $\otop$ is a sheaf of $\mathcal{O}_X$-algebras, so given any
coherent sheaf $\mathcal{F}_0$ on $X$, we can produce 
a quasi-coherent sheaf $\mathcal{F} = \otop \otimes_{\mathcal{O}_X} \mathcal{F}_0$ 
on $ \mathfrak{X}$. If $\mathcal{F}_0$ is such that $H^i( X,
\mathcal{F}_0 \otimes \omega^j) = 0$ for all $i, j$, then $\mathcal{F} \in
\qcoh( \mathfrak{X})$ has no global sections. To be concrete, we can take $X = \P^1_\Q$, $\omega = \mathcal{O}_X$ and $\FF_0 = \mathcal{O}[-1]$. 

\subsection{Finite group actions: $KO$-theory and $EO_n$ spectra}

Let $R$ be a Landweber-exact, $E_\infty$-ring with the action
of a finite group $G$. This induces an action of $G$ on the formal group of $R$
 compatible with the action on $\spec \pi_0 R$: as we have seen, we get a map
 \[ ( \spec \pi_0 R )/G \to M_{FG}.  \]
 This map is affine (equivalently, representable) precisely when, for every field-valued point $x\colon \spec k \to
 \spec \pi_0 R$, the stabilizer $G_x \subset G$ of $x$ acts faithfully on the
pull-back of the  formal group to $\spec k$. 
Under these hypotheses, it follows that 
\[ R^{hG} \to R  \]
is a faithful $G$-Galois extension, and Galois descent goes into effect. 

We discuss two basic examples of this. 

\begin{example}[$KO$-theory, again]
As discussed in \Cref{KO:example}, we have a map 
\[ B \mathbb{Z}/2 \to M_{FG},  \]
sending a one-dimensional torus (equivalently, $\mathbb{Z}/2$-torsor) to its
formal completion. It  is flat and affine.  The $\mathbb{Z}/2$-action on $KU$-theory
by complex conjugation enables the construction of a derived stack
$\mathfrak{B} \mathbb{Z}/2 = (B \mathbb{Z}/2, 
\otop)$, 
which is an even periodic refinement of the above map, such that $\Gamma(
\mathfrak{B} \mathbb{Z}/2, \otop) \simeq KO$. 

As a result, we recover the equivalence of $\infty$-categories
\[ \md(KO) \simeq \qcoh( \mathfrak{B}\mathbb{Z}/2) \simeq \md(KU)^{h
\mathbb{Z}/2},  \]
which we could have also seen by Galois descent. 
\end{example}

\begin{example}[$EO_n$]
Let $E_n$ be the Morava $E$-theory with coefficient ring $W(
\mathbb{F}_{p^n})[[v_1, \dots, v_{n-1}]]$. By the Goerss-Hopkins-Miller theorem
\cite{goersshopkins}, $E_n$ is an $E_\infty$-ring with an action of the \emph{extended
Morava stabilizer group} $\mathbb{G}$: that is, the semidirect product of the automorphism 
group of the Honda formal group with $\mathrm{Gal}(
\mathbb{F}_{p^n}/\mathbb{F}_p)$. 
For a discussion, see \S 5.4.1 of \cite{rognes}. 

Given a finite subgroup $H \subset \mathbb{G}$, it follows from the above discussion
that we can construct a derived stack $(\spec \pi_0E_n/H, \otop)$, and 
that $(E_n)^{hH} \to E_n$ is a faithful $H$-Galois extension. This is proved $K(n)$-locally in \cite{rognes}. Especially interesting is the case where $H$ is a \emph{maximal} finite subgroup, where $(E_n)^{hH}$ is denoted by $EO_n$ (with implicit dependence on $H$).
\end{example}

\subsection{Open subsets}
Let $R$ be a Landweber-exact, even periodic $E_\infty$-ring. Then any open
subset of $\spec \pi_0 R$ yields a derived stack, which by \Cref{open0affine} is
0-affine. 
This includes the case where we are localized at a prime $p$, so that $R$ is
$E_n$-local for some $n$. For $m < n$, the conclusion is that 
$$ \md( L_m R) \simeq \qcoh( \mathfrak{X}).$$ for $\mathfrak{X}$ an even periodic
refinement of $\spec \pi_0 R \times_{M_{FG}} M_{FG}^{\leq m}$.  

Although elementary, this construction has some uses because the associated rings of functions are definitely far
from being even periodic. 
For instance, in \cite[Theorem C]{Picard}, it is shown that the Picard groups of $L_m R$ can
be unexpectedly 
large, even when $R = E_n$ (although the algebraic Picard group is trivial). 

\subsection{The affine line}
In this subsection, we note an important example.
Let $\mathbb{Z}_{\geq 0}$ be the (discrete) topological, commutative monoid of
nonnegative integers. Since 
\[ \Sigma^\infty_+ \colon \mathcal{S} \to \sp,   \]
is a symmetric monoidal functor, it carries $E_\infty$-monoids in spaces to
$E_\infty$-ring spectra. In particular, we get an $E_\infty$-ring
$\Sigma^\infty_+ \mathbb{Z}_{\geq 0}$, which we can think of as the ``group
algebra'' on $\mathbb{Z}_{\geq 0}$. 
Given an even periodic $E_\infty$-ring $R$, the smash product
$R[\mathbb{Z}_{\geq 0}] \stackrel{\mathrm{def}}{=} R \otimes \Sigma^\infty_+
\mathbb{Z}_{\geq 0}$ is 
still even periodic, with $\pi_0 R[\mathbb{Z}_{\geq 0}] = (\pi_0 R)
\otimes_{\mathbb{Z}} \mathbb{Z}[x]$, and the map
\[ \spec (\pi_0 R)[x] \to M_{FG},  \]
associated to $R[\mathbb{Z}_{\geq 0}]$ is the one obtained from $\spec \pi_0 R
\to M_{FG}$ obtained by taking the product with the constant map $\spec
\mathbb{Z}[x] \to \spec \mathbb{Z}$. If $R$ is Landweber-exact, so is
$R[\mathbb{Z}_{\geq 0}]$. 

It follows from this that if $\mathfrak{X} = ({X}, \otop)$ is an even 
periodic refinement of a flat map $X \to M_{FG}$, we get a natural choice of
even periodic refinement $\mathbb{A}^1_{\mathfrak{X}}$ of $\mathbb{A}^1_X \to M_{FG}$ (together with a map  $\mathbb{A}^1_{\mathfrak{X}} \to \mathfrak{X}$). 
By \Cref{affineover0affine}, if $\mathfrak{X}$ is 0-affine, so is
$\mathbb{A}^1_{\mathfrak{X}}$.

\section{Applications to topological modular forms}\label{AppTMF}

In this section, we discuss the primary example that motivated this work,
and apply our results in this case. 

Let $\mellc$ be the moduli stack of stable, 1-pointed genus one curves (that
is, the Deligne-Mumford compactification of the moduli stack $\mell$ of elliptic
curves). 
A map $\spec R \to \mellc$ is equivalent to a flat family of proper curves $p\colon C
\to \spec R$ together with a section (or marked point) $e\colon \spec R \to C$ contained in the smooth
locus of $p$, such that each geometric fiber is irreducible of arithmetic genus
one with at worst nodal singularities. Then $\mellc$ is a Deligne-Mumford stack
of finite type over $\mathbb{Z}$. See \cite{DR} for more details. 

Given
such a  curve $ C \to \spec R$, one has a canonical group scheme structure on
the smooth locus $C^{\circ}$, with the
marked point as the identity, and taking the formal completion gives a morphism 
of stacks
\[ \mellc \to M_{FG},  \]
which one can check is flat using the Landweber exact functor theorem (see \cite[Chapter 4.4]{TMF}). 
In this case, one has the fundamental: 

\begin{theorem}[Goerss-Hopkins-Miller, Lurie]
The stack $\mellc$ (together with the map $\mellc \to M_{FG}$) admits an even
periodic refinement $\mellbc$. 
\end{theorem} 

A construction of $\mellbc$ is detailed in \cite{Beh}, and another is
sketched in \cite{survey}. In other words, the Goerss-Hopkins-Miller-Lurie
theorem states that given a stable 1-pointed genus one curve $C \to \spec R$,
such that the classifying map $\spec R \to \mellc$ is
\'etale,\footnote{This requires $
\spec R$ to  be regular, and that 
the Kodaira-Spencer maps at each point of the base be isomorphisms.} one can build an
$E_\infty$-ring spectrum from the associated formal group; moreover, one can do
this functorially in the elliptic curve. 

Using this derived stack, one defines the spectra of \emph{topological modular forms}:
\[ \Tmf = \Gamma(\mellbc, \otop), \quad \TMF = \Gamma(\mellb, \otop),   \]
where $\mellb \subset \mellbc$ is the open derived substack corresponding to
smooth elliptic curves. 
These will provide examples of the results in this paper. 

When $6$ is inverted, the moduli stack $\mellc$ is the weighted projective stack
$\mathbb{P}(4, 6)$, and the homotopy limits necessary to describe $\Tmf$ take
a simple form.  However, the stack $\mellc$ is quite complicated at the primes
2 and 3 (that is, there are elliptic curves with relatively large automorphism
groups), which contributes to significant torsion at those primes in $\pi_*
\Tmf$; moreover, it makes working with $\Tmf$-modules trickier, and it is not
a priori clear how well the homotopy limit  that builds $\Tmf$ behaves. The
results of this paper show that the homotopy limit behaves well. 

In fact, the idea of this paper arose, in part,  from the analysis of the homology of
connective $\tmf \stackrel{\mathrm{def}}{=} \tau_{\geq 0} \Tmf$  by the first
author in
\cite{htmf}. There, working over $\mathbb{Z}_{(2)}$ rather than $\mathbb{Z}$,
it was shown that there is a 2-local eight cell complex
$DA(1)$ such that 
the homotopy group sheaf $\pi_0$ of $\otop \otimes DA(1) \in \qcoh(\mellbc)$ is given by the
pushforward of the structure sheaf via an eight-fold cover
\[ p\colon  \mathbb{P}(1, 3) \to \mellc,  \]
where the weighted projective stack $\mathbb{P}(1, 3)$ is the quotient of a scheme by a $\mathbb{G}_m$-action, and in
particular is much simpler cohomologically than $\mellc$. Using this, it
followed that after smashing with $DA(1)$, the category of quasi-coherent
sheaves on $\mellbc$ becomes much better behaved. For example, it was
possible to conclude that
\[ \Gamma( \mellbc, \otop \otimes DA(1) \otimes T) 
\simeq \Gamma( \mellbc, \otop \otimes DA(1) )  \otimes T,
\]
for any spectrum $T$, because the spectral sequence to compute the homotopy
groups of $\Gamma( \mellbc, \mathcal{F} \otimes DA(1))$ is concentrated in the bottom two rows
(in dramatic contrast to the spectral sequence for $\Gamma(\mellbc,
\otop)$). 
In general, the global sections functor $\Gamma$ is exact, so it commutes with
\emph{finite} homotopy colimits and limits, but we cannot a priori expect it to commute with
arbitrary homotopy colimits. 

Applying the thick subcategory theorem of \cite{HS}, one may replace
$DA(1)$ with the sphere spectrum, and 
thus show that
\[ \Gamma( \mellbc, \otop \otimes T  ) 
\simeq \Gamma( \mellbc, \otop )  \otimes T,
\]
for all $T \in \sp$. 
As an application, it is possible to compute the $\Tmf$-homology of infinite
spectra such as $MU$ using the
descent spectral sequence.  
In this paper, we did not have such finite complexes available to work with,
but we used the $E_n$-spectra themselves to prove analogous results in more
generality.

We apply our results to the case of $\TMF$ (resp. $\Tmf$) and the 
derived stacks $\mellb$ (resp. $\mellbc$) that give rise to them; recall
that these are even periodic refinements of the moduli stacks of elliptic
curves (resp. possibly nodal elliptic curves). 
We will study both the $\infty$-categories of modules and the Galois theory.

\subsection{Modules over topological modular forms}
Our first main result is the following: 

\begin{theorem}\label{TMFmod}
\begin{enumerate}
\item 
The $\infty$-category of $\TMF$-modules is
equivalent (via $\Gamma$) to the $\infty$-category of quasi-coherent sheaves on $\mellb$. 
\item 
The $\infty$-category of $\Tmf$-modules is equivalent (via $\Gamma$) to
the $\infty$-category of quasi-coherent sheaves on the compactified derived stack
$\mellbc$. 
\end{enumerate}
\end{theorem}

Away from the prime 2, the first part of \thref{TMFmod} was originally proved
in \cite{meier}. The result was also known to Lurie.
\begin{proof}
Indeed, for the first claim, it suffices by \Cref{main} to show that the map 
\[ \mell \to M_{FG},  \]
is affine. 
To see this, observe that the moduli stack of elliptic curves together with a
coordinate to order four on the formal group is precisely $\spec \mathbb{Z}[a_1, a_2, a_3, a_4,
a_6][\Delta^{-1}]$: that is, a choice of coordinate to order four is precisely
the data needed to put an elliptic curve in a canonical \emph{Weierstrass
form.} See  \cite[Proposition~12.2]{rezk512}. 
The universal elliptic curve with such a coordinate is given by the equation
\[ y^2 + a_1xy + a_3 y =
x^3 + a_2x^2 + a_4x + a_6\]
and the coordinate on the formal group is given by $-x/y$. 

Since the moduli stack of formal groups with a
coordinate to order four $M_{FG}^{\leq 4}$ is affine over $M_{FG}$, it follows easily that $\mell
\to M_{FG}$ is affine. Indeed, $\spec L\times_{M_{FG}}M_{FG}^{\leq 4} \times_{M_{FG}}\mell$ is affine as the diagonal of $M_{FG}$ is affine and $\spec L\times_{M_{FG}}M_{FG}^{\leq 4}\to M_{FG}$ is an affine fpqc cover, for $L$ the Lazard ring.

The map $\mellc \to M_{FG}$ is not affine, but it is quasi-affine (and even of
cohomological dimension one). The moduli stack of generalized elliptic curves together with a coordinate to
order four is precisely $\spec \mathbb{Z}[a_1, a_2, a_3, a_4,
a_6] \setminus V(( c_4, \Delta))$, where $c_4, \Delta$ are the standard
modular forms evaluated on the cubic curve given by $y^2 + a_1xy + a_3 y =
x^3 + a_2x^2 + a_4x + a_6$.
Therefore, we can still apply \Cref{main} to the derived stack $\mellbc$ and
conclude that it is 0-affine, as desired. 
\end{proof}

\subsection{Galois theory}
Next, we study the Galois theory of $\TMF$ (resp. $\Tmf$). 

The use of level structures provides various covers of the moduli stack of
elliptic curves that rigidify the ``stackiness.'' These can be realized
topologically.

Fix a positive integer $n$. 
\begin{definition} 
Let $\mell(n)$ be the moduli stack (over $\mathbb{Z}[1/n]$) of elliptic curves
with a \emph{level $n$ structure:} that is, if $S$ is a scheme where $n$ is
invertible, then maps 
\[ S \to \mell(n),  \]
are given by (smooth) elliptic curves $p\colon C \to S, 0\colon S \to C$ together with
sections $\phi_1, \phi_2\colon S \to C$ contained in the $n$-torsion subgroup $C[n]
\subset C$, such that over each geometric fiber $C_{\overline{s}}$, for $s \in S$,
the sections $\phi_1, \phi_2$ form a basis
for the $n$-torsion $C_{\overline{s}}[n] \simeq (\mathbb{Z}/n \mathbb{Z})^2$. 
\end{definition} 

Then $\mell(n)$ is \'etale over $\mell[1/n]$, and in fact the natural forgetful
map
\[ \mell(n) \to \mell[1/n],  \]
is a $GL_2(\mathbb{Z}/n\mathbb{Z})$-torsor, where the
$GL_2(\mathbb{Z}/n\mathbb{Z})$ acts on
$\mell(n)$ by matrix multiplication on the level structure. It follows that the
composite map
\[ \mell(n) \to \mell[1/n] \to M_{FG},  \]
is flat, and $\mell(n)$ is realizable by a derived stack $\mellb(n)$ over $\mellb$. 

\begin{definition} 
The global sections $\Gamma( \mellb(n)[1/n], \otop)$ are called \emph{topological
modular forms of level $n$} and are denoted $\TMF(n)$. 
\end{definition} 

\newcommand{\mellon}{M_{ell, 1}(n)}
\newcommand{\mellocn}{M_{\overline{ell}, 1}(n)}
\newcommand{\mellocnz}{M_{\overline{ell}, 0}(n)}
\newcommand{\mello}{M_{ell, 1}}
It follows in particular that $\TMF(n)$ has a $GL_2(\mathbb{Z}/n\mathbb{Z})$-action, and
that $$\TMF[1/n] \simeq \TMF(n)^{h GL_2(\mathbb{Z}/n\mathbb{Z})}.$$
For $n \geq 3$, $\mell(n)$ is actually an affine scheme, and the resulting
spectra $\TMF(n)$ are therefore Landweber-exact, even periodic $E_\infty$-rings. 

\begin{example} 
The moduli stack (over $\mathbb{Z}[1/2]$) of elliptic curves together with a
full level $2$ structure is given by $\spec \mathbb{Z}[1/2,
\lambda][\lambda^{-1}, (\lambda - 1)^{-1}] \times B \mathbb{Z}/2$, given by
putting the elliptic curve in ``Legendre form''
\[ y^2 = x(x-1)(x - \lambda), \quad \lambda \neq 0, 1,  \]
together with the 2-torsion points $(0, 0), (0, 1)$. The $B \mathbb{Z}/2$
factor is necessary to account for the automorphism $-1$. 

Since $\mell(2)[1/2]$ is not affine (in fact, not even a scheme), the spectrum $\TMF(2)$ is not even periodic, but
only 4-periodic, with homotopy groups given by 
\[ \pi_* \TMF(2)[1/2] \simeq \mathbb{Z}[1/2, \lambda, t][\lambda^{-1},
(\lambda-1)^{-1}, t^{-1}], \quad |\lambda| = 0, |t| = 4.  \]
The $S_3 \simeq GL_2( \mathbb{Z}/2\mathbb{Z})$-Galois descent from
$\TMF(2)[1/2]$ to
$\TMF[1/2]$ is studied in detail in \cite{St12}. 

\end{example}

By taking various partial quotients of $\mell(n)$ over $\mell$, one can realize other
variants of ``moduli of elliptic curves with level structure.'' For instance,
let $\mellon$ be the moduli stack of elliptic curves with a
$\Gamma_1(n)$-structure: that is, a choice of an $n$-torsion point that
generates a $\mathbb{Z}/n\mathbb{Z}$-summand in the $n$-torsion on each fiber. 
Then $\mellon \simeq \mell(n)/H$ where $H \subset GL_2( \mathbb{Z}/n\mathbb{Z})$ consists
of matrices of the form $\begin{bmatrix}
1  & b \\ 
0 & c 
\end{bmatrix}$. 
The stack $\mellon$ is \'etale (though no longer Galois) over $\mell[1/n]$
and can consequently be realized by a derived stack, whose $E_\infty$-ring of
global sections is denoted $\TMF_1(n)$. 
Similarly, one defines $\TMF_0(n)$ from the moduli stack of elliptic curves
together with a cyclic degree $n$ subgroup.

Using the 0-affineness of $\mellb$, \Cref{affineover0affine} and \Cref{stackgivesgalois}, we find: 
\begin{theorem} 
The map $\TMF[1/n] \to \TMF(n)$ is a faithful $GL_2( \mathbb{Z}/n \mathbb{Z})$-Galois
extension. Similarly, the map $\TMF_0(n) \to \TMF_1(n)$ is a faithful
$(\mathbb{Z}/n\mathbb{Z})^{\times}$-Galois extension. In particular, the Tate
spectra of these group actions vanish. 
\end{theorem} 

The vanishing of Tate spectra in the latter case is proved for $n=5$ via different means in \cite{BO}. 

\begin{remark} 
One can show in fact that \emph{all} Galois covers of $\TMF$ and its
localizations arise from Galois extensions of the associated stack; in
particular, $\TMF$ over $\mathbb{Z}$ is ``separably closed,'' i.e., has no
nontrivial Galois extensions. This is carried out in \cite[\S 10]{galoischromatic}. 
\end{remark} 

\subsection{Tate spectra and compactified moduli}
Our earlier results showed that $\TMF[1/n] \to \TMF(n)$ is a faithful
$GL_2(\mathbb{Z}/n\mathbb{Z})$-Galois extension; in particular, the Tate
spectrum for the action of $GL_2( \mathbb{Z}/n\mathbb{Z})$ on the latter is
contractible.
In this subsection, we show the analogous Tate spectra for the non-periodic versions of
$\TMF(n)$ also vanish. The associated extensions are no longer Galois, as the
associated covers of stacks are now ramified. 
However, we will still be able to apply \Cref{tatevanish}. 

Recall that it is useful to compactify the moduli stack $\mell(n)$ by allowing the
elliptic curve to degenerate, although we will need to drop irreducibility
and allow slightly more complicated degenerations: instead of $\mathbb{P}^1$
with two points glued together (a nodal cubic), we need to allow \emph{N\'eron
$n$-gons}, which are obtained by gluing $n$ copies of $\mathbb{P}^1$, where
$0$ in the $i$th $\mathbb{P}^1$ (for $i \in \mathbb{Z}/n\mathbb{Z}$) is
attached to $\infty$ in the $(i+1)$st. This theory was developed in \cite{DR};
another helpful reference (which extends the theory to the cusps in 
characteristics dividing $n$, which we do not need) is \cite{conrad}.

\newcommand{\mnc}{\mellc^{(n)}}

\begin{definition} 
Let $\mnc$ be the moduli stack that assigns to a $\mathbb{Z}[1/n]$-scheme $S$
the groupoid of \emph{generalized elliptic curves} (\cite{DR}, Chapter II) $p\colon C \to S$, such that each 
geometric fiber of $p$ is either smooth or an $n$-gon. 
\end{definition} 

We do not review the definition of a generalized elliptic curve, except to note
that it requires more than  a curve over the base $S$ together with a section:
the group structure (on the smooth locus $C^{\circ} \subset C$) must be part of
the data, rather than a consequence of the definition. 
In Theorem 2.5, Chapter III of \cite{DR}, it is shown that $\mnc$ is a smooth
Deligne-Mumford stack of finite type over $\spec \mathbb{Z}[1/n]$. 
Moreover, there is a morphism
\[ \mnc \to \mellc[1/n],   \]
which sends a generalized elliptic curve $C \to S$ to the stable elliptic curve
$\overline{C} \to S$ obtained by fiberwise \emph{contracting} all irreducible
components not containing the identity section. (This process is discussed in
Section IV.1 of \cite{DR}.) 

In particular, $\mnc \to \mellc[1/n]$ is an equivalence of stacks away from the
``cusps.'' Near the cusps, it fails to be representable:
the automorphism group scheme of a N\'eron $n$-gon is a semidirect product
$\mathbb{Z}/2\mathbb{Z} \rtimes \mu_n$ (\S 2 of \cite{DR}, Proposition 1.10),
while the automorphism group scheme of a nodal elliptic curve (i.e., a N\'eron
$1$-gon) is simply $\mathbb{Z}/2\mathbb{Z}$. However, 
using $\mnc$, one can construct the compactification of $\mell(n)$. 

\begin{definition}[\cite{DR}, \cite{conrad}]
The stack $\mellc(n)$ classifies \emph{generalized elliptic curves} $p\colon C \to
S$ over a base $S$ with $n$ invertible, such that each  geometric fiber of $p$ is either smooth
or a N{\'e}ron $n$-gon,   together with 
an isomorphism of group schemes $\phi\colon ( \mathbb{Z}/n\mathbb{Z})^2 \simeq
C^{\circ}[n]$ (that is, a trivialization of the $n$-torsion points on the
smooth locus). 

Similarly, one defines $\mellocn$, a compactification of the moduli stack of
elliptic curves with $\Gamma_1(n)$-structure, to classify generalized elliptic curves over 
a base $S$  with an injection of group schemes $\mathbb{Z}/n \mathbb{Z} \to
C^{\circ}[n]$ such that the divisor cut out by the image of $\mathbb{Z}/n\mathbb{Z}$
is ample (i.e., intersects each irreducible component in every geometric fiber). For $n$ squarefree, one also defines $\mellocnz$, a
compactification of the moduli stack of elliptic curves with
$\Gamma_0(n)$-structure, to classify generalized elliptic curves over a base
$S$ with a (finite flat) subgroup $G \subset C^{\circ}[n]$ which intersects each irreducible component in every geometric fiber and
which is \'etale locally isomorphic to $\mathbb{Z}/n\mathbb{Z}$. 
\end{definition} 
Here, the moduli interpretation of $\mellocn$ and $\mellocnz$ can be found in \cite{conrad} or in \cite[IV.4]{DR}. Note that while Conrad only requires an fppf-local generator, in our situation we have actually an \'etale-local generator as we assume that $n$ is invertible so that $G$ is \'etale and \'etale locally free (see e.g. \cite[Theorem 34 and discussion below Theorem 33]{Stix}). 

\begin{remark}We will only use $\mellocnz$ in the case of $n$ being squarefree because the definitions in \cite{conrad} and \cite{DR} do not agree when $n$ is not squarefree (see \cite{Ces15}), but agree when $n$ is squarefree \cite[Remark 4.1.5]{conrad}. The problem is that with our definition of $\mellocnz$, the stack $\mellocnz$ is no longer representable over $\mellc$ when $n$ is not squarefree. \end{remark}

As $\mell(n)$ is to $\mell[1/n]$, the stack $\mellc(n)$ lives as a $GL_2(
\mathbb{Z}/n\mathbb{Z})$-torsor over $\mnc$. 
The moduli stack $\mellc(n)$ is a smooth Deligne-Mumford stack over
$\mathbb{Z}[1/n]$. There is a morphism of stacks
\[ \mellc(n) \to \mellc[1/n],  \]
which sends a pair $p\colon C \to S, \phi\colon (\mathbb{Z}/n\mathbb{Z})^{2} \simeq
C^{\circ}[n]$ as above (over some base $S$) to the stable elliptic curve over 
$S$ obtained from $C$
by fiberwise \emph{contracting} all non-identity irreducible components.
This map is naturally equivariant for the natural $GL_2( \mathbb{Z}/n\mathbb{Z})$ action
on the source and the trivial $GL_2(\mathbb{Z}/n\mathbb{Z})$-action on the target. 
This comes from the map $\mnc \to \mellc[1/n]$. 

In
\cite{DR}, it is shown that $\mellc(n) \to \mellc$ is finite and flat. It fails to be
\'etale over the cusps, and the existence of a topological realization is not
a direct consequence of the existence of $\Tmf$. 
However, one has: 

\begin{theorem}[Goerss-Hopkins; Hill--Lawson \cite{HillLawson}]\label{GHHL}
The moduli stack $\mellc(n)$ has an even periodic refinement $\mellbc(n)$, in
such a way that 
\[ \mellbc(n) \to \mellbc,  \]
is a $GL_2( \mathbb{Z}/n\mathbb{Z})$-equivariant morphism of derived stacks. 
Equivalently, $\mnc \to \mellc \to M_{FG}$ has an even periodic refinement. 
\end{theorem} 

In particular, it is possible to construct $E_\infty$-algebras
\[ \Tmf(n) \stackrel{\mathrm{def}}{=}  \Gamma( \mellbc(n), \otop) \]
 over $\Tmf$, which acquire $GL_2( \mathbb{Z}/n\mathbb{Z})$-actions. Using \cite{HillLawson}, one can similarly define even periodic refinements  of $\mellocn, \mellocnz$, and 
obtains $E_\infty$-rings $\Tmf_0(n), \Tmf_1(n)$, where $\Tmf_1(n)$ has a
$(\mathbb{Z}/n\mathbb{Z})^{\times}$-action with homotopy fixed points given
by $\Tmf_0(n)$. 

Our first main result in this section is: 
\begin{theorem}\label{TmfTateVanishing} 
The Tate spectrum of $GL_2( \mathbb{Z}/n\mathbb{Z})$ on $\Tmf(n)$ is
contractible.  
\end{theorem} 

For $n = 2$, this result appears in \cite{St12}. 
\begin{proof} 
This is a consequence of \Cref{tatevanish}, once we show that the map
\[ \mnc \to M_{FG},   \]
is quasi-compact, separated, and tame. 

To see this, it suffices  to check at the level of stabilizers. 
Away from the cusps, there is no issue: the stabilizers of $\mnc$
(equivalently, of $\mell$) inject into those of $M_{FG}$. At the cusps, we
recall  that the automorphism group of a N\'eron $n$-gon is a
semidirect product $(\mathbb{Z}/2) \rtimes \mu_n$, where the $\mathbb{Z}/2$
piece (which acts by inversion) injects into the associated stabilizer for
$M_{FG}$. Since we have inverted $n$, it follows that the kernels of the maps
of stabilizers are invertible and the map is tame. Thus by
\Cref{tatevanish}, we are done. 
\end{proof} 

For $\Tmf_1(n)$, the situation is even better. 
\begin{theorem}
 For $n$ squarefree, the map $\Tmf_0(n) \to \Tmf_1(n)$ is a faithful $(\Z/n\Z)^\times$-Galois extension. 
\end{theorem}
\begin{proof}
 By \cite[4.1.1]{conrad}, $\mellocnz$ is finite and representable over $\mellc$. By \Cref{affineover0affine}, \Cref{stackgivesgalois} and the 0-affineness of $\mellc$, it is enough to show that $\mellocn \to \mellocnz$ is a $(\Z/n\Z)^\times$-Galois cover. This means that we need to show the existence of an \'etale cover of $\mellocnz$ by maps $f\colon U\to \mellocnz$ such that we have a $(\Z/n\Z)^\times$-equivariant isomorphism 
 $$\mellocn\times_{\mellocnz}  U\cong (\Z/n\Z)^\times \times  U$$
 over $U$. By possibly refining a given \'etale cover, we can assume that $f$ classifies a generalized elliptic curve whose given subgroup is actually isomorphic to $\Z/n\Z$; in this case, the statement is clear. 
\end{proof}

\appendix
\section{Homotopy Limits and Sheaves}\label{AppHLS}
Let $U\cup V = X$ be an open covering of a topological space. Let $\FF$ be a sheaf on $X$ with values in an $\infty$-category. Is then the square
\[\xymatrix{ \FF(X) \ar[d]\ar[r]& \FF(U) \ar[d] \\
\FF(V) \ar[r]& \FF(U\cap V) } \]
(homotopy) cartesian? A priori, we can compute $\FF(X)$ only as the homotopy
limit over the (infinite) Cech cosimplicial object associated to the covering.
Nevertheless, in this appendix, we will give a positive answer to the question
for arbitrary finite covers. This is used in Proposition \ref{ZarColim}, where
we need to compute $\FF(X)$ as a finite homotopy limit. Our strategy is first
to compare an ordered and an unordered version of the Cech cosimplicial object for a sheaf on a space. 
At least in outline, this material is surely known to the experts. \\

Let $X$ be a topological space and let $(U_i)_{i\in I}$ be open subsets
covering $X$. Let $I$ be finite of cardinality $n$ and totally ordered. For a
tuple $\underline{i} = (i_1,\dots, i_k) \in I^k$, denote by $U_{\underline{i}}$
the intersection of $U_{i_1},\dots, U_{i_k}$. Let $I^k_{\leq}$ be the set of
weakly increasing $k$-tuples (i.e., $i_1\leq i_2 \leq \cdots \leq i_k$). 

To this data, we can associate (at least) two simplicial objects:
\[\mathfrak{C}^U_\bullet\colon k \mapsto \coprod_{\underline{i}\in I^{k+1}} U_{\underline{i}} \cong (\coprod_{i\in I} U_i)^{\times_X k+1}\] 
\[\mathfrak{C}^{U,\leq}_\bullet \colon k \mapsto \coprod_{\underline{i}\in I^{k+1}_{\leq}} U_{\underline{i}}\] 
The face maps are given by leaving out elements and the degeneracies by repeating elements. There is an obvious simplicial map $e \colon\mathfrak{C}^{U,\leq}_\bullet \to \mathfrak{C}^U_\bullet$. 

Note that we are given a presheaf $\FF$ on $X$, we can evaluate $\FF$ on a
disjoint union $\coprod_{\underline{i}\in I^k}U_{\underline{i}}$ by setting
$\FF(\coprod_{\underline{i}\in I^k}U_{\underline{i}}) = \prod_{\underline{i}\in
I^k} \FF(U_{\underline{i}})$. We view the disjoint unions here as formal and never identify the disjoint union of more than one open subsets of $X$ with an open subset of $X$.

The following proof is inspired by Proposition 2.7 of \cite{D-I04}. 

\begin{proposition}\label{TechnicalTot}Let $\mathrm{Top}$ be the category of topological spaces. Let $\FF$ be a $\mathrm{Top}$-valued presheaf on $X$. Then the canonical map 
 \[e^* \colon\Tot \FF(\mathfrak{C}^U_\bullet) \to  \Tot \FF(\mathfrak{C}^{U,\leq}_\bullet)\]
is a homotopy equivalence. \end{proposition}

\begin{proof}
For each multi-index $\underline{i} = (i_0,\dots, i_m)$, there is a unique permutation $\sigma_{\underline{i}} \in S_{m+1}$ with 
\begin{enumerate}
\item $i_{\sigma_{\underline{i}}(0)} \leq i_{\sigma_{\underline{i}}(1)} \leq \cdots \leq i_{\sigma_{\underline{i}}(m)}$, and
\item $\sigma_{\underline{i}}^{-1}(k)<\sigma_{\underline{i}}^{-1}(l)$ if $i_k = i_l$ for some $k<l$.
\end{enumerate}
This will allow us to define an inverse map $g\colon \Tot \FF(\mathfrak{C}^{U,\leq}_\bullet) \to \Tot \FF(\mathfrak{C}^U_\bullet)$ to $e^*$:

Recall that $\Tot \FF(\mathfrak{C}^U_\bullet) \subset \prod_{[m]\in \Delta} \FF(\mathfrak{C}^U_m)^{\Delta^m}$ consists of all maps of cosimplicial topological spaces $\Delta^\bullet \to \FF(\mathfrak{C}^U_\bullet)$. Thus, to give a map into $\Tot \FF(\mathfrak{C}^U_\bullet)$ is equivalent to giving maps into all $\FF(U_{\underline{i}})^{\Delta^m}$, $\underline{i}: [m] \to I$ a multi-index, compatible with coface and codegeneracy maps. Given a multi-index $\underline{i}: [m] \to I$, we define the map $g_{\underline{i}}: \Tot \FF(\mathfrak{C}^{U,\leq}_\bullet) \to \FF(U_{\underline{i}})^{\Delta^m}$ as the composition 
\[\Tot \FF(\mathfrak{C}^{U,\leq}_\bullet) \xrightarrow{\pr_{\underline{i}\sigma_{\underline{i}}}} \FF(U_{\underline{i}\sigma_{\underline{i}}})^{\Delta^m} \xrightarrow{\sigma_{\underline{i}}^*} \FF(U_{\underline{i}\sigma_{\underline{i}}})^{\Delta^m} \xrightarrow{=} \\
\FF(U_{\underline{i}})^{\Delta^m}.
\]
Here, $\sigma_{\underline{i}}$ sends a point $(t_0,\dots, t_m) \in \Delta^m$ to $(t_{\sigma_{\underline{i}}(0)}, t_{\sigma_{\underline{i}}(1)}, \dots, t_{\sigma_{\underline{i}}(m)})$. We will only check compatibility with coface maps. Let $f\in\Tot \FF(\mathfrak{C}^{U,\leq}_\bullet)$. We want to show that the diagram
\[\xymatrix{\Delta^{m-1} \ar[d]^{d^j} \ar[r]^{g_{\underline{i}d^j}(f)} & \FF(U_{\underline{i}d^j}) \ar[d]^{d^j_{\underline{i}}} \\
\Delta^m \ar[r]^{g_{\underline{i}}(f)} & \FF(U_{\underline{i}}) }\]
commutes for $d^j: [m-1]\to [m]$. Here, $d^j_{\underline{i}}$ denotes $\FF(\mathfrak{C}^U_{m-1})\xrightarrow{d^j} \FF(\mathfrak{C}^U_m) \xrightarrow{\pr_{\underline{i}}} \FF(U_{\underline{i}})$, factoring through $\FF(U_{\underline{i}d^j})$.

By definition, this is the outer part of the rectangle
\[\xymatrix{\Delta^{m-1} \ar[d]^{d^j} \ar[r]^{\tau} & \Delta^{m-1}\ar[rr]^-{f_{\underline{i}d^j\tau}} \ar[d]^{d^{\sigma^{-1}(j)}} && \FF(U_{\underline{i}d^j\tau}) = \FF(U_{\underline{i}\sigma d^{\sigma^{-1}(j)}})\ar[d]^{d^{\sigma^{-1}(j)}_{\underline{i}\sigma}} \ar[r]^-=&\FF(U_{\underline{i}d^j}) \ar[d]^{d^j_{\underline{i}}} \\
\Delta^m \ar[r]^{\sigma} & \Delta^m \ar[rr]^{f_{\underline{i}\sigma}} &&\FF(U_{\underline{i}\sigma}) \ar[r]^=& \FF(U_{\underline{i}}) }\]

Here, $\sigma = \sigma_{\underline{i}}$ and $\tau = \sigma_{\underline{i}d^j}$ for short. We claim that all the small squares commute (and make sense). 

One can check that $d^{\sigma^{-1}(j)}\tau^{-1} = \sigma^{-1}d^j$. This gives the commutativity of the first square (note how the permutations become inverted). The equality in the upper right corner of the next square follows from $d^j\tau = \sigma d^{\sigma^{-1}(j)}$. The commutativity of this square follows since $f\in \Tot \FF(\mathfrak{C}^{U,\leq}_\bullet)$. In the last square, the vertical morphisms are induced by inclusions between the same open subsets and thus have to be equal. The proof for the codegeneracies is similar. Thus, we get a well-defined map
\[g: \Tot \FF(\mathfrak{C}^{U, \leq}_\bullet) \to \Tot \FF(\mathfrak{C}^{U}_\bullet) \]

The composition $e^*g$ equals the identity. We want to show that $ge^*$ is homotopic to the identity. For a permutation $\sigma \in S_{m+1}$ and an element $s\in [0,1]$, define a map $u_{\sigma,s}: \Delta^m \to \Delta^{2m+1}$ by \[(t_0,\dots, t_m) \mapsto (st_0,\dots, st_m, (1-s)t_{\sigma(0)}, \dots, (1-s)t_{\sigma(m)}).\]
Then for $\sigma = \sigma_{\underline{i}}$, we define maps
\[\Tot \FF(\mathfrak{C}^U_\bullet) \times [0,1] \to \FF(U_{i_0i_1\dots i_m})^{\Delta^m} \]
as the composition
\[
\xymatrix{
\Tot \FF(\mathfrak{C}^U_\bullet) \times [0,1] \ar[d]^{\pr_{i_0\dots i_m i_{\sigma(0)}\dots i_{\sigma(m)}} \times u_{(\sigma,\bullet)}}\\
 \FF(U_{i_0\dots i_m i_{\sigma(0)}\dots i_{\sigma(m)}})^{\Delta^{2m+1}} \times \Map(\Delta^m, \Delta^{2m+1}) \ar[d]\\
\FF(U_{i_0\dots i_m i_{\sigma(0)}\dots i_{\sigma(m)}})^{\Delta^m} \ar[d]^= \\
 \FF(U_{i_0i_1\dots i_m})^{\Delta^m}
}
\]

This defines the required homotopy $H: \Tot \FF(\mathfrak{C}^U_\bullet) \times [0,1] \to \Tot \FF(\mathfrak{C}^U_\bullet)$ between $\id$ and $ge^*$, once we have checked the compatibility with coface and codegeneracies. We will only treat the coface maps again. We set $\tau = \sigma_{\underline{i}d^j}$ again and choose $(f,s)\in\Tot \FF(\mathfrak{C}^U_\bullet) \times [0,1]$. Furthermore, for functions $a,b: [m] \to I$ we use the notation $a|b: [2m+1] = [m]\sqcup [m] \to I$ for the sum of $a,b$. For example, $\underline{i}|\underline{i}\sigma = (i_0,\dots, i_m, i_{\sigma(0)}, \dots, i_{\sigma(m)})$. The compatibility follows from the commutative diagram
\[\xymatrix{\Delta^{m-1} \ar[d]^{d^j} \ar[r]^{u_{\tau, s}} & \Delta^{2m-1} \ar[d]^{d^j\sqcup d^{\sigma^{-1}(j)}}\ar[rr]^{f_{\underline{i}d^j|\underline{i}d^j\tau}} && \FF(U_{\underline{i}d^j|\underline{i}d^j\tau})\ar[d]^{d^j_{\underline{i}} \sqcup d^{\sigma^{-1}(j)}_{\underline{i}\sigma}} \ar[r]^=& \FF(U_{\underline{i}d^j})\ar[d]^{d^j_{\underline{i}}} \\
\Delta^m\ar[r]^{u_{\sigma, s}} & \Delta^{2m+1} \ar[rr]^{f_{\underline{i}|\underline{i}\sigma}} && \FF(U_{\underline{i}|\underline{i}\sigma}) \ar[r]^= & \FF(U_{\underline{i}}) 
}\]
The commutativity is shown as above. 
\end{proof}

Note that we did this proof in a topological and not in a simplicial setting since the symmetric group $S_{m+1}$ acts on the topological $m$-simplex, but not on the simplicial $m$-simplex. 

As a preparation for the following proof, we note that $\mathfrak{C}^U_\bullet$ and $\mathfrak{C}^{U,\leq}_\bullet$ have free degeneracies in the sense of the following definition:
\begin{definition}Let $\CC$ be a category with coproducts. A $\CC$-valued simplicial object $X_\bullet$ is said to have \textit{free degeneracies} if there exist maps $N_k\to X_k$ from $N_k\in\CC$ such that the canonical map
\[\coprod_{\sigma: [k] \twoheadrightarrow [m]} N_m \to X_k\]
is an isomorphism for every $k$.\end{definition}
Equivalently, the restriction of $X_\bullet$ to $(\Delta_{\mathrm{epi}})^{op}$ is isomorphic to the left Kan extension of $N\colon \mathbb{N}_0 \to \CC$ along $\mathbb{N}_0\to (\Delta_{\mathrm{epi}})^{op}$. Here, $\Delta_{\mathrm{epi}}$ is the subcategory of $\Delta$ consisting of order-preserving epimorphism. 

Both $\mathfrak{C}^{U,\leq}_\bullet$ and $\mathfrak{C}^U_\bullet$ have free degeneracies: In the case of $\mathfrak{C}^{U,\leq}_\bullet$, we choose $N_k = \coprod_{i_0<i_1<\cdots< i_k} U_{i_0i_1\dots i_k}$. In the case of $\mathfrak{C}^U_\bullet$, we choose $N_k = \coprod_{i_0\neq i_1\neq i_2 \neq \cdots \neq i_k} U_{i_0i_1\dots i_k}$. Here, we really do not mean pairwise inequality, but just $i_l \neq i_{l+1}$. This can be refined to the following statement, which we will use for Corollary \ref{OrderSheafCor}: 

\begin{lemma}\label{LKan}Define a functor $\mathfrak{C}^{U,<}_\bullet: \Delta_{\mathrm{mono}, \leq n-1}^{op} \to \mathrm{Top}$ by $\mathfrak{C}^{U,<}_k = \coprod_{i_0<i_1<\cdots< i_k} U_{i_0i_1\dots i_k}$. Then the canonical map $\mathrm{LKan}_F \mathfrak{C}^{U,<}_\bullet \to \mathfrak{C}^{U,\leq}_\bullet$ along the functor $F: \Delta_{\mathrm{mono}, \leq n-1}^{op} \to \Delta^{op}$ is an isomorphism.\end{lemma}
Here, $\Delta_{\mathrm{mono}, \leq n-1}$ denotes the subcategory of $\Delta$ consisting of order-preserving monomorphisms between $[0],\dots, [n-1]$. Note furthermore that a similar statement with $\Delta_{\mathrm{mono}}$ instead of $\Delta_{\mathrm{mono}, \leq n-1}$ is also true for for $\mathfrak{C}^U_\bullet$.
\begin{proof}By definition 
\[(\mathrm{LKan}_F \mathfrak{C}^{U,<}_\bullet)([k]) = \colim_{[k]\to [l], l\leq n-1} \mathfrak{C}^{U,<}_l\]
where a morphism in the index category betweeen $f: [k]\to [l]$ and $g: [k]\to [l']$ consists of an injection $i: [l'] \to [l]$ such that $f = i\circ g$. As every morphism in $\Delta$ factors uniquely as an epimorphism followed by a monomorphism, the full subcategory on all $[k]\twoheadrightarrow [l], l\leq n-1$ is final. As this category is discrete, the result follows.
\end{proof}

Let $X_\bullet$ be again a simplicial object in a category $\CC$ (with coproducts) and $\FF$ be a product-preserving functor $\CC^{op} \to \mathrm{Top}$. Assume that $X_\bullet$ has free degeneracies with maps $N_k\to X_k$ as above.  Then the $m$-th matching object of $\FF(X_\bullet)$ is isomorphic to 
\[\lim_{[m]\twoheadrightarrow [k], k<m} \prod_{[k]\twoheadrightarrow [l]}\FF(N_l) \cong \prod_{[m]\twoheadrightarrow [l], l<m}\FF(N_l).\]
Thus, the matching map \[\FF(X_m) \cong \prod_{[m]\twoheadrightarrow [l]}\FF(N_l) \to \prod_{[m]\twoheadrightarrow [l], l<m}\FF(N_l)\] is a projection and thus a fibration. 

\begin{corollary}Let $\FF$ be a $\mathrm{Top}$-valued presheaf on $X$. Then the canonical map 
 \[e^*: \holim_\Delta \FF(\mathfrak{C}^U_\bullet) \to  \holim_\Delta \FF(\mathfrak{C}^{U,\leq}_\bullet)\]
is a homotopy equivalence. \end{corollary}
\begin{proof}
The relevant cosimplicial objects are Reedy fibrant by the discussion preceeding this corollary. Thus, the statement follows from the Proposition \ref{TechnicalTot}.
\end{proof}

\begin{corollary}\label{OrderSheafCor}Let $\FF$ be a sheaf on $X$ with values in a complete $\infty$-category $\CC$. Then the maps
\[\FF(X) \xrightarrow{\simeq} \holim_{\Delta} \FF(\mathfrak{C}^U_\bullet) \to  \holim_{\Delta} \FF(\mathfrak{C}^{U,\leq}_\bullet),\]
induced by the inclusion $\mathfrak{C}^{U,\leq}_\bullet \to \mathfrak{C}^{U}_\bullet$, and 
\[ \holim_{\Delta} \FF(\mathfrak{C}^{U,\leq}_\bullet) \to \holim_{\Delta_{\mathrm{mono}, \leq n-1}} \FF(\mathfrak{C}^{U,<}_\bullet),\]
induced by the inclusions $\Delta_{\mathrm{mono}, \leq n-1}\to \Delta$ and $\mathfrak{C}^{U,<}_\bullet \to \mathfrak{C}^{U, \leq}_\bullet$, 
 are equivalences. \end{corollary}
\begin{proof}By the last corollary, the first map is an equivalence if $\CC$ is the $\infty$-category $\mathcal{S}$ of spaces. Indeed, $\mathcal{S}$ is equivalent to the coherent nerve of the simplicial category of Kan simplicial sets. Given an $\mathcal{S}$-valued sheaf $\FF$, this can be thus strictified to a presheaf of simplicial sets on $X$ by \cite[Theorem 4.2.4.4]{HTT}. Its geometric realization is a presheaf of topological spaces. As geometric realization commutes with homotopy limits and homotopy limits in $\mathcal{S}$ can be computed as homotopy limits in simplicial sets by \cite[Theorem 4.2.4.1]{HTT}, we can apply the last corollary. 

Thus, the first part of this corollary follows also for the $\infty$-category of presheaves $\mathcal{P}(\CC) = \mathrm{Fun}(\CC, \mathcal{S})$ for an arbitrary $\infty$-category $\CC$. As the canonical map $\CC \to \mathcal{P}(\CC)$ is a (homotopy) limit preserving embedding (\cite[Propositions 5.1.3.1 and 5.1.3.2]{HTT}), the corollary follows for an arbitrary complete $\infty$-category $\CC$.

The second part follows as we can see from (the proof of) Lemma \ref{LKan} that $\FF(\mathfrak{C}^{U,\leq}_\bullet)$ is (homotopy) right Kan extended from $\FF(\mathfrak{C}^{U,<}_\bullet)$, first for the $\infty$-category of spaces and then for all complete $\infty$-categories as above, so the homotopy limits agree. 
\end{proof}

For example, if $X= U\cup V$, this implies that $\FF(X)$ is the homotopy equalizer of 
\[\xymatrix{\FF(U)\times \FF(V) \ar@<0.5ex>[r] \ar@<-0.5ex>[r] &\FF(U\cap V).}\]
This formulation is all we need for this article, but to answer the question posed at the beginning of this appendix, we introduce one further reformulation of this homotopy limit. 

Let $\mathcal{P}_I$ be the poset of non-empty subsets of $I$. Each subset of $I$ with $k$ elements has a unique order-preserving bijection to $[k-1]$. This defines a functor $G: \mathcal{P}_I \to \Delta_{\mathrm{mono}, \leq n-1}$, where $n$ is still the cardinality of $I$. Furthermore, there is a functor 
\[\mathfrak{C}^{U,c}: \mathcal{P}_I^{op} \to \mathrm{Top},\qquad S \mapsto \bigcap_{i\in S}U_i.\]

\begin{proposition}\label{CubeLimit}The map $\mathrm{LKan}_{G^{op}} \mathfrak{C}^{U,c} \to \mathfrak{C}^{U,<}$ is an isomorphism. Thus, for $\FF$ a sheaf with values in a complete $\infty$-category, the map
\[\FF(X) \to \holim_{\mathcal{P}_I} \FF(\mathfrak{C}^{U,c})\]
is an equivalence.
\end{proposition}
\begin{proof}
By definition, 
\[(\mathrm{LKan}_{G^{op}}\mathfrak{C}^{U,c})([k]) = \colim_{\varnothing\neq S \subset I, [k]\hookrightarrow [|S|-1]} \bigcap_{i\in S} U_i.\]
The index category has the discrete subcategory of all subsets of $I$ with $k+1$ elements as a final subcategory (note the op's). This proves the first part of the proposition. The second part follows from Corollary \ref{OrderSheafCor} and the fact that $\FF$ sends this left Kan extension to a right Kan extension. 
\end{proof}

In particular, for $X = U\cup V$, this implies that the square
\[\xymatrix{ \FF(X) \ar[d]\ar[r]& \FF(U) \ar[d] \\
\FF(V) \ar[r]& \FF(U\cap V) } \]
is (homotopy) cartesian. People familiar with Goodwillie calculus will notice that this special case actually implies the last proposition for arbitrary finite covers. 

\begin{remark}\label{AlgStackZar}Note that we can apply the whole discussion also to a Zariski covering $\{U_i \to X\}$ of an algebraic stack $X$ by using the underlying space of $X$ (see Corollary \ref{Zariski-Equivalence}).\end{remark}

\bibliographystyle{alpha}
\bibliography{affine}

\end{document}